\documentclass[12pt]{article}
\DeclareUnicodeCharacter{FFFD}{}
\usepackage{amsmath, amssymb, amsthm}
\usepackage{ytableau}
\usepackage{hyperref}
\hypersetup{colorlinks, allcolors=blue}
\usepackage{graphicx}
\usepackage{enumitem}
\usepackage{appendix}
\usepackage{tikz}
\usepackage{makecell}
\usepackage{cleveref}
\usepackage[all]{xy}
\usepackage[numbers,sort&compress]{natbib} % biblatex → natbib に変更
\usepackage[a4paper, top=1in, bottom=1in, left=1.25in, right=1.25in]{geometry}

\setlist[enumerate,1]{leftmargin=*} % 左余白を本文と揃える

\setlist[itemize]{leftmargin=1em} 
\theoremstyle{plain}
\newtheorem{theorem}{Theorem}[section] % セクションごとにリセット
\newtheorem{proposition}[theorem]{Proposition}
\newtheorem{corollary}[theorem]{Corollary}
\newtheorem{lemma}[theorem]{Lemma}
\theoremstyle{definition}
\newtheorem{remark}[theorem]{Remark}
\newtheorem{definition}[theorem]{Definition}
\newtheorem{example}[theorem]{Example} % 定理と共有カウンタ

\usepackage{cleveref}

\crefname{theorem}{Theorem}{Theorems}       % 定理
\Crefname{theorem}{Theorem}{Theorems}       % 文頭用

\crefname{lemma}{Lemma}{Lemmas}             % 補題
\Crefname{lemma}{Lemma}{Lemmas}

\crefname{proposition}{Proposition}{Propositions} % 命題
\Crefname{proposition}{Proposition}{Propositions}

\crefname{corollary}{Corollary}{Corollaries} % 系
\Crefname{corollary}{Corollary}{Corollaries}

\crefname{example}{Example}{Examples}       % 例
\Crefname{example}{Example}{Examples}

\crefname{remark}{Remark}{Remarks}          % リマーク
\Crefname{remark}{Remark}{Remarks}

\crefname{definition}{Definition}{Definitions} % 定義
\Crefname{definition}{Definition}{Definitions}

\crefname{equation}{Equation}{Equations}    % 式
\Crefname{equation}{Equation}{Equations}

\crefname{figure}{Figure}{Figures}          % 図
\Crefname{figure}{Figure}{Figures}

\crefname{table}{Table}{Tables}             % 表
\Crefname{table}{Table}{Tables}

\Crefname{question}{Question}{Questions} % 単数形と複数形の指定
\crefname{question}{Question}{Questions} % 単数形と複数形の指定

\crefname{conjecture}{Conjecture}{Conjectures}
\Crefname{conjecture}{Conjecture}{Conjectures}
\renewenvironment{proof}{\noindent{\bf Proof.}}{\hfill$\square$\par}

\title{\textbf{Odd Verma's Theorem
}}
\author{\textbf{Shunsuke Hirota}}

\date{\textit{\today}}

\begin{document}

\maketitle
\begin{abstract}
We formulate several basic properties of Verma supermodules over regular symmetrizable Kac--Moody Lie superalgebras, exhibiting $\mathfrak{gl}(1|1)$-nature as revealed through changing Borel subalgebras.

We investigate variants of Verma modules obtained by changing Borel subalgebras, which enable us to realize the principal block of $\mathfrak{gl}(1|1)$ as an extension-closed abelian subcategory of category $\mathcal{O}$. This phenomenon is precisely formulated in terms of semibricks.

On the other hand, by applying the exchange property of odd reflections, we describe compositions of homomorphisms between Verma modules associated with different Borel subalgebras that share the same character. As an application, we refine existing results on the associated varieties and projective dimensions of Verma modules.
\end{abstract}
\tableofcontents

\section{Introduction}

Representation theory of Kac-Moody Lie superalgebras \cite{serganova2011kac} arises as a natural generalization of that of Kac-Moody Lie algebras. While several aspects of the representation theory of Kac-Moody Lie algebras were of $\mathfrak{sl}_{2}$-nature, its odd isotropic root analogue in the case of Kac-Moody Lie superalgebras, i.e.\ the $\mathfrak{gl}(1|1)$-nature, admits several claims that are intuitively evident, yet there has remained room for their precise formulation. The purpose of this work is to provide such precise formulations for several of these fundamental properties.

Whenever an atypical Verma module appears, it is well known that one obtains two long exact sequences of Verma modules attached, with respect to an odd reflection, to adjacent Borel subalgebras in opposite directions, together with homomorphisms between these sequences. Our first result states that, by treating certain non-simple modules as if they were simple and thereby viewing their structure in a coarse manner, this situation reduces to the representation theory of $\mathfrak{gl}(1|1)$. Although this may be regarded as obvious to specialists, its precise formulation requires the notion of semibricks, and we believe that such a formulation has several benefits.
 
 A \emph{semibrick} is, in simple terms, a class of objects in an abelian category that satisfy Schur's lemma (i.e., a collection of modules \( \{ M_i \}_{i \in I} \) such that \( \dim \operatorname{Hom}(M_i, M_j) = \delta_{ij} \)). By a classical result of Ringel~\cite{Ringel1976RepresentationsOK} (presented here as \cref{semibrick}), it is known that the \emph{filtration closure} of a semibrick forms an extension-closed abelian full subcategory. In the case of module categories of finite-dimensional algebras, semibricks have been shown to correspond bijectively to key structures and have been extensively studied in recent years \cite{asai2020semibricks}.

\begin{theorem}[\( \#J=1 \) case of \cref{main_result2}]\label{KhovanovThm}
Let us consider a regular Kac-Moody Lie superalgebra $\mathfrak{g}$.  
Let $\mathfrak{b}$ be a Borel subalgebra of $\mathfrak{g}$,  
$\alpha$ an isotropic $\mathfrak{b}$-simple root,  
and $\lambda$ a weight orthogonal to $\alpha$.  
We define $r_{\alpha}\mathfrak{b}$ to be the Borel subalgebra adjacent to $\mathfrak{b}$ with respect to $\alpha$.
Let $M^{\mathfrak{b}}(\lambda)$ denote the $\mathfrak{b}$-Verma module with $\mathfrak{b}$-highest weight $\lambda$.

Then the collection of images of  nonzero homomorphisms between Verma modules \( M^{\mathfrak{b}}(\lambda + n\alpha) \to M^{r_\alpha\mathfrak{b}}(\lambda + (n-1)\alpha) \), for all integers \( n \), forms a semibrick \( V^{\alpha}_{\lambda} \).
Let \( \operatorname{Filt} V^{\alpha}_{\lambda} \) denote the filtration closure of \( V^{\alpha}_{\lambda} \) in the category \( \mathcal{O} \). 

Then \( \operatorname{Filt} V^{\alpha}_{\lambda} \) does not depend on the choice of \( \mathfrak{g} \), \( \mathfrak{b} \), \( \alpha \), or \( \lambda \) as long as they satisfy the above conditions.

In particular, \( \operatorname{Filt} V^{\alpha}_{\lambda} \) is equivalent, as a category, to the principal block \( \mathcal{O}_{0}(\mathfrak{gl}(1|1)) \) of the category \( \mathcal{O} \) for \( \mathfrak{gl}(1|1) \).
\end{theorem}

 As a consequence, questions regarding \( \operatorname{Ext_{\mathcal{O}}^1}\) groups between Verma modules or their variants in this subcategory reduce to computations in \( \mathfrak{gl}(1|1)\). In this case, the indecomposable projective modules in the subcategory are given by certain variants of Verma modules,
which can be regarded as induced from the intersection of adjacent Borel subalgebras. (Such variant modules have also been studied, for example, in~\cite{cheng2015brundan} and~\cite{serganova2011kac}). This result also illustrates how the indecomposable zigzag modules for \( \mathfrak{gl}(1|1) \) admit natural analogues in the general category \( \mathcal{O} \).

The existence of essentially distinct Borel subalgebras and the odd reflections are intriguing phenomena absent in the classical theory of Kac-Moody Lie algebras \cite{serganova2011kac}. As suggested, for instance, by the classification of Nichols algebras of diagonal type \cite{heckenberger2009classification,andruskiewitsch2017finite}, these structures should be regarded as carrying essential information. While the representation theory attached to each pair of adjacent Borel subalgebras is, as the previous result shows, of $\mathfrak{gl}(1|1)$-nature, the global picture exhibits phenomena that cannot be explained solely in terms of $\mathfrak{gl}(1|1)$.

In the representation theory of classical semisimple Lie algebras, Verma modules in principal block of category $\mathcal{O}$ are indexed by elements of the Weyl group. The classical Verma’s theorem \cite{verma1968,bgg1976} can be understood as providing a complete description of the homomorphisms between Verma modules that are related, in an appropriate sense, by the Weyl group, that is, by even reflections. These homomorphisms are entirely governed by the combinatorics of the Weyl group. 

As a version of Verma's theorem for odd reflections, perhaps the most natural formulation is to give a complete description of the homomorphisms between Verma modules associated with different Borel subalgebras that share the same character—that is, to determine precisely when compositions of nonzero homomorphisms are themselves nonzero. However, unlike in the classical case, even when combining the Even Verma's theorem \cite{musson2012lie} and the Odd Verma's theorem \cref{1108}, we are still far from a complete description of \emph{all} homomorphisms between \emph{all} Verma modules; see, for example,~\cite{liu2019odd,sale2019singular}. 

For simplicity, we restrict our attention to basic Lie superalgebras, i.e., finite-dimensional Kac-Moody Lie superalgebras. However, the same arguments extend to general regular symmetrizable Kac-Moody Lie superalgebras, provided one takes appropriate care in the definition of the Weyl vector.

\begin{lemma}[Lemma 6.1 in \cite{cheng2015brundan}]
    For any basic Lie superalgebra \( \mathfrak{g} \), weight \( \lambda \), and its Borel subalgebras \( \mathfrak{b}, \mathfrak{b}' \) with the same even part, the following holds, where \( \rho^{\mathfrak{b}} \) denotes the Weyl vector associated with \( \mathfrak{b} \):
    \begin{itemize}
        \item \( \operatorname{ch} M^{\mathfrak{b}}( \lambda - \rho^{\mathfrak{b}}) = \operatorname{ch} M^{\mathfrak{b}'}(\lambda - \rho^{\mathfrak{b}'}) \);
        \item \( \dim \operatorname{Hom}(M^{\mathfrak{b}}(\lambda - \rho^{\mathfrak{b}}), M^{\mathfrak{b}'}(\lambda - \rho^{\mathfrak{b}'}) ) = 1 \).        
    \end{itemize}
\end{lemma}

In other words, for a given $\mathfrak{b}$ and a Verma module $M^{\mathfrak{b}}(\lambda)$, if we take any Borel subalgebra $\mathfrak{b}'$ having the same even part as $\mathfrak{b}$, then there exists a unique $\lambda'$ such that the character of $M^{\mathfrak{b}'}(\lambda')$ coincides with that of $M^{\mathfrak{b}}(\lambda)$.

In contrast to even reflections, which generate a group known as the Coxeter group, odd reflections form a groupoid. Since morphisms in this groupoid are determined by their domain and codomain—that is, the groupoid is contractible—it is natural, as also noted in~\cite{gorelik2022root}, to consider the \emph{odd reflection graph} \( OR(\mathfrak{g}) \), where Borel subalgebras sharing same even part are represented as vertices, odd reflections as edges, and colors correspond to reflective hyperplanes.

For simplicity, we assume \( \lambda = 0 \). In this case, the following additional propertie also hold.
\begin{itemize}

    \item Composition of nonzero homomorphisms
    \[
        M^{\mathfrak{b}}(- \rho^{\mathfrak{b}}) \longrightarrow M^{\mathfrak{b}'}(- \rho^{\mathfrak{b}'}) \longrightarrow M^{\mathfrak{b}}(- \rho^{\mathfrak{b}})
    \]
    is zero whenever \( \mathfrak{b} \neq \mathfrak{b}' \).
\end{itemize} 

For example, if \( \mathfrak{g} = \mathfrak{gl}(2|2) \), then the odd reflection graph \( OR(\mathfrak{g}) \) becomes a  finite Young lattice \( L(2,2) \) (edge colored by coordinate of a box in Young diagrams).

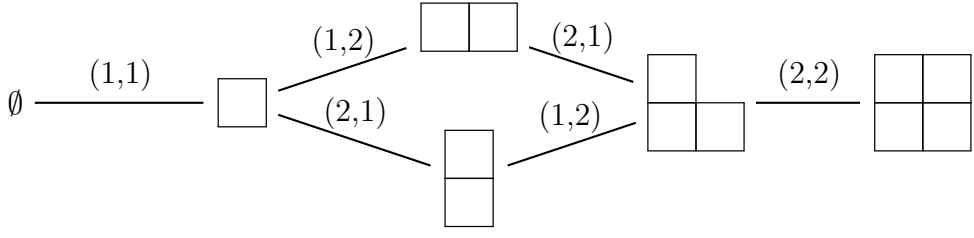
\begin{tikzpicture}[, thick, shorten >=1pt, scale=1, yscale=0.5]
    % Nodes with Young diagrams in French notation and adjusted coordinates
     \node (A) at (0, 0) {\(\emptyset\)};
    \node (B) at (3, 0) {\(\begin{ytableau} ~ \end{ytableau}\)};

    \node (C) at (6, 2) {\(\begin{ytableau} ~ & ~ \\ \end{ytableau}\)};
    \node (D) at (9, 0) {\(\begin{ytableau} ~ \\ ~ & ~ \end{ytableau}\)};
    \node (E) at (12, 0) {\(\begin{ytableau} ~ & ~ \\ ~ & ~ \end{ytableau}\)};
    
    \node (H) at (6, -2) {\(\begin{ytableau} ~ \\ ~ \end{ytableau}\)};

 % Edges with labels
\draw (A) -- node[above] {(1,1)} (B);
\draw (B) -- node[above] {(1,2)} (C);
\draw (C) -- node[above] {(2,1)} (D);
\draw (D) -- node[above] {(2,2)} (E);
\draw (B) -- node[above] {(2,1)} (H);
\draw (D) -- node[above] {(1,2)} (H);

\end{tikzpicture}

Each Borel subalgebra can be naturally identified with a Young diagram (in French notatio) contained in an \( 2\times 2 \) rectangle.

\[
\emptyset \leftrightarrow
\begin{pmatrix}
    * & * & \vline & * & * \\
    0 & * & \vline & * & * \\
    \hline
    0 & 0 & \vline & * & * \\
    0 & 0 & \vline & 0 & *
\end{pmatrix},
\quad
\begin{ytableau}
~
\end{ytableau} \leftrightarrow
\begin{pmatrix}
    * & * & \vline & * & * \\
    0 & * & \vline & 0 & * \\
    \hline
    0 & * & \vline & * & * \\
    0 & 0 & \vline & 0 & *
\end{pmatrix},
\quad
\begin{ytableau}
~ \\
~
\end{ytableau} \leftrightarrow
\begin{pmatrix}
    * & * & \vline & 0 & * \\
    0 & * & \vline & 0 & * \\
    \hline
    0 & * & \vline & * & * \\
    0 & * & \vline & 0 & *
\end{pmatrix}
\]

\[
\begin{ytableau}
~ & ~ \\
\end{ytableau} \leftrightarrow
\begin{pmatrix}
    * & * & \vline & * & * \\
    0 & * & \vline & 0 & 0 \\
    \hline
    * & * & \vline & * & * \\
    0 & 0 & \vline & 0 & *
\end{pmatrix},
\quad
\begin{ytableau}
~ \\
~ & ~ \\
\end{ytableau} \leftrightarrow
\begin{pmatrix}
    * & * & \vline & 0 & * \\
    0 & * & \vline & 0 & 0 \\
    \hline
    * & * & \vline & * & * \\
    0 & * & \vline & 0 & *
\end{pmatrix},
\quad
\begin{ytableau}
~ & ~ \\
~ & ~ \\
\end{ytableau} \leftrightarrow
\begin{pmatrix}
    * & * & \vline & 0 & 0 \\
    0 & * & \vline & 0 & 0 \\
    \hline
    * & * & \vline & * & * \\
    * & * & \vline & 0 & *
\end{pmatrix}
\].
It is natural to identify the vertices of the odd reflection graph with the corresponding Verma modules as follows.

\begin{center}
\begin{tikzpicture}[->, thick, shorten >=1pt, scale=1, yscale=0.5]

% Nodes as labels only (no extra dots)
\node (A) at (0,0) {\( M^{\emptyset}(-\rho^{\emptyset}) \)};
\node (B) at (3,0) {\( M^{(1)}(-\rho^{(1)}) \)};
\node (C) at (6,2) {\( M^{(1^2)}(-\rho^{(1^2)}) \)};
\node (D) at (6,-2) {\( M^{(2)}(-\rho^{(2)}) \)};
\node (E) at (9,0) {\( M^{(21)}(-\rho^{(21)}) \)};
\node (F) at (12,0) {\( M^{(2^2)}(-\rho^{(2^2)}) \)};

% Double arrows (opposite directions), curved to avoid overlapping labels
\draw[->] (A) to[bend left=20] (B);
\draw[->] (B) to[bend left=20] (A);

\draw[->] (B) to[bend left=20] (C);
\draw[->] (C) to[bend left=20] (B);

\draw[->] (B) to[bend left=20] (D);
\draw[->] (D) to[bend left=20] (B);

\draw[->] (C) to[bend left=20] (E);
\draw[->] (E) to[bend left=20] (C);

\draw[->] (D) to[bend left=20] (E);
\draw[->] (E) to[bend left=20] (D);

\draw[->] (E) to[bend left=20] (F);
\draw[->] (F) to[bend left=20] (E);

\end{tikzpicture}
\end{center}

In this way,  odd reflection version of Verma's theorem can be stated in a concise and explicit form as follows:

\begin{theorem}\label{1108}[The case \(\lambda = 0\) of \cref{5.3main}]
A composition of homomorphisms between modules \( M^{\mathfrak{b}}(-\rho^{\mathfrak{b}}) \) that are adjacent via odd reflections can be identified with a walk \( w \) on an edge-colored odd reflection graph \( OR(\mathfrak{g}) \). Under this identification, the followings are equivalent:
\begin{enumerate}
    \item \( w \neq 0 \), \item \( w \) is rainbow (i.e., each edge has a different color); \item \( w \) is shortest (among walks sharing the same endpoints).
\end{enumerate}
\end{theorem}

The equivalence between (2) and (3) is nothing but the property known as the exchange property of odd reflections, as noted in work of Gorelik-Hinich-Serganova \cite{gorelik2022root}, and \cref{1108} can be regarded as a representation-theoretic consequence of this property. This property can be regarded as an analogue of the well-known exchange property of Coxeter groups, since the fact that the length of an element and the inversion number of it is same translates into the statement that, in the Cayley graph, a walk is shortest if and only if it is rainbow. 

Although the exchange property of odd reflections can be verified directly in the case of \( \mathfrak{gl}(2|2) \), it can also be conceptually understood as follows. The groupoid generated by all even and odd reflections can be identified with the Cayley graph of the symmetric group \( S_{2+2} \), that is, a graph whose vertices correspond to bases of the $A_3$ root system, and whose edges correspond to simple reflections between them.. Upon removing all even reflections from this graph, we obtain a disjoint union of four finite Young lattices \( L(2,2) \), which allows us to illuminate the exchange property of \( L(2,2) \) from that of \( S_{2+2} \).
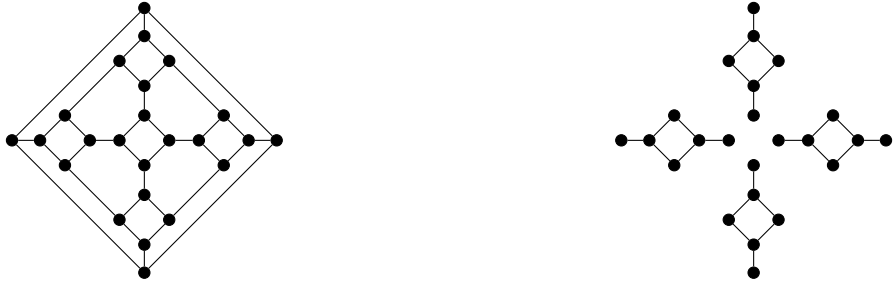
\begin{center}
 
\begin{minipage}{0.45\textwidth}
    \centering
 \begin{tikzpicture}[scale=0.7]
    % Center square
    \node[draw, circle, fill=black, inner sep=1.5pt] (A0) at (0,0.47) {};
    \node[draw, circle, fill=black, inner sep=1.5pt] (B0) at (0.47,0) {};
    \node[draw, circle, fill=black, inner sep=1.5pt] (C0) at (0,-0.47) {};
    \node[draw, circle, fill=black, inner sep=1.5pt] (D0) at (-0.47,0) {};
    \draw (A0) -- (B0) -- (C0) -- (D0) -- (A0);

    % Top square
    \node[draw, circle, fill=black, inner sep=1.5pt] (A1) at (0,1.97) {};
    \node[draw, circle, fill=black, inner sep=1.5pt] (B1) at (0.47,1.5) {};
    \node[draw, circle, fill=black, inner sep=1.5pt] (C1) at (0,1.03) {};
    \node[draw, circle, fill=black, inner sep=1.5pt] (D1) at (-0.47,1.5) {};
    \draw (A1) -- (B1) -- (C1) -- (D1) -- (A1);

    % Bottom square
    \node[draw, circle, fill=black, inner sep=1.5pt] (A2) at (0,-1.03) {};
    \node[draw, circle, fill=black, inner sep=1.5pt] (B2) at (0.47,-1.5) {};
    \node[draw, circle, fill=black, inner sep=1.5pt] (C2) at (0,-1.97) {};
    \node[draw, circle, fill=black, inner sep=1.5pt] (D2) at (-0.47,-1.5) {};
    \draw (A2) -- (B2) -- (C2) -- (D2) -- (A2);

    % Right square
    \node[draw, circle, fill=black, inner sep=1.5pt] (A3) at (1.5,0.47) {};
    \node[draw, circle, fill=black, inner sep=1.5pt] (B3) at (1.97,0) {};
    \node[draw, circle, fill=black, inner sep=1.5pt] (C3) at (1.5,-0.47) {};
    \node[draw, circle, fill=black, inner sep=1.5pt] (D3) at (1.03,0) {};
    \draw (A3) -- (B3) -- (C3) -- (D3) -- (A3);

    % Left square
    \node[draw, circle, fill=black, inner sep=1.5pt] (A4) at (-1.5,0.47) {};
    \node[draw, circle, fill=black, inner sep=1.5pt] (B4) at (-1.03,0) {};
    \node[draw, circle, fill=black, inner sep=1.5pt] (C4) at (-1.5,-0.47) {};
    \node[draw, circle, fill=black, inner sep=1.5pt] (D4) at (-1.97,0) {};
    \draw (A4) -- (B4) -- (C4) -- (D4) -- (A4);

    % Additional edges as requested
    \draw (A0) -- (C1);
    \draw (B0) -- (D3);
    \draw (C0) -- (A2);
    \draw (D0) -- (B4);

    % New external nodes
    \node[draw, circle, fill=black, inner sep=1.5pt] (NA1) at (0,2.5) {}; % Above A1
    \node[draw, circle, fill=black, inner sep=1.5pt] (NC2) at (0,-2.5) {}; % Below C2
    \node[draw, circle, fill=black, inner sep=1.5pt] (NB3) at (2.5,0) {}; % Right of B3
    \node[draw, circle, fill=black, inner sep=1.5pt] (ND4) at (-2.5,0) {}; % Left of D4

% Connect new nodes
    \draw (A1) -- (NA1);
    \draw (C2) -- (NC2);
    \draw (B3) -- (NB3);
    \draw (D4) -- (ND4);

    % Connect new external nodes
    \draw (NA1) -- (NB3);
    \draw (NB3) -- (NC2);
    \draw (NC2) -- (ND4);
    \draw (ND4) -- (NA1);

    % Connect additional requested edges
    \draw (B1) -- (A3);
    \draw (C3) -- (B2);
    \draw (D2) -- (C4);
    \draw (A4) -- (D1);

\end{tikzpicture}
\end{minipage}
\hfill
\begin{minipage}{0.45\textwidth}
    \centering
\begin{tikzpicture}[scale=0.7]
    % Center square
    \node[draw, circle, fill=black, inner sep=1.5pt] (A0) at (0,0.47) {};
    \node[draw, circle, fill=black, inner sep=1.5pt] (B0) at (0.47,0) {};
    \node[draw, circle, fill=black, inner sep=1.5pt] (C0) at (0,-0.47) {};
    \node[draw, circle, fill=black, inner sep=1.5pt] (D0) at (-0.47,0) {};

    % Top square
    \node[draw, circle, fill=black, inner sep=1.5pt] (A1) at (0,1.97) {};
    \node[draw, circle, fill=black, inner sep=1.5pt] (B1) at (0.47,1.5) {};
    \node[draw, circle, fill=black, inner sep=1.5pt] (C1) at (0,1.03) {};
    \node[draw, circle, fill=black, inner sep=1.5pt] (D1) at (-0.47,1.5) {};
    \draw (A1) -- (B1) -- (C1) -- (D1) -- (A1);

    % Bottom square
    \node[draw, circle, fill=black, inner sep=1.5pt] (A2) at (0,-1.03) {};
    \node[draw, circle, fill=black, inner sep=1.5pt] (B2) at (0.47,-1.5) {};
    \node[draw, circle, fill=black, inner sep=1.5pt] (C2) at (0,-1.97) {};
    \node[draw, circle, fill=black, inner sep=1.5pt] (D2) at (-0.47,-1.5) {};
    \draw (A2) -- (B2) -- (C2) -- (D2) -- (A2);

    % Right square
    \node[draw, circle, fill=black, inner sep=1.5pt] (A3) at (1.5,0.47) {};
    \node[draw, circle, fill=black, inner sep=1.5pt] (B3) at (1.97,0) {};
    \node[draw, circle, fill=black, inner sep=1.5pt] (C3) at (1.5,-0.47) {};
    \node[draw, circle, fill=black, inner sep=1.5pt] (D3) at (1.03,0) {};
    \draw (A3) -- (B3) -- (C3) -- (D3) -- (A3);

    % Left square
    \node[draw, circle, fill=black, inner sep=1.5pt] (A4) at (-1.5,0.47) {};
    \node[draw, circle, fill=black, inner sep=1.5pt] (B4) at (-1.03,0) {};
    \node[draw, circle, fill=black, inner sep=1.5pt] (C4) at (-1.5,-0.47) {};
    \node[draw, circle, fill=black, inner sep=1.5pt] (D4) at (-1.97,0) {};
    \draw (A4) -- (B4) -- (C4) -- (D4) -- (A4);

    % Additional edges as requested
    \draw (A0) -- (C1);
    \draw (B0) -- (D3);
    \draw (C0) -- (A2);
    \draw (D0) -- (B4);

    % New external nodes
    \node[draw, circle, fill=black, inner sep=1.5pt] (NA1) at (0,2.5) {}; % Above A1
    \node[draw, circle, fill=black, inner sep=1.5pt] (NC2) at (0,-2.5) {}; % Below C2
    \node[draw, circle, fill=black, inner sep=1.5pt] (NB3) at (2.5,0) {}; % Right of B3
    \node[draw, circle, fill=black, inner sep=1.5pt] (ND4) at (-2.5,0) {}; % Left of D4

% Connect new nodes
    \draw (A1) -- (NA1);
    \draw (C2) -- (NC2);
    \draw (B3) -- (NB3);
    \draw (D4) -- (ND4);

\end{tikzpicture}
\end{minipage}
   
\end{center}

As explained in the companion manuscript \cite{Hirota_PathSubgroupoids}, in general case, we can illuminate the exchange property of odd reflections from that of Weyl groupoids in the sence of Heckenberger and Yamane \cite{heckenberger2008generalization,heckenberger2020hopf}.  For $\mathfrak{gl}({2|2})$, the structure corresponding to the Weyl groupoid appears to be $S_{2+2}$. In this way,  \cref{1108} can be extended to the setting of regular symmetrizable Kac–Moody Lie superalgebras \cite{bonfert2024weyl,serganova2011kac} as well as Nichols algebras of diagonal type \cite{andruskiewitsch2017finite}.

For an arbitrary weight \( \lambda \), it may happen in this setting that \( M^{\mathfrak{b}}(\lambda - \rho^{\mathfrak{b}}) \cong M^{\mathfrak{b}'}(\lambda - \rho^{\mathfrak{b}'}) \) even when \( \mathfrak{b} \neq \mathfrak{b}' \). To ensure that the modules along a walk remain pairwise non-isomorphic, we introduce a suitable quotient graph. Remarkably, the exchange property is inherited by this quotient graph, and the  \cref{1108} continues to hold in the same form.

We also apply the exchange  property of odd reflections in Section 6 to improve the known results, \cite[Lemma 5.12]{coulembier2017homological} and a part of \cite[Proposition 34]{chen2023some}, on the associated variety and projective dimension of Verma modules in category \( \mathcal{O}\) in a more transparent manner to the setting of arbitrary basic Lie superalgebras and Borel subalgebras.  
As a consequence, we prove the following, which was previously known in the case of type I Lie superalgebras with a distinguished choice of Borel subalgebra. 

\begin{theorem}(\cref{d21main,typeI_RBtriv})
If \( \mathfrak{g} \) is of type I or \( D(2, 1; \alpha) \), then for arbitrary Borel subalgebra \( \mathfrak{b} \) and weight \( \lambda \), the following hold:
\begin{itemize}
    \item \( \lambda \) is \( \mathfrak{b} \)-typical;
    \item The projective dimension of \( M^{\mathfrak{b}}(\lambda) \) in \(\mathcal{O}\) is finite.
\end{itemize}
\end{theorem}

In the case of type II, partial results also were obtained.

\subsection{Acknowledgements}

I would like to express my heartfelt gratitude to my supervisor, Syu Kato, for his patient and extensive guidance throughout the preparation of master's thesis. Discussions with Istvan Heckenberger, Yoshiyuki Koga, and Hiroyuki Yamane, to whom I am also grateful, have influenced this work.  The author would like to thank the Kumano Dormitory community at Kyoto University for their generous financial and living assistance.

This work was supported by the Japan Society for the Promotion of Science (JSPS) through the Research Fellowship for Young Scientists (DC1), Grant Number JP25KJ1664.

\section{Edge colored graphs} \label{sec:RBgraph}

In this section, we recall basic terminology related to edge colored graphs.

\begin{definition}[Edge-Colored Graph]
An \emph{edge-colored graph} is a triple \( (G, \phi, C) \), where:
\begin{itemize}
    \item \( G = (V, E) \) is a graph with a vertex set \( V \) and an edge set \( E \);
    \item \( \phi: E \to C \) is a function that assigns a color \( \phi(e) \in C \) to each edge \( e \in E \), where \( C \) is a set of colors.
\end{itemize}

Additionally, we assume \(\phi\) is surjective.

\end{definition}

\begin{definition}

Let \( G \) be a edge-colored graph with a vertex set \( V \) and a color set \( C \).

A \emph{walk} of length \(k\) in a colored graph \( G \) is a finite non-null sequence
\[
W = v_0 c_1 v_1 \dots c_k v_k,
\]
indicating that for each \( 1 \leq i \leq k \), there exists an edge between \( v_{i-1} \) and \( v_i \) colored \( c_i \).  
Since the graphs we consider have all edge colors mutually distinct, this notation causes no ambiguity.

A walk \( v_0 c_1 v_1 \dots c_k v_k \) is called a \emph{path} if \( v_0, v_1, \dots, v_k \) are all distinct.

If \( W = v_0 c_1 v_1 \dots c_k v_k \) and \( W' = v_k c_{k+1} v_{k+1} \dots c_l v_l \) are walks, then the walk \( v_k c_k v_{k-1} \dots v_1 c_1 v_0 \), obtained by reversing \( W \), is denoted by \( W^{-1} \), and the walk \( v_0 c_1 v_1 \dots c_k v_k c_{k+1} v_{k+1} \dots c_l v_l \), obtained by concatenating \( W \) and \( W' \) at \( v_k \), is denoted by \( WW' \).

 A walk is termed \emph{shortest} if there is no shorter walk between the same pair of vertices. Note that  a shortest walk is a path. In a connected graph, for any pair of vertices, a shortest walk (not necessarily unique) always exists.

 A walk is called a \emph{rainbow} if all the edge colors in its sequence are distinct. \cite{li2013rainbow}
    
\end{definition}

\begin{definition}\label{d1}
    Let \( G \) be an edge-colored graph with color set \( C \). Let \( D \subseteq C \). We define an equivalence relation \( \sim_D \) on \( V \) as follows: For \( x, y \in V \), we say \( x \sim_D y \) if there exists a walk from \( x \) to \( y \) consisting only of edges with colors in \( D \). We denote the equivalence class of \( x \) by \( [x] \).

    We define the edge-colored graph \( G/D \) as follows:
    \begin{itemize}
        \item The vertex set is \( V/\sim_D \), the set of equivalence classes under \( \sim_D \);
        \item The color set is \( C \setminus D \);
        \item For a color \( c \in C \setminus D \), there uniquely exists an edge of color \( c \) between \( [x] \) and \( [y] \) 
if there exist \( u \in [x] \) and \( v \in [y] \) such that there is an edge of color \( c \) between \( u \) and \( v \) in \( G \).
    \end{itemize}
Given a walk \( W \) in \( G \):
\[
v_0 c_0 v_1 c_1 \dotsc c_{k-1} v_k,
\]
we define the \emph{induced walk} \( \overline{W} \) in \( G/D \) as:
\[
[v_0] [c_0] [v_1] [c_1] \dotsc [c_{k-1}] [v_k],
\]
where \( [c_i] = c_i \) if \( c_i \in C\setminus D \), and \( [c_i] \) represents an empty walk if \( c_i \in D\).

\end{definition}

In the following, we present several examples of edge-colored graphs that are of particular importance in our setting.

   \begin{definition}\label{g.ydi}
    
A partition of an integer \( n \) is often represented as a sequence of non-increasing integers, where exponents are used to indicate repeated parts more compactly. For instance, the partition \( (4432221) \) can be written as \( (4^232^31) \). We denote the empty partition, by \( \emptyset\).

We identify partitions with their corresponding Young diagrams. Let \( P_{m \times n} \) denote the set of all Young diagrams that are contained within an \( m \times n \) rectangle. To better reflect the structure of Borel subalgebras, We adopt the French convention, where each row of the Young diagram corresponds to a part of the partition, with larger parts appearing lower in the diagram. For instance, the partition \( (4,3,2) \) is represented with the row of 4 boxes at the bottom, the row of 3 boxes above it, and the row of 2 boxes at the top..
We represent each box in a Young diagram by its coordinates \( (i,j) \), where \( j \) is the row (counted from the bottom) and \( i \) is the column (counted from the left).
\end{definition}

\begin{example}[Edge-colored Young's Lattice]
Define an edge-colored graph \( Y \) as follows:
\begin{itemize}
    \item \textbf{Vertex set :} The set of all Young diagrams (in French notation);
    
    \item \textbf{Color set :} \( C = \mathbb{N}_{>0} \times \mathbb{N}_{>0} \);
    
    \item \textbf{Edges:} There is an edge of color \( (i, j) \) between vertices \( x \) and \( y \) if and only if \( y \) is obtained from \( x \) by adding or removing a box at coordinate \( (i, j) \) and not changing the position of the other boxes.
\end{itemize}

\end{example}

 \begin{example}[Edge-colored Finite Young's Lattice \cite{coggins2024visual,stanley2012topics}.] \label{3.1lmn}

    Let \( D = C \setminus (\{1, \dots, n\} \times \{1, \dots, m\}) \), and define an edge colored graph \( L(m, n) = Y / D \). The vertex set is precisely \( P_{m \times n} \). 
\end{example} 
\begin{example}
    
\label{L32}

    \(L(3,2)\) is the following edge-colored graph.

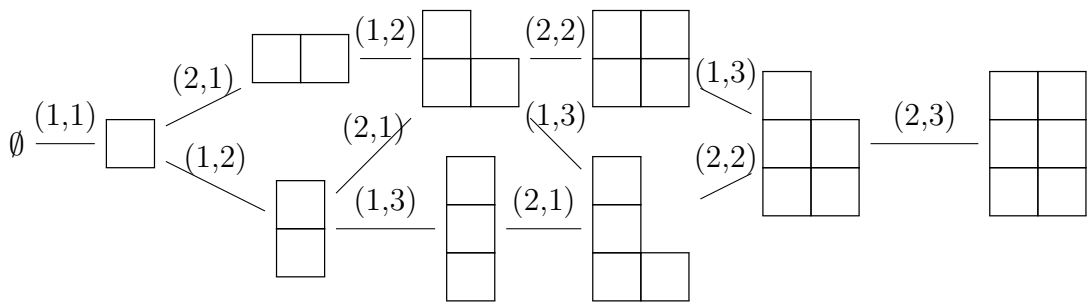
\begin{center}

\begin{tikzpicture}[scale=1.5]
    % Nodes with Young diagrams in French notation and adjusted coordinates
     \node (A) at (0, 0) {\(\emptyset\)};
    \node (B) at (1, 0) {\(\begin{ytableau} ~ \end{ytableau}\)};

    \node (C) at (2.5, 0.75) {\(\begin{ytableau} ~ & ~ \\ \end{ytableau}\)};
    \node (D) at (4, 0.75) {\(\begin{ytableau} ~ \\ ~ & ~ \end{ytableau}\)};
    \node (E) at (5.5, 0.75) {\(\begin{ytableau} ~ & ~ \\ ~ & ~ \end{ytableau}\)};
    \node (F) at (7, 0) {\(\begin{ytableau} ~ \\ ~ & ~ \\ ~ & ~  \end{ytableau}\)};
        \node (G) at (9, 0) {\(\begin{ytableau} ~ & ~ \\  ~ & ~  \\ ~ & ~ \end{ytableau}\)};
    \node (H) at (2.5, -0.75) {\(\begin{ytableau} ~ \\ ~ \end{ytableau}\)};
    \node (I) at (4, -0.75) {\(\begin{ytableau} ~ \\ ~ \\ ~ \end{ytableau}\)};
    \node (J) at (5.5, -0.75) {\(\begin{ytableau} ~ \\ ~ \\ ~ & ~ \end{ytableau}\)};
    
 % Edges with labels
\draw (A) -- node[above] {(1,1)} (B);
\draw (B) -- node[above] {(2,1)} (C);
\draw (C) -- node[above] {(1,2)} (D);
\draw (D) -- node[above] {(2,2)} (E);
\draw (E) -- node[above] {(1,3)} (F);
\draw (F) -- node[above] {(2,3)} (G);
\draw (B) -- node[above] {(1,2)} (H);
\draw (D) -- node[above] {(2,1)} (H);
\draw (D) -- node[above] {(1,3)} (J);
\draw (F) -- node[above] {(2,2)} (J);
\draw (H) -- node[above] {(1,3)} (I);
\draw (I) -- node[above] {(2,1)} (J);
\end{tikzpicture}.
\end{center}

If \( D = \{(1,1), (1,2), (2,3)\} \) , then  \(L(3,2)/D \) is the following edge-colored graph.

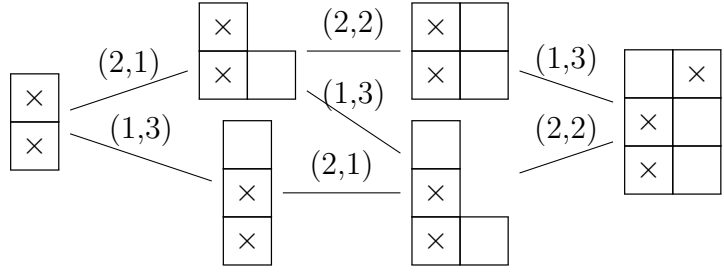
\begin{center}
    
\begin{tikzpicture}
    % Nodes with coordinates scaled by 1.25
    \node (A) at (0, 1.875) {
    \begin{ytableau}
    \times \\ 
    \times  
    \end{ytableau}
    };
    \node (B) at (2.8125, 2.8125) {
    
    \begin{ytableau}
    \times \\ 
    \times  & ~
    \end{ytableau}};
    \node (C) at (5.625, 2.8125) {
    \begin{ytableau}
    \times & ~ \\ 
    \times  & ~
    \end{ytableau}
    };
    \node (D) at (8.4375, 1.875) {
    \begin{ytableau}
    ~ & \times \\ 
    \times & ~ \\ 
    \times  & ~
    \end{ytableau}
    };
    \node (E) at (2.8125, 0.9375) {
    
    \begin{ytableau}
    ~ \\ 
    \times \\ 
    \times  
    \end{ytableau}
    };
    \node (F) at (5.625, 0.9375) {
    \begin{ytableau}
    ~ \\ 
    \times \\ 
    \times  & ~
    \end{ytableau}
    };

    \draw (A) -- node[above] {(2,1)} (B);
    \draw (B) -- node[above] {(2,2)} (C);
    \draw (A) -- node[above] {(1,3)} (E);
    \draw (B) -- node[above] {(1,3)} (F);
    \draw (C) -- node[above] {(1,3)} (D);
    \draw (E) -- node[above] {(2,1)} (F);
    \draw (D) -- node[above] {(2,2)} (F);
\end{tikzpicture}.
\end{center}

\end{example}

\begin{example}
Define the edge-colored graph \( LD(\infty, n) \) as follows:
\begin{itemize}
    \item \textbf{Vertex set \( V \):} 
    \hspace{-15mm}\[
    V = \left( \{ \text{Young diagrams of height less than } n \} \times \{ +,- \} \right) \cup \{ \text{Young diagrams of height } n \}.
    \]
    (We denote as \( (Y_{<n} \times \{ +,- \}) \cup Y_n \)).

    \item \textbf{Color set \( C \):} 
    \[
    C = \left( \{ 1, \dotsc, n \} \times \{ +,- \} \right) \cup \left( \{ 1, \dotsc, n \} \times \mathbb{N}_{>1} \right).
    \]

    \item \textbf{Edges:} For \( x, y \in Y_{<n} \), and \( w, z \in Y_n \), the edges are defined as follows:
    \begin{itemize}
        \item There are no edges between \( (x, +) \) and \( (y, -) \);
        \item Between \( (x, \pm) \) and \( (y, \pm) \), there is an edge of color \( (n, \pm) \) if and only if \( y \) is obtained from \( x \) by adding or removing a box at coordinate \( (n, 1) \);
        \item Between \( (x, \pm) \) and \( z \), there is an edge of color \( (n, \pm) \) if and only if \( z \) is obtained from \( x \) by adding the box at \( (n, 1) \);
        \item Between \( (x, \pm) \) and \( (y, \pm) \), there is an edge of color \( (i, j) \) if and only if \( y \) is obtained from \( x \) by adding or removing a box at coordinate \( (i, j) \) such that \( j > 1 \);
        \item Between \( w \) and \( z \), there is an edge of color \( (i, j) \) if and only if \( w \) is obtained from \( z \) by adding or removing a box at coordinate \( (i, j) \).
    \end{itemize}
\end{itemize}

Note that, as we adopt the French notation, the height is actually represented as the width.
\end{example}

 \begin{example}\label{Dmn}\cite{bonfert2024weyl}
Let 
\[
D = C \setminus \left( \{ 1, \dotsc, n \} \times \{ +,- \} \cup \{ 1, \dotsc, n \} \times \{ 2, \dotsc, n \} \right),
\]
and define the  edge clored graph
\[
LD(m, n) := LD(\infty, n) / D.
\]
\end{example}

\begin{example}
 \( LD(2,3) \) is the following edge-colored graph.
\\
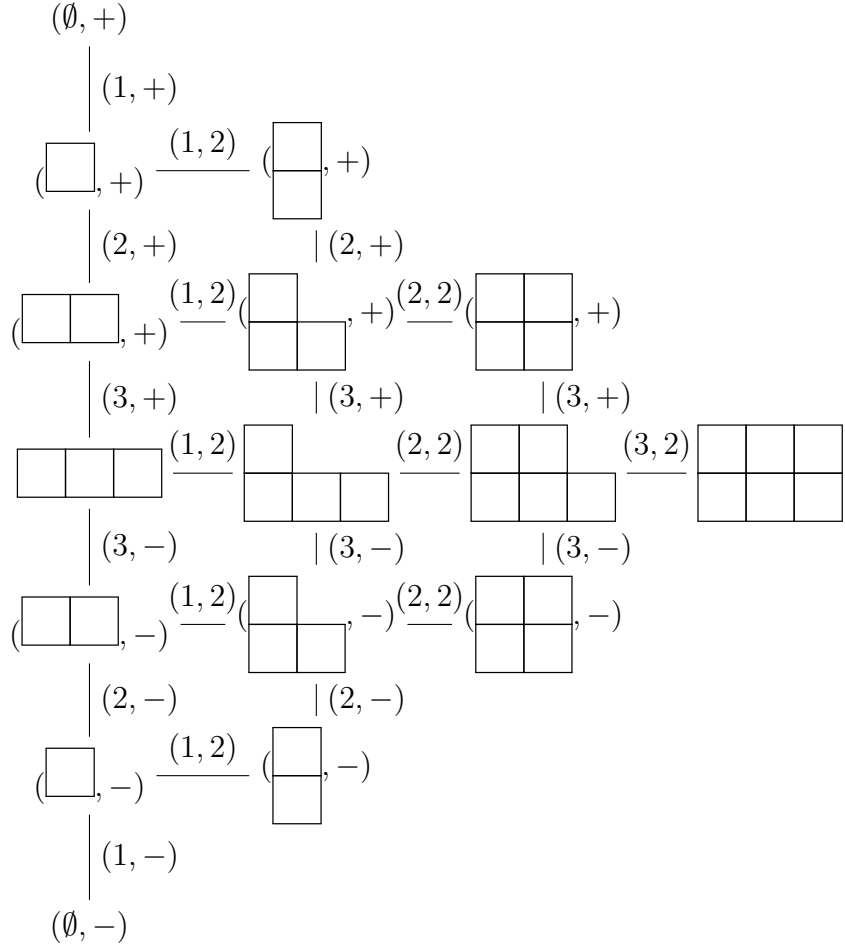
\begin{center}
     \begin{tikzpicture}
    % 上のグラフ
    % Nodes
    \node (A) at (3,2) {\((\emptyset,+ )\)};
    \node (B) at (3,0) {\(( \begin{ytableau} \ \end{ytableau},+ )\)};
    \node (C) at (6,0) {\(( \begin{ytableau} \ \\ \ \end{ytableau},+ )\)};
    \node (D) at (3,-2) {\(( \begin{ytableau} \ & \ \\ \end{ytableau},+ )\)};
    \node (E) at (6,-2) {\(( \begin{ytableau} \ \\ \ & \ \\ \end{ytableau},+ )\)};
    \node (F) at (9,-2) {\(( \begin{ytableau} \ & \ \\ \ & \ \\ \end{ytableau},+ )\)};
     
    % 高さ -2.5 の線分用ノード
    \node (P1) at (3,-4) {\begin{ytableau}
\ & \ & \ \\
\end{ytableau}};
    \node (P2) at (6,-4) {\begin{ytableau}
\ \\
\ & \ & \ \\
\end{ytableau}};
    \node (P3) at (9,-4) {\begin{ytableau}
\ &\ \\
\ & \ & \ \\
\end{ytableau}};
    \node (P4) at (12,-4) {\begin{ytableau}
\ & \ & \ \\
\ & \ & \ \\
\end{ytableau}};
    % 下のグラフ
    % Shift all nodes downward
    \node (A2) at (3,-10) {\( (\emptyset,- )\)};
    \node (B2) at (3,-8) {\(( \begin{ytableau} \ \end{ytableau},- )\)};
    \node (C2) at (3,-6) {\(( \begin{ytableau} \ & \ \\ \end{ytableau},- )\)};
    \node (D2) at (6,-8) {\(( \begin{ytableau} \ \\ \ \end{ytableau},- )\)};
    \node (E2) at (6,-6) {\(( \begin{ytableau} \ \\ \ & \ \\ \end{ytableau},- )\)};
    \node (F2) at (9,-6) {\(( \begin{ytableau} \ & \ \\ \ & \ \\ \end{ytableau},- )\)};

    % Edges with labels
    \draw (A) -- node[right] {\( (1,+ )\)} (B);
    \draw (B) -- node[above ] {\( (1,2)\)} (C);
    \draw (B) -- node[right] {\( (2,+ )\)} (D);
    \draw (D) -- node[above ] {\( (1,2)\)} (E);
    \draw (C) -- node[right] {\((2,+ )\)} (E);
    \draw (E) -- node[above] {\( (2,2)\)} (F);

    \draw (A2) -- node[right] {\( (1,- )\)} (B2);
    \draw (B2) -- node[right] {\( (2,- )\)} (C2);
    \draw (B2) -- node[above] {\( (1,2)\)} (D2);
    \draw (D2) -- node[right] {\( (2,- )\)} (E2);
    \draw (C2) -- node[above ] {\((1,2)\)} (E2);
    \draw (E2) -- node[above] {\( (2,2)\)} (F2);

    % Horizontal line for P-nodes
    \draw (P1) -- node[above] {\((1,2)\)} (P2);
    \draw (P2) -- node[above] {\((2,2)\)} (P3);
    \draw (P3) -- node[above] {\((3,2)\)} (P4);

% P1 connections
\draw (P1) -- node[right] {\((3,+)\)} (D);
\draw (P1) -- node[right] {\((3,-)\)} (C2);

% P2 connections
\draw (P2) -- node[right] {\((3,+)\)} (E);
\draw (P2) -- node[right] {\((3,-)\)} (E2);

% P3 connections
\draw (P3) -- node[right] {\((3,+)\)} (F);
\draw (P3) -- node[right] {\((3,-)\)} (F2);
\end{tikzpicture}

.
\end{center}
   
\end{example}

\begin{example}\label{g.qn}
The hypercube graph \( Q_n \) is defined as follows:
\begin{itemize}
    \item Vertex set: \( (\mathbb{Z}/2\mathbb{Z})^{\oplus n} \);
    \item Edge set: An edge exists between two vertices if they differ in exactly one coordinate.
\end{itemize}
 The graph \( Q_n \) admits a natural proper edge coloring with the color set of size \( n \).

\end{example}

\section{Changing Borel subalgebras} \label{sec:changing_Borel}
Fix the base field \( k \) as an algebraically closed field of characteristic 0.
The results of this section are essentially not new. See \cite{serganova2011kac}.

\subsection{Basic Lie superalgebras and their odd reflection graphs}

From now on, we will denote by \( \mathfrak{g} \) a finite dimensional Lie superalgebra.
 In the case of infinite-dimensional Lie superalgebras, attention must be paid to the definition of the Weyl vector and infiniteness of the lengths of modules; however, thanks to the results of \cite{serganova2011kac}, our results still make sense. 
 We denote the even and odd parts of \( \mathfrak{g} \) as \( \mathfrak{g}_{\overline{0}} \) and \( \mathfrak{g}_{\overline{1}} \), respectively.

In this text, whenever we refer to "dimension", we mean the dimension as a vector space, forgetting the \(\mathbb{Z}/2\mathbb{Z}\)-grading (not the dimension in the sense of the theory of symmetric monoidal categories in \cite{etingof2015tensor}, i.e., the superdimension).

The following theorem plays a significant role throughout this paper.
\begin{theorem}[PBW Theorem, \cite{musson2012lie} Theorem 6.1.2]\label{2.3.pbw}
    Let \( \mathfrak{g} \) be a finite-dimensional Lie superalgebra, and let \( x_1, \dots, x_n \) be a \( \mathbb{Z}/2\mathbb{Z} \)-homogeneous basis of \( \mathfrak{g} \). Then
    \[
    \{ x_1^{a_1} \dots x_n^{a_n} \mid a_i \in \mathbb{Z}_{\geq 0} \text{ if } x_i \in \mathfrak{g}_{\overline{0}}, \text{ and } a_i \in \{0,1\} \text{ if } x_i \in \mathfrak{g}_{\overline{1}} \}
    \]
    forms a basis for \( U(\mathfrak{g}) \), the universal enveloping algebra of \( \mathfrak{g} \).
\end{theorem}

\begin{definition}

We consider the category \( \mathfrak{g}\text{-sMod} \), where morphisms respects \(\mathbb{Z}/2\mathbb{Z}\)-grading. (This is the module category of a monoid object in the monoidal category of super vector spaces in the sense of \cite{etingof2015tensor}.)

Let \( \mathfrak{g}\text{-smod} \) denote the full subcategory of \( \mathfrak{g}\text{-sMod} \) consisting of finite-dimensional modules. 

The parity shift functor \( \Pi \) on \( \mathfrak{g}\text{-sMod} \) is an exact functor that acts by preserving the underlying \( \mathfrak{g} \)-module structure but reversing the \( \mathbb{Z}/2\mathbb{Z} \)-grading, while also preserving all morphisms.

It is well known that the restriction functor \( \operatorname{Res}^{\mathfrak{g}}_{\mathfrak{g}_{\overline{0}}}: \mathfrak{g}\text{-sMod} \to \mathfrak{g}_{\overline{0}}\text{-sMod} \) and the induction functor \( \operatorname{Ind}^{\mathfrak{g}}_{\mathfrak{g}_{\overline{0}}}: \mathfrak{g}_{\overline{0}}\text{-sMod} \to \mathfrak{g}\text{-sMod} \) are exact functors \cite[Theorem 2.2]{bell1993theory}, \cite[Section 6.1]{coulembier2017gorenstein}.
\end{definition}

From now on, our \( \mathfrak{g} \) will be a direct sum of one of the finite-dimensional basic Lie superalgebras from the following list:
\[
\mathfrak{gl}(m|n), \, \mathfrak{osp}(m|2n), \, D(2, 1; \alpha), \, G(3), \, F(4).
\]
For concrete definitions, we refer to \cite[Chapters 1-4]{musson2012lie} and \cite{cheng2012dualities}.

According to \cite{kac1977lie}, if \( \mathfrak{g} =  \mathfrak{gl}(m|n), \, \mathfrak{osp}(2|2n) \), then \( \mathfrak{g} \) is said to be of type I; otherwise, it is of type II. If \( \mathfrak{g} = D(2, 1; \alpha), \, G(3), \, F(4) \), then \( \mathfrak{g} \) is said to be of exceptional type.

\begin{definition}[\cite{musson2012lie, cheng2012dualities, serganova2017representations}]\label{3.1delta}
A Cartan subalgebra and the Weyl group of the reductive Lie algebra \( \mathfrak{g}_{\overline{0}} \) are denoted by \( \mathfrak{h} \) and \( W \), respectively.

A basic Lie superalgebra \( \mathfrak{g} \) has a supersymmetric, superinvariant, even bilinear form \( \langle \, , \, \rangle \), which induces a \( W \)-invariant bilinear form \( (\, , \,) \) on \( \mathfrak{h}^{*} \) via duality.

The root space \( \mathfrak{g}_\alpha \) associated with \( \alpha \in \mathfrak{h}^* \) is defined as 
\(
\mathfrak{g}_\alpha := \{ x \in \mathfrak{g} \mid [h, x] = \alpha(h)x \, \text{for all } h \in \mathfrak{h} \}.
\)

The set of roots \( \Delta \) is defined as 
\(
\Delta := \{ \alpha \in \mathfrak{h}^* \mid \mathfrak{g}_\alpha \neq 0 \} \setminus \{0\}.
\)
Each \( \mathfrak{g}_\alpha \) is either purely even or purely odd and is one-dimensional. Therefore, the notions of even roots and odd roots are well defined. An odd root \( \alpha \) is said to be \emph{isotropic} if \( (\alpha, \alpha) = 0 \). The sets of all even roots, even positive roots, odd roots and odd isotropic roots are denoted by \( \Delta_{\overline{0}} \), \( \Delta_{\overline{0}}^+ \)), \( \Delta_{\overline{1}} \) and \( \Delta_{\otimes} \), respectively.

\end{definition}

\begin{definition}[\cite{musson2012lie,cheng2012dualities,serganova2017representations}]
We fix a Borel subalgebra \( \mathfrak{b}_{\overline{0}} \) of \( \mathfrak{g}_{\overline{0}} \). The set of all Borel subalgebras \( \mathfrak{b} \) of \( \mathfrak{g} \) that contain \( \mathfrak{b}_{\overline{0}} \) is denoted by \( \mathfrak{B(g)} \).

For a Borel subalgebra \( \mathfrak{b}  \in \mathfrak{B(g)} \), we express the triangular decomposition of \( \mathfrak{g} \) as
\[
\mathfrak{g} = \mathfrak{n}^{\mathfrak{b}-} \oplus \mathfrak{h} \oplus \mathfrak{n}^{\mathfrak{b}+},
\]
where \( \mathfrak{b} = \mathfrak{h} \oplus \mathfrak{n}^{\mathfrak{b}+} \).

 The sets of positive roots, odd positive roots, and odd isotropic positive roots corresponding to \( \mathfrak{b} \) are denoted by \( \Delta^{\mathfrak{b}+} \), \( \Delta_{\overline{1}}^{\mathfrak{b}+} \), and \( \Delta_{\otimes}^{\mathfrak{b}+} \), respectively. The set of simple roots (basis) corresponding to \( \Delta^{\mathfrak{b}+} \) is denoted by \( \Pi^{\mathfrak{b}} \). We define \( \Pi_{\otimes}^{\mathfrak{b}} :=  \Pi^{\mathfrak{b}} \cap \Delta_{\otimes} \).
We define
\[
\Delta^{\operatorname{pure+}} := \bigcap_{\mathfrak{b} \in \mathfrak{B(g)}} \Delta^{\mathfrak{b}+},
\]
\[
\Delta_{\otimes}^{\operatorname{pure+}} := \bigcap_{\mathfrak{b} \in \mathfrak{B(g)}} \Delta_{\otimes}^{\mathfrak{b}+} = \Delta^{\operatorname{pure+}} \cap \Delta_{\otimes}.
\]
\end{definition}

\begin{theorem}[Odd reflection \cite{musson2012lie} 3.5]\label{2.3.oddref} 
    For \( \alpha \in \Pi_{\otimes}^{\mathfrak{b}} \), define \( r^{\mathfrak{b}}_{\alpha} \in \operatorname{Map}(\Pi^{\mathfrak{b}}, \Delta )\) by
    \[
        r^{\mathfrak{b}}_{\alpha}(\beta) = 
        \begin{cases}
            -\alpha & (\beta = \alpha), \\
            \alpha + \beta & (\alpha + \beta \in \Delta), \\
            \beta & (\text{otherwise}).
        \end{cases}
    \]
    for \( \beta \in \Pi^{\mathfrak{b}} \).
(When there is no risk of confusion, \( r_\alpha^\mathfrak{b} \) is abbreviated as \( r_\alpha \).)
   A Borel subalgebra \( r_{\alpha} \mathfrak{b} \in \mathfrak{B(g)} \) exists, with the corresponding basis given by 
\[
\Pi^{r_{\alpha} \mathfrak{b}} := \{ r^{\mathfrak{b}}_{\alpha}(\beta) \}_{\beta \in \Pi^{\mathfrak{b}}}.
\]

\end{theorem}
The linear transformation of \( \mathfrak{h}^* \) induced by an odd reflection does not necessarily map a Borel subalgebra to another Borel subalgebra. For example, when \( \mathfrak{g} = D(2,1;\alpha) \), none of the odd reflections have this property. (see, \cref{2.3.d21}).

The following is well-known:
\begin{proposition}[\cite{musson2012lie,cheng2012dualities}]\label{2.3.odd_kyoyaku}
Each pair of elements \( \mathfrak{b},\mathfrak{b'} \in \mathfrak{B(g)} \) due to transferred to each other by a finite number of odd reflections.
\end{proposition}

\begin{definition}\label{RBgdef}
    The edge colored graph \( OR(\mathfrak{g}) \) is defined as follows:
    \begin{itemize}
        \item \textbf{Vertex set:} \( \mathfrak{B(g)} \);
        \item \textbf{Color set:} For a fixed \( \mathfrak{b} \in \mathfrak{B(g)} \), the set \( \Delta^{\mathfrak{b}+} \setminus \Delta^{\text{pure+}} \);
        \item \textbf{Edges and colors:} An edge is drawn between two vertices if they are related by an odd reflection. The edge is assigned a color corresponding to the unique \( \alpha \in \Delta^{\mathfrak{b}+} \setminus \Delta^{\text{pure+}} \) such that \( \alpha \) belongs to the positive root system of one vertex but not the other.
    \end{itemize}
    Since the positive root systems associated with different Borel subalgebras are in one-to-one correspondence, the structure of the odd reflection graph \( OR(\mathfrak{g}) \) does not depend on the choice of \( \mathfrak{b} \).
\end{definition}

% Definition of \( OR(\mathfrak{g}) \)
\begin{definition}
    For \( \lambda \in \mathfrak{h}^* \) and \( \mathfrak{b} \in \mathfrak{B(g)} \), we set
    \[
        D_{\lambda} := \{ \alpha \in \Delta_\otimes^{\mathfrak{b}+} \setminus \Delta^{\text{pure}+} \mid (\lambda, \alpha) \neq 0 \}.
    \]
    We define an edge colored graph \( OR(\mathfrak{g}, \lambda) := OR(\mathfrak{g}) / D_{\lambda} \) in the sense of \cref{d1}. For example, we have
    \( OR(\mathfrak{g}, 0) = OR(\mathfrak{g}) \).
\end{definition}

The following is a fundamental property that we refer to as the \emph{exchange property} of odd reflections:

\begin{theorem}[Exchange Property of Odd Reflections] \label{thm:odd_exchange}
Let \( w \) be any walk on the graph \( \operatorname{OR}(\mathfrak{g}, \lambda) \). Then the following are equivalent:
\begin{itemize}
    \item[(1)]  \( w \) is rainbow.
    \item[(2)]  \( w \) is shortest.
\end{itemize}
\end{theorem}

\begin{remark}
  You can verify \cref{thm:odd_exchange} case by case refering \cref{2}.
  A conceptual proof of \cref{thm:odd_exchange} is given in \cite{Hirota_PathSubgroupoids} in more general setting including the case of regular symmetrizable Kac-Moody Lie superalgebras and Nichols algebras of diagonal type.
\end{remark}

\begin{remark}
The condition in \cref{thm:odd_exchange} is very strong. In fact, the only connected cycles satisfying condition in \cref{thm:odd_exchange} are the following:
    \begin{center}
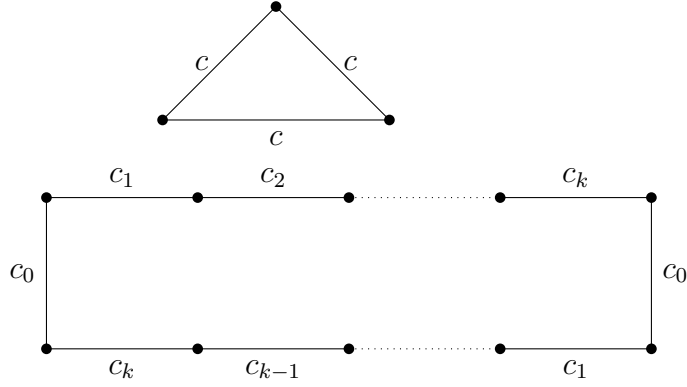
\begin{minipage}{0.45\textwidth}
\centering
\begin{tikzpicture}
    % Triangle graph
    \coordinate (A) at (0, 1.5);
    \coordinate (B) at (-1.5, 0);
    \coordinate (C) at (1.5, 0);

    % Edges with labels
    \draw (A) -- node[left] {\( c \)} (B);
    \draw (B) -- node[below] {\( c \)} (C);
    \draw (C) -- node[right] {\( c \)} (A);

    % Nodes (black circles)
    \foreach \point in {A,B,C} {
        \fill (\point) circle (2pt);
    }
\end{tikzpicture}
\end{minipage}
\hspace{1cm}
\begin{minipage}{0.5\textwidth}
\centering
\begin{tikzpicture}
    % Top path
    \coordinate (A1) at (0, 0);
    \coordinate (A2) at (2, 0);
    \coordinate (A3) at (4, 0);
    \coordinate (Ak1) at (6, 0);
    \coordinate (Ak) at (8, 0);

    % Bottom path
    \coordinate (B1) at (0, -2);
    \coordinate (B2) at (2, -2);
    \coordinate (B3) at (4, -2);
    \coordinate (Bk1) at (6, -2);
    \coordinate (Bk) at (8, -2);

    % Top edges
    \draw (A1) -- node[above] {\( c_1 \)} (A2);
    \draw (A2) -- node[above] {\( c_2 \)} (A3);
    \draw[dotted] (A3) -- (Ak1);
    \draw (Ak1) -- node[above] {\( c_k \)} (Ak);

    % Bottom edges
    \draw (B1) -- node[below] {\( c_k \)} (B2);
    \draw (B2) -- node[below] {\( c_{k-1} \)} (B3);
    \draw[dotted] (B3) -- (Bk1);
    \draw (Bk1) -- node[below] {\( c_1 \)} (Bk);

    % Vertical edges
    \draw (A1) -- node[left] {\( c_0 \)} (B1);
    \draw (Ak) -- node[right] {\( c_0 \)} (Bk);

    % Nodes as black dots
    \foreach \point in {A1,A2,A3,Ak1,Ak,B1,B2,B3,Bk1,Bk} {
        \fill (\point) circle (2pt);
    }
\end{tikzpicture}
\end{minipage}
\end{center}

\vspace{0.5em}
\noindent
Here, the colors satisfy the conditions that
\(
c_0, c_1, \ldots, c_k
\)
are all distinct, and \(k \geq 0\).
\end{remark}

As explained in \cite{Hirota_PathSubgroupoids}, the following can be directly deduced from \cref{thm:odd_exchange}.
\begin{theorem}
\label{rbgrb}[cf.\ {\cite[Proposition 5.3.5]{gorelik2022root}}]\label{g.rb2.kai}
The graph \( OR(\mathfrak{g}, \lambda) \) satisfies following propery:

If there exists a rainbow walk
\[
v_0c_0v_1c_1\dots c_kv_{k+1} \quad (k > 0),
\]
and an edge \( v_{k+1}c_0v_{k+2} \),
then there exists a rainbow walk
\[
v_{k+2}d_1v_{k+3}\dots v_{2k+1}d_kv_0
\]
such that \( \{ c_1, \dots, c_k \} = \{ d_1, \dots, d_k \} \).

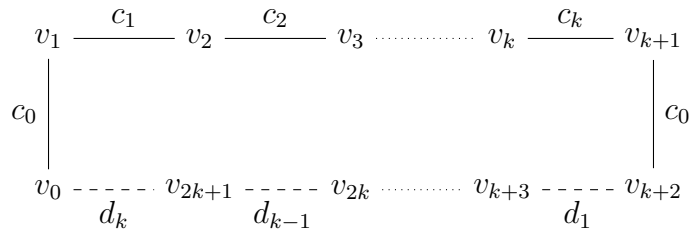
\begin{center}
\begin{tikzpicture}
    % 上のパス
    \node (A1) at (0, 0) {$v_1$};
    \node (A2) at (2, 0) {$v_2$};
    \node (A3) at (4, 0) {$v_3$};
    \node (Ak1) at (6, 0) {$v_k$};
    \node (Ak) at (8, 0) {$v_{k+1}$};

    % 下のパス
    \node (B1) at (0, -2) {$v_0$};
    \node (B2) at (2, -2) {$v_{2k+1}$};
    \node (B3) at (4, -2) {$v_{2k}$};
    \node (Bk1) at (6, -2) {$v_{k+3}$};
    \node (Bk) at (8, -2) {$v_{k+2}$};

    % 上段のエッジ
    \draw (A1) -- node[above] {$c_1$} (A2);
    \draw (A2) -- node[above] {$c_2$} (A3);
    \draw[dotted] (A3) -- (Ak1);
    \draw (Ak1) -- node[above] {$c_k$} (Ak);

    % 下段のエッジ
    \draw[dashed] (B1) -- node[below] {$d_k$} (B2);
    \draw[dashed] (B2) -- node[below] {$d_{k-1}$} (B3);
    \draw[dotted] (B3) -- (Bk1);
    \draw[dashed] (Bk1) -- node[below] {$d_1$} (Bk);

    % 縦のエッジ
    \draw (A1) -- node[left] {$c_0$} (B1);
    \draw (Ak) -- node[right] {$c_0$} (Bk);
\end{tikzpicture}
\end{center}
\end{theorem}

\begin{example} \label{glmndef}
The general linear Lie superalgebra \( \mathfrak{gl}(m|n) \) is defined as the Lie superalgebra spanned by all \( E_{ij} \) with \( 1 \leq i, j \leq m+n \), under the supercommutator:
\[
[E_{ij}, E_{kl}] = \delta_{jk} E_{il} - (-1)^{|E_{ij}||E_{kl}|} \delta_{il} E_{kj},
\]
where \( |E_{ij}| = \overline{0} \) if \( E_{ij} \) acts within \( V_{\overline{0}} \) or \( V_{\overline{1}} \) (even), and \( |E_{ij}| = \overline{1} \) if it maps between \( V_{\overline{0}} \) and \( V_{\overline{1}} \) (odd).

The Cartan subalgebra \( \mathfrak{h} \) is given by \( \mathfrak{h} = \bigoplus k E_{ii} \).

Let \( E_{ii} \) be associated with dual basis elements \( \varepsilon_i \) for \( 1 \leq i \leq m+n \). Then we have \( \mathfrak{g}_{\varepsilon_i - \varepsilon_j} = k E_{ij} \).

The bilinear form \( (\, , \,) \) is computed as follows:
\[
(\varepsilon_i, \varepsilon_j) = 
\begin{cases}
1 & \text{if } i = j \leq m, \\
-1 & \text{if } i = j \geq m+1, \\
0 & \text{if } i \neq j.
\end{cases}
\]

Define \( \delta_i = \varepsilon_{m+i} \) for \( 1 \leq i \leq n \). The sets of roots are as follows:
\[
\Delta_{\overline{0}} = \{ \varepsilon_i - \varepsilon_j, \delta_i - \delta_j \mid i \neq j \},
\]
\[
\Delta_{\overline{1}} = \{ \varepsilon_i - \delta_j \mid 1 \leq i \leq m, \, 1 \leq j \leq n \}.
\]

For the even part \( \mathfrak{g}_{\overline{0}} = \mathfrak{gl}(m) \oplus \mathfrak{gl}(n) \), we fix the standard Borel subalgebra \( \mathfrak{b}_{\overline{0}} \) as:
\[
\mathfrak{b}_{\overline{0}} = \bigoplus_{1 \leq i \leq j \leq m} k E_{ij} \oplus \bigoplus_{m+1 \leq i \leq j \leq n} k E_{ij}.
\]

We assume that the Borel subalgebras we consider all contain \( \mathfrak{b}_{\overline{0}} \).
 
The Borel subalgebras of \( \mathfrak{gl}(m|n) \) correspond naturally in a one-to-one manner to the elements of \( P_{m \times n} \) (in \cref{g.ydi}).

The color set of \( OR(\mathfrak{gl}(m|n)) \) can be identified with
\[
    \Delta_\otimes^{\emptyset+} = \{ \varepsilon_i - \delta_j \mid 1 \leq i \leq m, \, 1 \leq j \leq n \}.
\]
This set can naturally be identified with the collection of boxes in an \( m \times n \) rectangle.

More precisely, with this representation, by assigning the color \( \varepsilon_i - \delta_j \) to the box at coordinates \( (j, m+1-i) \), we obtain a natural edge-colored graph isomorphism:
\[
    OR(\mathfrak{gl}(m|n)) \cong L(m, n) \quad \text{(see \cref{3.1lmn}).}
\]
We also note that we obtain a natural edge-colored graph isomorphism:
\[
    OR(\mathfrak{gl}(1|1)^{\oplus n}) \cong Q_n \quad \text{(see \cref{g.qn}).}
\]

\end{example}

\begin{example}
Let \( \mathfrak{g} = \mathfrak{gl}(3|2) \). We compute 
\[
D_{\varepsilon_1 - \delta_1} = \{ \varepsilon_3 - \delta_1, \varepsilon_2 - \delta_1, \varepsilon_1 - \delta_2 \}.
\]
Therefore, the edge-colored graph \( OR( \mathfrak{gl}(3|2), \varepsilon_1 - \delta_1) \) is isomorphic to \( L(3,2) / D \), as described in \cref{L32}.
\end{example}

\begin{example}\label{2}
According to \cite{bonfert2024weyl,andruskiewitsch2017finite,musson2012lie}, we have the following isomorphism of edge-colored graphs:
\[
\begin{array}{llcl}
    OR(\mathfrak{osp}(2m+1|2n)) &\cong L(m, n)    && (\text{see } \cref{3.1lmn}), \\
    OR(\mathfrak{osp}(2m|2n))   &\cong LD(m, n)   && (\text{see } \cref{Dmn}), \\
    OR(\mathfrak{g})            &:\, \text{tree if } \mathfrak{g} \text{ is exceptional type}.
\end{array}
\]
\end{example}

\subsection{Verma modules}

In this subsection, we summarize some basic results—likely well-known among experts—that play an important role in later chapters and follow easily from \cref{2.3.pbw}, for the sake of convenience.

Let s$\mathcal{W}$ denote the category of weight $\mathfrak{g}$-modules, i.e., $\mathfrak{h}$-semisimple $\mathfrak{g}$-modules.
Similarly, by replacing \( \mathfrak{g} \) with \( \mathfrak{g}_{\overline{0}} \), we define \( s\mathcal{W}_{\overline{0}} \) as a full subcategory of \( \mathfrak{g}_{\overline{0}} \)-sMod. 

If \( \mathcal{W}_{\overline{0}} \) is the usual category of weight modules for \( \mathfrak{g}_{\overline{0}} \), then it is important to note that \( s\mathcal{W}_{\overline{0}} = \mathcal{W}_{\overline{0}} \oplus \Pi \mathcal{W}_{\overline{0}} \), where \( \Pi \mathcal{W}_{\overline{0}} \) denotes the parity shift of \( \mathcal{W}_{\overline{0}} \).

\begin{lemma}[See also \cite{brundan2014representations} Lemma 2.2,  \cite{chen2020primitive} Proposition 2.2.3]
We can choose \( p \in \operatorname{Map}(\mathfrak{h}^*, \mathbb{Z}/2\mathbb{Z}) \) such that
\[
\mathcal{W} := \{ M \in s\mathcal{W} \mid \deg M_{\lambda} = p(\lambda) \text{ for } \lambda \in \mathfrak{h}^*, M_\lambda \neq 0 \}
\]
forms a Serre subcategory, \( \mathcal{W}_{\overline{0}} \subseteq \mathcal{W} \) and \( s\mathcal{W} = \mathcal{W} \oplus \Pi \mathcal{W} \).

\end{lemma}

From this point onward, we  ignore the parity and work within \( \mathcal{W} \).

\begin{definition}
Let \( k_{\lambda}^{\mathfrak{b}} \) be the one-dimensional \( \mathfrak{b} \)-module corresponding to \( \lambda \in \mathfrak{h}^* \). The \( \mathfrak{b} \)-Verma module with highest weight \( \lambda \) is defined by
\[
M^{\mathfrak{b}}(\lambda) = U(\mathfrak{g}) \otimes_{U(\mathfrak{b})} k_{\lambda}^{\mathfrak{b}}.
\]
(Here, the parity of \( k_{\lambda}^{\mathfrak{b}} \) is chosen so that \( M^{\mathfrak{b}}(\lambda) \in \mathcal{W} \).) Its simple top is denoted by \( L^{\mathfrak{b}}(\lambda) \).

Similarly, for the even part \( \mathfrak{g}_{\overline{0}} \), the corresponding Verma module and simple module are denoted by \( M_{\overline{0}}(\lambda) \) and \( L_{\overline{0}}(\lambda) \), respectively.

For a module \( M \) in the category \( \mathcal{W} \), the character \( \operatorname{ch} M \) is a formal sum that encodes the dimensions of the weight spaces of \( M \). Specifically, if \( M \) has a weight space decomposition \( M = \bigoplus_{\lambda} M_{\lambda} \), then
\(
\operatorname{ch} M := \sum_{\lambda} \dim M_{\lambda} e^{\lambda},
\)
where \( \lambda \) runs over the weights of \( M \) and \( e^{\lambda} \) denotes the formal exponential corresponding to the weight \( \lambda \).
    
\end{definition}

From now on, for a Borel subalgebra \( \mathfrak{b} \), we fix a total order on the basis and write \( \Pi^{\mathfrak{b}} = \{ \alpha_1^{\mathfrak{b}}, \dots, \alpha_\theta^{\mathfrak{b}} \} \). Odd reflections \( r_{\alpha_i^\mathfrak{b}} \) are concisely written as \( r_i^\mathfrak{b} \) or \( r_i \). Correspondingly, the total order of the basis is uniquely determined for all Borel subalgebras, as explained in \cite{bonfert2024weyl}.

\begin{definition}
For \( \alpha \in \Delta^{\mathfrak{b+}}\), we take a basis element of \( \mathfrak{g}_{\alpha} \) and denote it by \( e_{\alpha}^{\mathfrak{b}} \). Similarly, we take a basis element of \( \mathfrak{g}_{-\alpha} \) and denote it by \( f_{\alpha}^{\mathfrak{b}} \).

In particular, when \( \alpha \) is a \( \mathfrak{b} \)-simple root \( \alpha_i^{\mathfrak{b}} \), we denote \( e_{\alpha}^{\mathfrak{b}} \) by \( e_i^{\mathfrak{b}} \) and \( f_{\alpha}^{\mathfrak{b}} \) by \( f_i^{\mathfrak{b}} \).

These are unique up to scalar multiplication, so from now on, all equations involving them are understood to hold up to scalar.

Note that \( f_i^{\mathfrak{b}} = e_i^{r_i \mathfrak{b}} \) and \( e_i^{\mathfrak{b}} = f_i^{r_i \mathfrak{b}} \).

\end{definition}

\begin{lemma}[\cite{cheng2015brundan}, Lemma 6.1]
For \( \alpha_i^\mathfrak{b} \in \Pi_\otimes^{\mathfrak{b}} \) and \( \lambda \in \mathfrak{h}^* \), we have 
\[
\operatorname{ch} M^{\mathfrak{b}}(\lambda) = \operatorname{ch} M^{r_i \mathfrak{b}}(\lambda - \alpha_i^\mathfrak{b}).
\]
\end{lemma}

\begin{definition}[Weyl vectors]
Define the following vectors in \( \mathfrak{h}^* \):
\[
\rho_{\overline{0}} :=  \frac{1}{2} \sum_{\beta \in \Delta_{\overline{0}}^+} \beta, \quad
\rho_{\overline{1}}^\mathfrak{b} :=  \frac{1}{2} \sum_{\gamma \in \Delta_{\overline{1}}^{\mathfrak{b} +}} \gamma, \quad
\rho^\mathfrak{b} := \rho_{\overline{0}} - \rho_{\overline{1}}^\mathfrak{b}.
\]

It is worth noting that \( \rho^{r_i \mathfrak{b}} = \rho^{\mathfrak{b}} + \alpha_i^{\mathfrak{b}} \).
\end{definition}

\begin{proposition}[\cite{cheng2012dualities,musson2012lie} Corollary 8.5.4]

\label{5.2.Weylvec}
If \( \alpha_i^{\mathfrak{b}} \in \Pi^{\mathfrak{b}} \), then we have \( (\rho^{\mathfrak{b}}, \alpha_i^{\mathfrak{b}}) = \frac{1}{2} (\alpha_i^{\mathfrak{b}}, \alpha_i^{\mathfrak{b}}) \). In particular, if \( \alpha_i^{\mathfrak{b}} \) is isotropic, then we have \( (\rho^{\mathfrak{b}}, \alpha_i^{\mathfrak{b}}) = 0 \).

\end{proposition}

\begin{corollary}[see also \cite{vay2023linkage}]\label{5.2.ch_eq}
For a pair \( \mathfrak{b}, \mathfrak{b}' \in \mathfrak{B} \) and \( \lambda \in \mathfrak{h}^* \), the following statements hold:
\begin{enumerate}
    \item \( \operatorname{ch} M^{\mathfrak{b}}(\lambda - \rho^{\mathfrak{b}}) = \operatorname{ch} M^{\mathfrak{b}'}(\lambda - \rho^{\mathfrak{b}'}) \);
    \item \( \dim M^{\mathfrak{b}'}(\lambda - \rho^{\mathfrak{b}'})_{\lambda - \rho^{\mathfrak{b}}} = 1 \);
    \item \( \operatorname{ch} M^{\mathfrak{b}}(\lambda) = \operatorname{ch} M^{\mathfrak{b}'}(\lambda) \iff \lambda + \rho^{\mathfrak{b}} = \lambda' + \rho^{\mathfrak{b}'} \);
    \item \( \left[ M^{\mathfrak{b}}(\lambda - \rho^{\mathfrak{b}}) : L^{\mathfrak{b}'}(\lambda - \rho^{\mathfrak{b}'}) \right] = 1 \);
    \item \( \dim \operatorname{Hom}(M^{\mathfrak{b}}(\lambda - \rho^{\mathfrak{b}}), M^{\mathfrak{b}'}(\lambda - \rho^{\mathfrak{b}'})) = 1 \);
    \item If \( M^{\mathfrak{b}}(\lambda - \rho^{\mathfrak{b}}) \not\cong M^{\mathfrak{b}'}(\lambda - \rho^{\mathfrak{b}'}) \), then \( M^{\mathfrak{b}}(\lambda - \rho^{\mathfrak{b}}) \) is not isomorphic to any \( \mathfrak{b}' \)-Verma module.
\end{enumerate}
\end{corollary}

% Definition
\begin{definition}
Let \( \mathfrak{b}, \mathfrak{b}' \in B(\mathfrak{g}) \) and \( \lambda \in \mathfrak{h}^* \).
Let \( v_{\lambda}^{\mathfrak{b}} \) be the highest weight vector of \( M^{\mathfrak{b}}(\lambda) \).

We denote a nonzero homomorphism from \( M^{\mathfrak{b}}(\lambda - \rho^{\mathfrak{b}}) \) to \( M^{\mathfrak{b}'}(\lambda - \rho^{\mathfrak{b}'}) \) by \( \psi_{\lambda}^{\mathfrak{b}, \mathfrak{b}'}. \)
Note that \( \psi_{\lambda}^{\mathfrak{b}, \mathfrak{b}} = \operatorname{id}_{M^{\mathfrak{b}}(\lambda - \rho^{\mathfrak{b}})} \).

\end{definition}

\begin{proposition}\label{prop:verma_homomorphism}
    Let \( \alpha_{i}^{\mathfrak{b}} \in \Pi_{\otimes}^{\mathfrak{b}} \), \( \lambda \in \mathfrak{h}^* \), and \( \mathfrak{b} \in B(\mathfrak{g}) \). We have:
    \begin{enumerate}
        \item 
        \(
        \psi_{\lambda}^{r_i \mathfrak{b}, \mathfrak{b}} 
        \left( v_{\lambda - \rho^{r_i \mathfrak{b}}}^{r_i \mathfrak{b}} \right) 
        = f_i^{\mathfrak{b}} v_{\lambda - \rho^{\mathfrak{b}}}^{\mathfrak{b}} 
        = e_i^{r_i \mathfrak{b}} v_{\lambda - \rho^{\mathfrak{b}}}^{\mathfrak{b}},\)
    
       \( \psi_{\lambda}^{\mathfrak{b}, r_i \mathfrak{b}} 
        \left( v_{\lambda - \rho^{\mathfrak{b}}}^{\mathfrak{b}} \right) 
        = f_i^{r_i \mathfrak{b}} v_{\lambda - \rho^{r_i \mathfrak{b}}}^{r_i \mathfrak{b}} 
        = e_i^{\mathfrak{b}} v_{\lambda - \rho^{r_i \mathfrak{b}}}^{r_i \mathfrak{b}}.
        \)

        Note that the homomorphism from the Verma module is completely determined by the image of the highest weight vector.

        \item If \( (\lambda, \alpha_{i}^{\mathfrak{b}}) \neq 0 \), then we have \( \psi_{\lambda}^{r_i \mathfrak{b}, \mathfrak{b}} \) and \( \psi_{\lambda}^{\mathfrak{b}, r_i \mathfrak{b}} \) are isomorphisms.  
        Therefore, 
        \[
        L^{\mathfrak{b}}(\lambda - \rho^{\mathfrak{b}}) \cong L^{r_i \mathfrak{b}}(\lambda - \rho^{r_i \mathfrak{b}})
        \quad \text{and} \quad
        P^{\mathfrak{b}}(\lambda - \rho^{\mathfrak{b}}) \cong P^{r_i \mathfrak{b}}(\lambda - \rho^{r_i \mathfrak{b}}).
        \]

        \item If \( (\lambda, \alpha_{i}^{\mathfrak{b}}) = 0 \), then we have
        \(
        \psi_{\lambda}^{r_i \mathfrak{b}, \mathfrak{b}} \circ \psi_{\lambda}^{\mathfrak{b}, r_i \mathfrak{b}} = \psi_{\lambda}^{\mathfrak{b}, r_i \mathfrak{b}} \circ \psi_{\lambda}^{r_i \mathfrak{b}, \mathfrak{b}} = 0.
        \)
        Furthermore, we have
        \(
        L^{\mathfrak{b}}(\lambda) \cong L^{r_i \mathfrak{b}}(\lambda)
        \quad \text{and} \quad
        P^{\mathfrak{b}}(\lambda) \cong P^{r_i \mathfrak{b}}(\lambda).
        \)
    \end{enumerate}
\end{proposition}

% Corollary 5.2
\begin{corollary}\label{5.2RB}
The vertex set of \( OR(\mathfrak{g}, \lambda) \) can be identified with the isomorphism classes of \({ \{ M^{\mathfrak{b}}(\lambda - \rho^{\mathfrak{b}}) \} }_{\mathfrak{b} \in B(\mathfrak{g})}\).
Precisely, let \( [\mathfrak{b}] \) denote the equivalence class of \( \mathfrak{b} \in B(\mathfrak{g}) \) in \( OR(\mathfrak{g}, \lambda) \). Then, we have:
\[
M^{\mathfrak{b}}(\lambda - \rho^{\mathfrak{b}}) \cong M^{\mathfrak{b}'}(\lambda - \rho^{\mathfrak{b}'}) \iff \mathfrak{b}' \in [\mathfrak{b}].
\]
\end{corollary}

\begin{proof}
The homomorphisms corresponding to edges of colors in \( D_{\lambda} \) are isomorphisms by \cref{prop:verma_homomorphism}, so any homomorphism corresponding to a walk consisting only of edges of colors in \( D_{\lambda} \) is also an isomorphism. Conversely, any homomorphism corresponding to an edge of a color not in \( D_{\lambda} \) is necessarily non-isomorphic by \cref{prop:verma_homomorphism}, and thus any homomorphism corresponding to a walk that includes edges of colors not in \( D_{\lambda} \) is also non-isomorphic.
\end{proof}

\begin{definition}
Let \( \lambda \in \mathfrak{h}^* \) and \( \mathfrak{b} \in B(\mathfrak{g}) \). 

\begin{itemize}
    \item \( \lambda \) is called \emph{\( \mathfrak{b} \)-typical} if there exists no \( \alpha \in \Delta_\otimes \) such that \( (\alpha, \lambda + \rho^{\mathfrak{b}}) = 0 \).
    \item Otherwise, \( \lambda \) is called \emph{\( \mathfrak{b} \)-atypical}.
\end{itemize}

\end{definition}

\begin{proposition}\label{RBtriv}
The following are equivalent for \( \lambda \in \mathfrak{h}^* \):
\begin{enumerate}
    \item \( (\lambda, \beta) \neq 0 \text{ for all } \beta \in \Delta_{\otimes} \setminus \Delta_{\otimes}^{\operatorname{pure+}} \);
    \item \( M^{\mathfrak{b}}(\lambda - \rho^{\mathfrak{b}}) \cong M^{\mathfrak{b}'}(\lambda - \rho^{\mathfrak{b}'}) \text{ for all } \mathfrak{b}, \mathfrak{b}' \in B(\mathfrak{g}) \);
    \item \( M^{\mathfrak{b}}(\lambda - \rho^{\mathfrak{b}}) \in \bigcap_{\mathfrak{b} \in B(\mathfrak{g})} \mathcal{F}\Delta^{\mathfrak{b}} \);
    \item The vertex set of \( OR(\mathfrak{g}, \lambda) \) consists of a single point;
    \item \( L^{\mathfrak{b}}(\lambda - \rho^{\mathfrak{b}}) \cong L^{\mathfrak{b}'}(\lambda - \rho^{\mathfrak{b}'}) \text{ for all } \mathfrak{b}, \mathfrak{b}' \in B(\mathfrak{g}) \).
\end{enumerate}

\end{proposition}

\begin{proof}
The equivalence of (1) and (4) follows directly from \cref{5.2RB}. The equivalence of (1) and (2) is a consequence of \cref{prop:verma_homomorphism}. The implication (2) \(\Rightarrow\) (3) is trivial by definition, while (3) \(\Rightarrow\) (2) follows from \cref{5.2.ch_eq}. The implication (2) \(\Rightarrow\) (5) is immediate since Verma modules have simple tops. Finally, (5) \(\Rightarrow\) (1) is a direct consequence of \cref{prop:verma_homomorphism}.
\end{proof}

\begin{remark}
The situation described in \cref{RBtriv} can be seen, in a certain sense, as the opposite of being “reflection complete” in the sense of \cite{musson2023weyl}.
\end{remark}

\begin{example}\label{typeIRB}
    Suppose \( \mathfrak{g} \) is of type I. Then \( \Delta_{\otimes}^{\operatorname{pure+}} = \emptyset \), and we have the following equivalence:
    \[
    \lambda - \rho^{\mathfrak{b}} \text{ is } \mathfrak{b}\text{-typical} \quad \Longleftrightarrow \quad \text{The vertex set of } OR(\mathfrak{g}, \lambda) \text{ consists of a single point.}
    \]
\end{example}

\subsection{Semibricks}\label{subsec_semib}

Here, following \cite{enomoto2021schur}, we review the basics of semibricks. 

\begin{definition}
Let \(\mathcal{E}\) be an additive category.
\begin{itemize}
    \item An object \( S \) in \( \mathcal{E} \) is called a \emph{brick} if \( \operatorname{End}_{\mathcal{E}}(S) = k \).
   \item A collection \( S \) of all bricks in \( \mathcal{E} \) is called a \emph{semibrick} if \( \operatorname{Hom}_\mathcal{E}(S, T) = 0 \) holds for every pairwise nonisomorphic elements \( S \) and \( T \) in \( \mathcal{E} \).
\end{itemize}
\end{definition}

From here on, let \( \mathcal{A} \) denote an abelian category.

\begin{itemize}
    \item We denote by \( \operatorname{sim} \mathcal{A} \) the collection of isomorphism classes of simple objects in \( \mathcal{A} \).
    \item For a collection \( \mathcal{C} \) of objects in \( \mathcal{A} \), we denote by \( \operatorname{Filt} \mathcal{C} \) the subcategory of \( \mathcal{A} \) consisting of objects \( X \in \mathcal{A} \) such that there exists a chain \( 0 = X_0 \subset X_1 \subset \cdots \subset X_l = X \) of submodules with \( X_i / X_{i-1} \in \mathcal{C} \) for each \( i \). We call this the filtration closure of \( \mathcal{C} \) in \( \mathcal{A} \).
\end{itemize}

We say that an abelian category \( \mathcal{A} \) is \emph{length} if \( \operatorname{sim} \mathcal{A} \) is a set and \( \mathcal{A} = \operatorname{Filt}(\operatorname{sim} \mathcal{A}) \) holds.

The following is a classical result by Ringel.

\begin{theorem}[\cite{Ringel1976RepresentationsOK} 2.1, \cite{enomoto2021schur} 2.5]\label{semibrick}
Let \( \mathcal{A} \) be an abelian category. Then the assignments \( S \mapsto \operatorname{Filt} S \) and \( \mathcal{W} \mapsto \operatorname{sim} \mathcal{W} \) establish a one-to-one correspondence between the following two classes.
\begin{itemize}
    \item[(1)] The class of semibricks \( S \) in \( \mathcal{A} \).
    \item[(2)] The class of extension closed length subcategories \( \mathcal{W} \) in \( \mathcal{A} \).
\end{itemize}
\end{theorem} 

For example, a Serre subcategory is the only extension closed abelian subcategory that can be obtained by applying \( \operatorname{Filt} \) to a subclass of \( \operatorname{sim} \mathcal{A} \).

\begin{lemma}\label{h.w.semibrick}
Fix \( \mathfrak{b} \in \mathfrak{B}(\mathfrak{g}) \). Consider the collection of \( \mathfrak{b} \)-highest weight modules \( \{ S(\lambda) \}_{\lambda \in \Lambda} \), where the \( \mathfrak{b} \)-highest weight of \( S(\lambda) \) is \( \lambda \). If for any pair \( \lambda, \lambda' \in \Lambda \), \( S(\lambda)_{\lambda'} = 0 \), then \( \{ S(\lambda) \}_{\lambda \in \Lambda} \) is a semibrick.
\end{lemma}

\begin{proof}
Since endomorphisms of the \( \mathfrak{b} \)-highest weight module send the \( \mathfrak{b} \)-highest weight vector to a \( \mathfrak{b} \)-highest weight vector (or 0), it is a brick. If \( S(\lambda)_{\lambda'} = 0 \), then by the property that homomorphisms preserve weights and any homomorphism sending a \( \mathfrak{b} \)-highest weight vector to zero is zero, we have \( \operatorname{Hom}(S(\lambda'), S(\lambda)) = 0 \). This proves the hom-orthogonality.
\end{proof}

\begin{example}\label{M3}

Let \(\alpha\) be an isotropic \(\mathfrak{b}\)-simple root orthogonal to the weight \(\lambda\), the following forms a semibrick :  
\[
X(\mathfrak{b}, \lambda, \alpha) := \{\operatorname{Im} \psi_{\lambda+\rho^{\mathfrak{b}}}^{r_{\alpha} \mathfrak{b}, \mathfrak{b}},\operatorname{Im} \psi_{\lambda+\rho^{\mathfrak{b}}}^{ \mathfrak{b},r_{\alpha} \mathfrak{b}}  \}.
\]

Let \( \operatorname{Filt}X(\mathfrak{b}, \lambda, \alpha)  \) denote the filtration closure of \( X(\mathfrak{b}, \lambda, \alpha)  \) in the category \( \mathcal{W} \).

The corresponding projective covers of \(  \operatorname{Im} \psi_{\lambda+\rho^{\mathfrak{b}}}^{r_{\alpha} \mathfrak{b}, \mathfrak{b}},\operatorname{Im} \psi_{\lambda+\rho^{\mathfrak{b}}}^{ \mathfrak{b},r_{\alpha} \mathfrak{b}}\) in \( \operatorname{Filt} Z(\mathfrak{b}, \lambda, \alpha^{ \mathfrak{b}}_1,\alpha^{ \mathfrak{b}}_3) \) are 
\[
M^{r_{\alpha}\mathfrak{b}}(\lambda - \alpha), \quad M^{\mathfrak{b}}(\lambda )
\]
respectively (these are Ext-orthogonal to these bricks by \cref{7.2.kiso}).

By Morita theory, it can be verified that \( \operatorname{Filt}X(\mathfrak{b}, \lambda, \alpha)  \) is equivalent to the category of finite-dimensional modules over a finite-dimensional algebra defined by the following quiver and relations.

\begin{center}
\begin{tikzpicture}[->, thick, shorten >=1pt, node distance=2cm]

% Nodes
\node (A) at (0,0) {1};
\node (B) [right of=A] {2};

% Arrows (bidirectional with curves)
\draw[->] (A) to[bend left=20] node[above] {$a$} (B);
\draw[->] (B) to[bend left=20] node[below] {$b$} (A);

\end{tikzpicture}
\end{center}

The relations are:
\[
ab = ba  = 0.
\]

This finite-dimensional algebra is known as the preprojective algebra of type \( A_2 \).
\end{example}

\begin{remark}
Unless otherwise stated, we take $\mathrm{Filt}$ to be computed in the category $\mathcal{W}$. This assumption is essential, as $\mathcal{W}$ is not an extension-closed subcategory of the category of $\mathfrak{g}$-supermodules (see,  Exercise 3.1 in \cite{humphreys2021representations}).
On the other hand, category $\mathcal{O}$ is a Serre subcategory of category $\mathcal{W}$, and more strongly, it is extension-full in  category $\mathcal{W}$\cite{coulembier2015homological,delorme1980extensions}.
\end{remark}

\section{Semibricks realizing the infinite zigzag algebra} \label{sec:semibricks}
We retain the setting of   \cref{sec:changing_Borel}.
\subsection{Adjusted Borel subalgebras}

% Definition of quasi Borel subalgebra and Verma modules
\begin{definition}
Let \( \Delta^{\mathfrak{a}} \subseteq \Delta^{\mathfrak{b}}_{\overline{1}} \), and define 
\[
\mathfrak{a} := \bigoplus_{\alpha \in \Delta^{+}_{\overline{0}}\sqcup \Delta^{\mathfrak{a}}} \mathfrak{g}_{\alpha}.
\]

Let \( \mathfrak{a} = \mathfrak{h} \oplus \mathfrak{n}_{\mathfrak{a}} \).

For \( \lambda \in \mathfrak{h}^{*} \), we call the above \( \mathfrak{a} \) a \(\lambda\)-adjusted Borel subalgebra if the following defines a well-defined one-dimensional representation of \( \mathfrak{a} \):
\[
h v_{\lambda}^{\mathfrak{a}} = \lambda(h) v_{\lambda}^{\mathfrak{a}} \quad \text{for any } h \in \mathfrak{h}, \quad \text{and} \quad n_{\mathfrak{a}} v_{\lambda}^{\mathfrak{a}} = 0.
\]
We denote this representation by 
\(
k_{\lambda}^{\mathfrak{a}} = k v_{\lambda}^{\mathfrak{a}}.
\)

\end{definition}

\begin{definition}

    Let \( \mathfrak{a} \) denote a \(\lambda\)-adjusted Borel subalgebra. For \( \lambda \in \mathfrak{h}^{*} \), define
    \[
        M^{\mathfrak{a}}(\lambda) := U(\mathfrak{g}) \otimes_{U(\mathfrak{a})} k_{\lambda}^{\mathfrak{a}}.
    \]
    It is easy to see that \( M^{\mathfrak{a}}(\lambda) \) belongs to \( \mathcal{W} \).
We will simply write \( x v_{\lambda}^{\mathfrak{a}} \) instead of \( x \otimes v_{\lambda}^{\mathfrak{a}} \).
\end{definition}

% Examples
\begin{example}
\begin{enumerate}
\item If \( [\mathfrak{n}_{\mathfrak{a}}, \mathfrak{n}_{\mathfrak{a}}] \cap \mathfrak{h} = 0 \), then \( \mathfrak{a} \) is a \(\lambda\)-adjusted Borel subalgebra for any \( \lambda \in \mathfrak{h}^* \).
  \item If \( \mathfrak{a} = \mathfrak{b}_{\overline{0}} \), then \( M^{\mathfrak{a}}(\lambda) = \operatorname{Ind}_{\mathfrak{g}_{\overline{0}}}^{\mathfrak{g}} M_{\overline{0}}(\lambda) \), where \( \Delta^{\mathfrak{a}} = \varnothing \).
    \item If \( \mathfrak{a} = \mathfrak{b} \), then \( M^{\mathfrak{a}}(\lambda) = M^{\mathfrak{b}}(\lambda) \), where \( \Delta^{\mathfrak{a}} = \Delta_{\overline{1}}^{\mathfrak{b}^{+}} \).
    \item The intersection of \(\lambda\)-adjusted Borel subalgebras is a \(\lambda\)-adjusted Borel subalgebra.
    \item Let \( \mathfrak{g} = \mathfrak{gl}(1|1) \) and let \( \alpha \) be an odd root. Then, \( \mathfrak{g} \) itself is a \( \lambda \)-adjusted Borel subalgebra if and only if \( \lambda \in k\alpha \).
\end{enumerate}
\end{example}

% Proposition for character formula
Below, we summarize fundamental results derived from the PBW theorem and Frobenius reciprocity.

\begin{proposition}\label{ch_neo}
    
   Let \( \mathfrak{a} \) denote a \(\lambda\)-adjusted Borel subalgebra. The character of \( M^{\mathfrak{a}}(\lambda) \) is given by
    \[
        \operatorname{ch} M^{\mathfrak{a}}(\lambda) = e^{\lambda} \frac{\prod_{\beta \in \Delta_{\overline{1}} \setminus \Delta^{\mathfrak{a}}} (1 + e^{\beta})}{\prod_{\gamma \in \Delta_{\overline{0}}^{+}} (1 - e^{-\gamma})}.
    \]
    In particular,  we have \( \lambda = \lambda' \iff M^{\mathfrak{a}}(\lambda) \simeq M^{\mathfrak{a}}(\lambda') \).
\end{proposition}

% Proposition using Frobenius reciprocity
\begin{proposition}\label{fuhen}
    Let \( \mathfrak{a} \) denote a \(\lambda\)-adjusted Borel subalgebra. Let \( M \in \mathcal{W} \) and \( v \in M_\lambda \) such that \( \mathfrak{a} v = 0 \). Then there exists a unique homomorphism 
    \[
M^{\mathfrak{a}}(\lambda) \to M \quad \text{sending} \quad v_{\lambda}^{\mathfrak{a}} \mapsto v.
\]

In particular, if \( \mathfrak{a}' \subseteq \mathfrak{a} \) is an inclusion of \(\lambda\)-adjusted Borel subalgebras, then \( M^{\mathfrak{a}}(\lambda) \) is a quotient of \( M^{\mathfrak{a}'}(\lambda) \).
  
\end{proposition}

\begin{example}
Let \( \mathfrak{g} = \mathfrak{gl}(2|1) \).
We have,  
\[
M_{\bar{0}}(0) = P_{\bar{0}}(0), \text{ so }
\operatorname{Ind}_{\mathfrak{g}_{\bar{0}}}^{\mathfrak{g}} P_{\bar{0}}(0) = \operatorname{Ind}_{\mathfrak{g}_{\bar{0}}}^{\mathfrak{g}} M_{\bar{0}}(0).
\]
which has a Verma flag whose constituent is:
\[
\left\{ M^{\emptyset}(0), M^{\emptyset}(\varepsilon_1 - \delta_1), M^{\emptyset}(\varepsilon_2 - \delta_1), M^{\emptyset}(\varepsilon_1 + \varepsilon_2 - 2\delta_1) \right\}.
\]

Among these four modules, only \( M^{\emptyset}(\varepsilon_1 - \delta_1) \) belongs to a typical block, and
\[
M^{\emptyset}(\varepsilon_1 - \delta_1) = P^{\emptyset}(\varepsilon_1 - \delta_1).
\]
We can write \( P^{\emptyset}(0) = P^{(1)}(0) = P^{(1^2)}(0) =: P(0) \), and thus
\[
\operatorname{Ind}_{\mathfrak{g}_{\bar{0}}}^{\mathfrak{g}} P_{\bar{0}}(0) = P(0) \oplus P^{\emptyset}(\varepsilon_1 - \delta_1).
\]

There exists an exact sequence:
\[
0 \longrightarrow M^{\emptyset}(\varepsilon_1 + \varepsilon_2 - 2\delta_1) \longrightarrow P(0) \longrightarrow M^{\emptyset \cap (1)}(0) \longrightarrow 0,
\]
and similarly,
\[
0 \longrightarrow M^{(1^2)}(-\varepsilon_1 - \varepsilon_2 + 2\delta_1) \longrightarrow P(0) \longrightarrow M^{(1) \cap (1^2)}(0) \longrightarrow 0.
\]

\[
\xymatrix@C=6pc@R=4pc{
M^{\emptyset}(0) 
    \ar[r] \ar@{<->}[d] &
M^{\emptyset}(\varepsilon_2 - \delta_1) 
    \ar[r] \ar@{<->}[d] &
M^{\emptyset}(\varepsilon_1 + \varepsilon_2 - 2\delta_1) 
    \ar@{<->}[d]^{\cong} \\
M^{(1)}(-\varepsilon_2 + \delta_1) 
    \ar@{<->}[d]^{\cong} &
M^{(1)}(0) 
    \ar[r]  \ar[l] \ar@{<->}[d] &
M^{(1)}(\varepsilon_1 - \delta_1) 
    \ar@{<->}[d] \\
M^{(1^2)}(-\varepsilon_1 - \varepsilon_2 + 2\delta_1) 
    &
M^{(1^2)}(-\varepsilon_1 + \delta_1) 
    \ar[l] &
M^{(1^2)}(0)  \ar[l] 
}
\]
\begin{small}
Arrows are drawn between the Verma modules whose highest weight is \( 0 \) and those with the same character, indicating the existence of a nonzero homomorphism between them.
\end{small}
\end{example}

% Introduction and Lemma
Hereafter, when we write \( \operatorname{Ext}^1 \), we mean \( \operatorname{Ext}^1_{\mathcal{O}} \) unless otherwise specified.
The following Lemma is a generalization of \cite[Propsition 3.1]{humphreys2021representations} and serves as a fundamental result in this subsection.

\begin{lemma}\label{7.2.kiso}
    Let \( M \in \mathcal{W} \), and let \( \mathfrak{a} \) be a \( \lambda \)-adjusted Borel subalgebra. If \( M_{\lambda + \beta} = 0 \) for all \( \beta \in \Delta^{+}_{\overline{0}} \sqcup \Delta^{\mathfrak{a}} \), then
    \[
        \operatorname{Ext}^1(M, M^{\mathfrak{a}}(\lambda)) = 0.
    \]
\end{lemma}

\begin{proof}
  Suppose we have a short exact sequence in \(\mathcal{O}\):
\[
    0 \to M \to E \xrightarrow{\psi} M^{\mathfrak{a}}(\lambda) \to 0.
\]
A preimage \( v \) of \( v_{\lambda}^{\mathfrak{a}} \) in \( E \) also satisfies \( \mathfrak{a} v = 0 \) by assumption. Therefore, by the universal property of \( M^{\mathfrak{a}}(\lambda) \), the sequence splits.
\end{proof}

\subsection{Hypercubic decomposition of Verma modules}

\label{subsec_6.3}
In this subsection, we focus on \(\lambda\)-adjusted Borel subalgebras, which are important for our study.

\begin{definition}\label{hq.c_def}
Let \( \mathfrak{b} \in \mathfrak{B}(\mathfrak{g}) \) and \( \lambda \in \mathfrak{h}^* \). Let \( \Pi^{\mathfrak{b}} = \{ \alpha_i \}_{i \in I} \). A subset \( J \subset I \) is called a \( (\mathfrak{b}, \lambda) \)-hypercubic collection if \( \{ \alpha_j \}_{j \in J} \) is a collection of distinct isotropic \( \mathfrak{b} \)-simple roots that are pairwise orthogonal and orthogonal to \( \lambda \).

Let \( \Sigma^{\mathfrak{b}}_ J \) denote the sum of the simple roots indexed by \( J \). Let \( E_J^{\mathfrak{b}} \) (resp. \( F_J^{\mathfrak{b}} \)) denote the product of the \( \mathfrak{b} \)-positive (resp. \( \mathfrak{b} \)-negative) root vectors indexed by \( J \). This product is independent of the order of multiplication.

Simple roots that are orthogonal to each other are unaffected by mutual odd reflections. Therefore, the Borel subalgebra  \( r_J \mathfrak{b} \) determined by applying the sequence of odd reflections indexed by \( J \) to \( \mathfrak{b} \) is well-defined and independent of the order of the indices.

\end{definition}

\begin{lemma}\label{hyq_lemma}
Let \( J \) be a \( (\mathfrak{b}, \lambda) \)-hypercubic collection, and let \( J' \subseteq J \). Then the following hold:
\begin{enumerate}
    \item \( J' \) is a \( (\mathfrak{b}, \lambda) \)-hypercubic collection;
    \item \( J \) is a \( (\mathfrak{b}, \lambda \pm \Sigma^{\mathfrak{b}}_{J'}) \)-hypercubic collection;
    \item \( \mathfrak{b} \cap r_{J'} \mathfrak{b} \cap r_J \mathfrak{b} = \mathfrak{b} \cap r_J \mathfrak{b} \subseteq \mathfrak{g} \) is a \(\lambda\)-adjusted Borel subalgebra;
    \item \( \mathfrak{b} + r_{J'} \mathfrak{b} + r_J \mathfrak{b} = \mathfrak{b} + r_J \mathfrak{b} \subseteq \mathfrak{g} \) is a \(\lambda\)-adjusted Borel subalgebra;
    \item \( E_J^{\mathfrak{b}} = F_J^{r_J \mathfrak{b}} \quad \text{and} \quad \Sigma_J^{\mathfrak{b}} = -\Sigma_J^{r_J \mathfrak{b}} \);
    \item \( \operatorname{ch} M^{\mathfrak{b}}(\lambda) = \operatorname{ch} M^{r_J \mathfrak{b}}(\lambda - \Sigma^{\mathfrak{b}}_J) \);
    \item \( \dim \operatorname{Hom} \big( M^{\mathfrak{b}}(\lambda), M^{r_J \mathfrak{b}}(\lambda - \Sigma^{\mathfrak{b}}_J) \big) = 1 \), and the image of this nonzero homomorphism is isomorphic to \( M^{\mathfrak{b} + r_J \mathfrak{b}}(\lambda) \);
    \item \(
    \dim \operatorname{Hom}(M^{\mathfrak{b}}(\lambda), M^{\mathfrak{b}}(\lambda + \Sigma^{\mathfrak{b}}_J)) = 1,
    \)
    and the image of this nonzero homomorphism is isomorphic to \( M^{\mathfrak{b} + r_J \mathfrak{b}}(\lambda) \).
\end{enumerate}
\end{lemma}
\begin{proof}
Statements (1) through (5) are clear from orthogonality.

For (6), it is well known, as also stated in \cite[Lemma 6.9]{cheng2015brundan}.

For (7), since the characters are equal by (6), it is clear that the dimension of the Hom space is 1.

The image of this homomorphism is a module generated by \( E_J^{\mathfrak{b}} v_{\lambda - \Sigma^{\mathfrak{b}}_J}^{r_J \mathfrak{b}} \), so the claim follows from \cref{2.3.pbw}, \cref{ch_neo} and \cref{fuhen}.

For (8), since we have
\(
\# \bigl\{ D \subset \Delta^{\mathfrak{b}+} \mid \sum_{\beta \in D} \beta = \Sigma_J^{\mathfrak{b}} \bigr\} = 1,\) it follows that \(\dim M^{\mathfrak{b}}(\lambda + \Sigma_J^{\mathfrak{b}})_{\lambda} = 1.
\)
Thus, the Hom space is one-dimensional.

The image of the nonzero homomorphism is a submodule of \( M^{\mathfrak{b}}(\lambda + \Sigma^{\mathfrak{b}}_J) \) generated by \( F_J^{\mathfrak{b}} v_{\lambda + \Sigma^{\mathfrak{b}}_J}^{\mathfrak{b}} \), so the claim follows from \cref{2.3.pbw}, \cref{ch_neo} and \cref{fuhen}.
\end{proof}

\begin{lemma}\label{Mc_indec}
    Let \( J \) be a \( (\mathfrak{b}, \lambda) \)-hypercubic collection. Then, \( M^{\mathfrak{b} \cap r_J \mathfrak{b}}(\lambda) \) has a simple top \( L^{\mathfrak{b}}(\lambda) \). In particular, it is indecomposable.
\end{lemma}

\begin{proof}
The PBW basis for \( M^{\mathfrak{b} \cap r_J \mathfrak{b}}(\lambda) \) can be written in the form of PBW monomials as
\[
F F_{J_1}^{\mathfrak{b}} E_{J_2}^{\mathfrak{b}} v_\lambda^{\mathfrak{b} \cap r_J \mathfrak{b}},
\]
where \( J_1, J_2 \subset J \), and \( F \) is a product of independent \( \mathfrak{b} \)-negative root vectors that cannot be indexed by \( J \).

By orthogonality, it is easy to see that the submodule generated by elements of the form
\[
F_{J_1}^{\mathfrak{b}} E_{J_2}^{\mathfrak{b}} v_\lambda^{\mathfrak{b} \cap r_J \mathfrak{b}}
\]
that are not equal to \( v_\lambda^{\mathfrak{b} \cap r_J \mathfrak{b}} \) does not contain \( v_\lambda^{\mathfrak{b} \cap r_J \mathfrak{b}} \).

By \cref{2.3.pbw}, \cref{ch_neo} and \cref{fuhen}, we have the following isomorphism:
\[
\left. M^{\mathfrak{b} \cap r_J \mathfrak{b}}(\lambda)  \middle/
\sum_{J_1  \subset J, J_2 \subset J, (J_1, J_2) \neq (J, J)} U(\mathfrak{g}) F_{J_1}^{\mathfrak{b}} E_{J_2}^{\mathfrak{b}} v_\lambda^{\mathfrak{b} \cap r_J \mathfrak{b}}
\right.
\cong M^{\mathfrak{b} + r_J \mathfrak{b}}(\lambda).
\]

Thus, \( M^{\mathfrak{b} \cap r_J \mathfrak{b}}(\lambda) \) has a simple top \( L^{\mathfrak{b}}(\lambda) \) and is therefore indecomposable.
\end{proof}

\begin{remark}\label{7.2.2Verma_split}
    Let \( \alpha \in \Pi_\otimes^{\mathfrak{b}} \). Then \( M^{\mathfrak{b} \cap r_{\alpha} {\mathfrak{b}}}(\lambda) \) is indecomposable if and only if \( (\lambda, \alpha) \neq 0 \).

    Indeed, we have the following exact sequences:
    \[
        0 \to M^{r_{\alpha} {\mathfrak{b}}}(\lambda - \alpha) \to M^{\mathfrak{b} \cap r_{\alpha} {\mathfrak{b}}}(\lambda) \to M^{r_{\alpha} {\mathfrak{b}}}(\lambda) \to 0
    \]
    and
    \[
        0 \to M^{\mathfrak{b}}(\lambda + \alpha) \to M^{\mathfrak{b} \cap r_{\alpha} {\mathfrak{b}}}(\lambda) \to M^{\mathfrak{b}}(\lambda) \to 0.
    \]

    If \( (\lambda, \alpha) \neq 0 \), then \( M^{r_{\alpha} {\mathfrak{b}}}(\lambda - \alpha) \simeq M^{\mathfrak{b}}(\lambda) \) and \( M^{r_{\alpha} {\mathfrak{b}}}(\lambda) \simeq M^{\mathfrak{b}}(\lambda + \alpha) \).

Therefore, 
\[
\operatorname{Ext}^{1}(M^{\mathfrak{b}}(\lambda + \alpha), M^{\mathfrak{b}}(\lambda)) 
= \operatorname{Ext}^1(M^{r_{\alpha} \mathfrak{b}}(\lambda), M^{r_{\alpha} \mathfrak{b}}(\lambda - \alpha)) = 0 
\]
by \cref{7.2.kiso}. Thus \( M^{\mathfrak{b} \cap r_{\alpha} {\mathfrak{b}}}(\lambda) \) is not indecomposable.

If \( (\lambda, \alpha) = 0 \), the claim follows as a special case of \cref{Mc_indec}. Thus \( M^{\mathfrak{b} \cap r_{\alpha} {\mathfrak{b}}}(\lambda) \) is indecomposable.

\end{remark}

\subsection{Representation Theory of 
\texorpdfstring{$\mathfrak{gl}(1|1)^{\oplus n}$}{gl(1|1)⊕n}}

In this subsection, let \( \mathfrak{g} = \mathfrak{gl}(1|1)^{\oplus n} \). 
Then, 
\[
\mathfrak{g}_{\overline{0}} = \mathfrak{b}_{\overline{0}} = \mathfrak{h} \cong (\mathfrak{gl}(1) \oplus \mathfrak{gl}(1))^{\oplus n}.
\]

By embedding each \( i \)-th copy of \( \mathfrak{gl}(1|1) \) into \( \mathfrak{gl}(n|n) \) via the assignments:
\[
E_{11} \mapsto E_{ii}, \quad 
E_{12} \mapsto E_{i,n-i+1}, \quad 
E_{21} \mapsto E_{n-i+1,i}, \quad 
E_{22} \mapsto E_{n-i+1,n-i+1}.
\]
and by identifying their Cartan subalgebras. Let \( E_{ii} \in \mathfrak{gl}(n|n) \) be associated with the dual basis elements \( \varepsilon_i \) for \( 1 \leq i \leq 2n \), and define \( \delta_i = \varepsilon_{n+i} \) for \( 1 \leq i \leq n \). We have 
\[
\Delta = \Delta_{\overline{1}} = \Delta_{\otimes} = \{ \pm(\varepsilon_i - \delta_{n-i+1}) \mid i = 1, \dots, n \}.
\]

Let \( \mathfrak{b}_{\mathrm{st}} \) be the standard Borel subalgebra such that 
\[
\Delta^{\mathfrak{b}_{\mathrm{st}}+}=\Delta^{\mathfrak{b}_{\mathrm{st}}+}_{\overline{1}} = \Delta^{\mathfrak{b}_{\mathrm{st}}+}_{\otimes} = \Pi ^{\mathfrak{b}_{\mathrm{st}}+} = \Pi ^{\mathfrak{b}_{\mathrm{st}}+}_{\otimes} = \{ \varepsilon_i - \delta_{n-i+1} \mid i = 1, \dots, n \}.
\]

We also note that we obtain a natural edge-colored graph isomorphism:
\[
    OR(\mathfrak{gl}(1|1)^{\oplus n}) \cong Q_n \quad \text{(see \cref{g.qn}).}
\]

The Verma module is \( 2^n \)-dimensional.

Let the block of  \(  \mathfrak{g}\text{-smod} \) containing the trivial module (i.e., the principal block) be denoted by \( \mathcal{O}_0(\mathfrak{gl}(1|1)^{\oplus n}) \). 

We denote \( \lambda = m_1(\varepsilon_1 - \delta_1) + \dots + m_n(\varepsilon_n - \delta_n) \) more compactly as the tuple \( (m_1, \dots, m_n) \).

Then,
\[
\mathcal{O}_0(\mathfrak{gl}(1|1)^{\oplus n}) = \operatorname{Filt} \{ L^{\mathfrak{b}_\text{st}}(m_1, \dots, m_n) \}_{(m_1, \dots, m_n) \in \mathbb{Z}^n}.
\]

$I=\{1,2,\dots,n\}$ is clearly a $(\mathfrak{b}_{\text{st}}, 0)$-hypercubic collection.

\begin{remark}\label{gl11n}
\( \mathfrak{g} = \mathfrak{b}_\text{st} + r_I \mathfrak{b}_\text{st} \) is a \( (m_1, \dots, m_n) \)-adjusted Borel subalgebra, and  \[ M^{\mathfrak{b}_\text{st} + r_I \mathfrak{b}_\text{st}}(m_1, \dots, m_n) = L^{\mathfrak{b}_\text{st}}(m_1, \dots, m_n)  \] is one-dimensional.

\(\mathfrak{g}_{\overline{0}} = \mathfrak{b}_{\overline{0}} = \mathfrak{h} = \mathfrak{b}_\text{st} \cap r_I \mathfrak{b}_\text{st} \) is also a \( (m_1, \dots, m_n) \)-adjusted Borel subalgebra, and \[M^{\mathfrak{b}_\text{st} \cap r_I \mathfrak{b}_\text{st} }(m_1, \dots, m_n) = P^{\mathfrak{b}_\text{st}}(m_1, \dots, m_n) = \operatorname{Ind}^{\mathfrak{g}}_{\mathfrak{g}_{\overline{0}}} P_{\overline{0}}(m_1, \dots, m_n) \].

\end{remark}

\begin{remark}

As is well known, \( \mathcal{O}_0(\mathfrak{gl}(1|1)) \) is simply the category of finite-dimensional representations of the path algebra \( K_1^\infty \) of the infinite quiver:

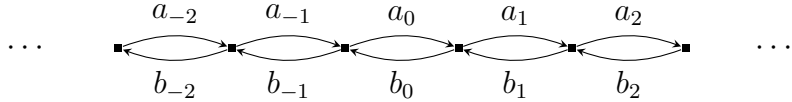
\begin{center}
\begin{tikzpicture}[scale=1.5, >=stealth, auto]
    % Define vertices explicitly
    \node[fill=black, inner sep=1.5pt] (v-3) at (-3, 0) {};
    \node[fill=black, inner sep=1.5pt] (v-2) at (-2, 0) {};
    \node[fill=black, inner sep=1.5pt] (v-1) at (-1, 0) {};
    \node[fill=black, inner sep=1.5pt] (v0) at (0, 0) {};
    \node[fill=black, inner sep=1.5pt] (v1) at (1, 0) {};
    \node[fill=black, inner sep=1.5pt] (v2) at (2, 0) {};

    % Draw two arrows for each pair of neighboring vertices
    \draw[->] (v-3) to[bend left=20] node[above] {$a_{-2}$} (v-2);
    \draw[->] (v-2) to[bend left=20] node[below] {$b_{-2}$} (v-3);

    \draw[->] (v-2) to[bend left=20] node[above] {$a_{-1}$} (v-1);
    \draw[->] (v-1) to[bend left=20] node[below] {$b_{-1}$} (v-2);

    \draw[->] (v-1) to[bend left=20] node[above] {$a_{0}$} (v0);
    \draw[->] (v0) to[bend left=20] node[below] {$b_{0}$} (v-1);

    \draw[->] (v0) to[bend left=20] node[above] {$a_{1}$} (v1);
    \draw[->] (v1) to[bend left=20] node[below] {$b_{1}$} (v0);

    \draw[->] (v1) to[bend left=20] node[above] {$a_{2}$} (v2);
    \draw[->] (v2) to[bend left=20] node[below] {$b_{2}$} (v1);

    % Add the dots on both ends
    \node at (-3.8, 0) {$\cdots$};
    \node at (2.8, 0) {$\cdots$};
\end{tikzpicture}
\end{center}

modulo the relations \( a_i b_i = b_{i-1} a_{i-1} \) and \( a_i a_{i+1} = 0 = b_{i+1} b_i \) for all \( i \in \mathbb{Z} \).

As described in \cite{stroppel2012highest}, \( K_1^\infty \) is the simplest nontrivial example of an algebra that belongs to the class of (generalized) Khovanov algebras. As shown in \cite{stroppel2012highest} (originally due to Serganova) any atypicality 1 block of $\mathfrak{gl}(m|n)$-smod is equivalent to $\mathcal{O}_0(\mathfrak{gl}(1|1))$. 

\end{remark}

\subsection{Realizing the Principal Block of 
\texorpdfstring{$\mathfrak{gl}(1|1)^{\oplus n}$}{gl(1|1)⊕n}}

Let \( J \) be a \( (\mathfrak{b}, \lambda) \)-hypercubic collection of size  \(n\).

% Definition
\begin{definition}
As defined in \cite{stroppel2012highest}(5.6), for \( M \in \mathfrak{g}\text{-sMod} \), we define
\[
\operatorname{Hom}^{\text{fin}}\biggl(
    \bigoplus_{(m_j)_{j \in J} \in \mathbb{Z}^J} 
    M^{\mathfrak{b} \cap r_J \mathfrak{b}}\biggl(\lambda + \sum_{j \in J} m_j \alpha_j\biggr), M
\biggr)
:= 
\]
\[
\bigoplus_{(m_j)_{j \in J} \in \mathbb{Z}^J} 
\operatorname{Hom}\biggl(
    M^{\mathfrak{b} \cap r_J \mathfrak{b}}\biggl(\lambda + \sum_{j \in J} m_j \alpha_j\biggr), M
\biggr)
\subseteq 
\]
\[
\operatorname{Hom}\biggl(
    \bigoplus_{(m_j)_{j \in J} \in \mathbb{Z}^J} 
    M^{\mathfrak{b} \cap r_J \mathfrak{b}}\biggl(\lambda + \sum_{j \in J} m_j \alpha_j\biggr), M
\biggr).
\]
In particular, we define:
\[
\operatorname{End}^{\text{fin}}\biggl(
    \bigoplus_{(m_j)_{j \in J} \in \mathbb{Z}^J} 
    M^{\mathfrak{b} \cap r_J \mathfrak{b}}\biggl(\lambda + \sum_{j \in J} m_j \alpha_j\biggr)
\biggr)
:= 
\]
\[
\operatorname{Hom}^{\text{fin}}\biggl(
    \bigoplus_{(m_j)_{j \in J} \in \mathbb{Z}^J} 
    M^{\mathfrak{b} \cap r_J \mathfrak{b}}\biggl(\lambda + \sum_{j \in J} m_j \alpha_j\biggr),
    \bigoplus_{(m_j)_{j \in J} \in \mathbb{Z}^J} 
    M^{\mathfrak{b} \cap r_J \mathfrak{b}}\biggl(\lambda + \sum_{j \in J} m_j \alpha_j\biggr)
\biggr).
\]

This algebra is not unital but locally unital, and it is not finite-dimensional but locally finite-dimensional.
\end{definition}

\begin{proposition}\label{111111}
The algebra structure of
\[
\operatorname{End}^{\text{fin}}\biggl(
    \bigoplus_{(m_j)_{j \in J} \in \mathbb{Z}^J} 
    M^{\mathfrak{b} \cap r_J \mathfrak{b}}\biggl(\lambda + \sum_{j \in J} m_j \alpha_j\biggr)
\biggr)
\]
depends only on the size of \( J \), and it does not depend on  \( \mathfrak{g}, \mathfrak{b}, J \), or \( \lambda \).
\end{proposition}

\begin{proof}
Let \( J_1, J_2 \subset J \). By \cref{2.3.pbw}, \cref{ch_neo} and \cref{fuhen}, we have the following isomorphism:
\[
\left. U(\mathfrak{g}) F_{J_1}^{\mathfrak{b}} E_{J_2}^{\mathfrak{b}} v_\lambda^{\mathfrak{b} \cap r_J \mathfrak{b}} \middle/
\sum_{J_1 \subset J_1' \subset J, J_2 \subset J_2' \subset J, (J_1', J_2') \neq (J_1, J_2)} U(\mathfrak{g}) F_{J_1'}^{\mathfrak{b}} E_{J_2'}^{\mathfrak{b}} v_\lambda^{\mathfrak{b} \cap r_J \mathfrak{b}}
\right.
\cong M^{\mathfrak{b} + r_J \mathfrak{b}}(\lambda - \Sigma_{J_1}^{\mathfrak{b}} + \Sigma_{J_2}^{\mathfrak{b}}).
\]

The module \( M^{\mathfrak{b} \cap r_J \mathfrak{b}}(\lambda) \) is filtered by \( 4^n \) bricks of the form \( M^{\mathfrak{b} + r_J \mathfrak{b}}(\lambda - \Sigma_{J_1}^{\mathfrak{b}} + \Sigma_{J_2}^{\mathfrak{b}}) \), and the relations between these bricks are determined by the relations between their highest weight vectors of the form \( F_{J_1}^{\mathfrak{b}} E_{J_2}^{\mathfrak{b}} v_\lambda^{\mathfrak{b} \cap r_J \mathfrak{b}} \).

By the orthogonality of the roots, we can see that this structure does not depend on  \( \mathfrak{g}, \mathfrak{b}, J \), or \( \lambda \).

\end{proof}

\begin{remark}
In particular, applying the above when \( J \) has size 1, by \cref{gl11n}, we obtain the following:

For \( \mathfrak{b} \in \mathfrak{B}(\mathfrak{g}) \), \( \lambda \in \mathfrak{h}^* \), \( \alpha \in \Pi_\otimes^{\mathfrak{b}} \), if \( (\lambda, \alpha) = 0 \), then there is an isomorphism of algebras:
\[
\operatorname{End}^{\text{fin}} \left( \bigoplus_{n \in \mathbb{Z}} M^{\mathfrak{b} \cap r_{\alpha} \mathfrak{b}}(\lambda + n \alpha) \right) 
\cong K_1^\infty.
\]
\end{remark}

\begin{proposition}\label{12}

   Let \( J \) be a \( (\mathfrak{b}, \lambda) \)-hypercubic collection. Then, the set 
    \[
    V^J_{\lambda} := \{ M^{\mathfrak{b} + r_J \mathfrak{b}}\biggl(\lambda + \sum_{j \in J} m_j \alpha_j\biggr)
\biggr) \mid (m_j)_{j \in J} \in \mathbb{Z}^J \}
    \]
    forms a semibrick. Moreover, \( M^{\mathfrak{b} \cap r_J \mathfrak{b}}\biggl(\lambda + \sum_{j \in J} m_j \alpha_j\biggr)
\biggr)  \) is the projective cover of \( M^{\mathfrak{b} + r_J \mathfrak{b}}\biggl(\lambda + \sum_{j \in J} m_j \alpha_j\biggr)
\biggr)  \) in   the filtration closure  \( \operatorname{Filt} V^J_{\lambda} \) of \( V^J_{\lambda} \) in the category \( \mathcal{W} \).
\end{proposition}
\begin{proof} The set \( V^J_{\lambda} \) forms a semibrick by \cref{ch_neo} and \cref{h.w.semibrick}. By \cref{Mc_indec}, \( M^{\mathfrak{b} \cap r_J \mathfrak{b}}\biggl(\lambda + \sum_{j \in J} m_j \alpha_j\biggr)
\biggr) \) is indecomposable. Its projectivity in  \(V^J_{\lambda}\) follows from \cref{7.2.kiso}.

\end{proof}

% Theorem for Category Equivalence
\begin{theorem}\label{main_result2}
Let \( J \) be a \( (\mathfrak{b}, \lambda) \)-hypercubic collection of size \(n\).

Then, \( \operatorname{Filt}(V^J_{\lambda}) \) depends only on  \( n\), and it does not depend on  \( \mathfrak{g}, \mathfrak{b}, J \), or \( \lambda \).

In particular, by considering the case \( \mathfrak{g} = \mathfrak{gl}(1|1)^{\oplus n} \), it follows from \cref{12} and \cref{gl11n} that
\[
\operatorname{Filt}(V^J_{\lambda}) = \mathcal{O}_{0}(\mathfrak{gl}(1|1)^{\oplus n}).
\]
\end{theorem}

% Proof Reference
\begin{proof} \
By \cref{111111} and The locally finite version of abstract Morita theory, we have that \( \operatorname{Filt}(V^J_{\lambda}) \) is independent of  \( \mathfrak{g}, \mathfrak{b}, J \), or \( \lambda \).. See \cite[Theorem 5.11]{stroppel2012highest} for details.

\end{proof}

\section{Categorification of odd reflection graphs}
 \label{sec:shortest_paths}
We retain the setting of \cref{sec:changing_Borel}.

% Definition of walks and homomorphisms
\begin{definition}\label{5.3.walkhom}
   Let 
\(
W = \mathfrak{b} c_1 (r_{i_1} \mathfrak{b}) c_2 \dots c_t (r_{i_t} \dots r_{i_1} \mathfrak{b})
\)
be a walk in \( OR(\mathfrak{g}) \).

The corresponding composition of nonzero homomorphisms
\[
M^{\mathfrak{b}}(\lambda - \rho^{\mathfrak{b}}) 
\xrightarrow{\psi_{\lambda}^{\mathfrak{b}, r_{i_1} \mathfrak{b}}} 
M^{r_{i_1} \mathfrak{b}}(\lambda - \rho^{r_{i_1} \mathfrak{b}}) 
\xrightarrow{\psi_{\lambda}^{r_{i_1} \mathfrak{b}, r_{i_2} r_{i_1} \mathfrak{b}}} 
\dots 
\xrightarrow{\psi_{\lambda}^{r_{i_{t-1}} \dots r_{i_1} \mathfrak{b}, r_{i_t} \dots r_{i_1} \mathfrak{b}}} 
M^{r_{i_t} \dots r_{i_1} \mathfrak{b}}(\lambda - \rho^{r_{i_t} \dots r_{i_1} \mathfrak{b}})
\]
is denoted by \( \psi_{\lambda}^{W} \).
\end{definition}

\begin{remark}
\begin{itemize}
    \item \( \psi_{\lambda}^{W} = \psi_{\lambda}^{\mathfrak{b}, \mathfrak{b}'} \iff \psi_{\lambda}^{W} \neq 0 \) by \cref{5.2.ch_eq}.
   
    \item If \( W = \mathfrak{b} c (r_i \mathfrak{b}) \), then \( (\lambda, \alpha_i^{\mathfrak{b}}) = 0 \) if and only if \( \psi_{\lambda}^{W^{-1}} \circ \psi_{\lambda}^{W} = 0 \) by \cref{prop:verma_homomorphism}.
    
    \item If \( W = W_1 W_2 \), then, 
    \(\psi_{\lambda}^{W} = \psi_{\lambda}^{W_2} \circ \psi_{\lambda}^{W_1}\).
    
\end{itemize}
\end{remark}

% Lemma
\begin{lemma}\label{phi_nonzeroPBW}
We identify the color set of \( OR(\mathfrak{g}) \) with \( \Delta^{\mathfrak{b}+} \setminus\Delta^{\operatorname{pure+}} \).
    If \( W = (r_{i_t} \dots r_{i_1} \mathfrak{b}) \beta_t \dots \beta_2 (r_{i_1} \mathfrak{b}) \beta_1 \mathfrak{b} \) is a rainbow walk in \( OR(\mathfrak{g}) \), then we have 
\[
\psi_{\lambda}^{W}(v_{\lambda - \rho^{r_{i_t} \dots r_{i_1} \mathfrak{b}}}^{r_{i_t} \dots r_{i_1} \mathfrak{b}}) 
= f_{\beta_t}^{r_{i_{t-1}} \dots r_{i_1} \mathfrak{b}} \dots f_{\beta_2}^{r_{i_1} \mathfrak{b}} f_{\beta_1}^{\mathfrak{b}} v_{\lambda - \rho^{\mathfrak{b}}}^{\mathfrak{b}} 
= f_{\beta_t}^{\mathfrak{b}} \dots f_{\beta_2}^{\mathfrak{b}} f_{\beta_1}^{\mathfrak{b}} v_{\lambda - \rho^{\mathfrak{b}}}^{\mathfrak{b}}
\]
where \( f^{\mathfrak{b}}_{\beta_t} \dots f^{\mathfrak{b}}_{\beta_2} f^{\mathfrak{b}}_{\beta_1} \) is the product of root vectors corresponding to the distinct positive roots related to \( \mathfrak{b} \).

In particular,
\(
\psi_{\lambda}^{W} \neq 0.
\)
\end{lemma}
\begin{proof}
The first claim is clear by repeatedly applying \cref{prop:verma_homomorphism}. The second follows from \cref{2.3.pbw}.
\end{proof}

% Corollary
\begin{corollary}\label{5.3.nei.dec}
The map \( \psi_{\lambda}^{\mathfrak{b}, \mathfrak{b}'} \) can be decomposed into a composition of homomorphisms between neighboring Verma modules.
\end{corollary}

\begin{proof}
There always exists a shortest walk in  between \( \mathfrak{b} \) and \( \mathfrak{b}' \). Thus the statement follows from \cref{5.2.ch_eq} and \cref{phi_nonzeroPBW}.
\end{proof}

\begin{proposition}\label{hom.ignore}
    Let \( W \) be a walk in \( OR(\mathfrak{g}) \). Then, \( \psi_{\lambda}^{W} \neq 0 \) if and only if \(\overline W \) is a rainbow walk  in \( OR(\mathfrak{g}, \lambda) \).
\end{proposition}
\begin{proof}
(\( \Leftarrow \)) Suppose \(\overline W \) is a rainbow walk in \( OR(\mathfrak{g}, \lambda) \). If there is a subwalk of \( W \), \( W_0 = \mathfrak{b}_0 d \mathfrak{b}_1 c_1 \dots c_k \mathfrak{b}_{k+1} d \mathfrak{b}_{k+2} \), where \( d \in D_{\lambda} \) and \( c_1, \dots, c_k \notin D_{\lambda} \), such that \( d, c_1, \dots, c_k \) are distinct colors, then by \cref{g.rb2.kai}, we obtain a rainbow walk \( W_1 = \mathfrak{b}_{k+2} c'_1 \mathfrak{b}_{k+3} \dots \mathfrak{b}_{2k+1} c'_k \mathfrak{b}_{0} \) such that \( \{c_1, \dots, c_k\} = \{c'_1, \dots, c'_k\} \). Since the edge with color \( d \) corresponds to an isomorphism, we have \( \psi_{\lambda}^{W_0} = \psi_{\lambda}^{W^{-1}_1}. \) If necessary, repeat the above operations so that \( \psi_{\lambda}^{W} \) is equal to \( \psi_{\lambda}^{W_2} \) for some rainbow walk \( W_2 \). By \cref{phi_nonzeroPBW}, this is nonzero.

(\( \Rightarrow \)) If \(\overline W \) is not a rainbow walk  in \( OR(\mathfrak{g}, \lambda) \), then there exists a subwalk \( W_0 = \mathfrak{b}_0 d \mathfrak{b}_1 c_1 \dots c_k \mathfrak{b}_{k+1} d \mathfrak{b}_{k+2}, \) such that \( W_1 = \mathfrak{b}_1 c_1 \dots c_k \mathfrak{b}_{k+1} d \mathfrak{b}_{k+2} \) is rainbow when ignoring the colors in \( D_{\lambda} \), and \( d \notin D_{\lambda} \). If necessary, apply \cref{g.rb2.kai} repeatedly to assume that \( c_1, \dots, c_k, d \) are distinct. By \cref{g.rb2.kai}, there exists a rainbow walk \( W_2 = \mathfrak{b}_0 c'_1 \mathfrak{b'}_1 \dots \mathfrak{b'}_{k-1} c'_k \mathfrak{b}_{k+2} \) such that \( \{c_1, \dots, c_k\} = \{c'_1, \dots, c'_k\} \). We have \( \psi_{\lambda}^{W_1} \neq 0 \) and \( \psi_{\lambda}^{W_2} \circ \psi_{\lambda}^{\mathfrak{b}_1 d \mathfrak{b}_0} \neq 0 \), so we have \( \psi_{\lambda}^{W_1} = \psi_{\lambda}^{W_2} \circ \psi_{\lambda}^{\mathfrak{b}_1 d \mathfrak{b}_0}. \) However, \( \psi_{\lambda}^{W_0} = \psi_{\lambda}^{W_1} \circ \psi_{\lambda}^{\mathfrak{b}_0 d \mathfrak{b}_1} = \psi_{\lambda}^{W_2} \circ \psi_{\lambda}^{\mathfrak{b}_1 d \mathfrak{b}_0} \circ \psi_{\lambda}^{\mathfrak{b}_0 d \mathfrak{b}_1} = \psi_{\lambda}^{W_2} \circ 0 = 0, \) which contradicts the assumption. Thus, \( W \) must be rainbow when ignoring the colors in \( D_{\lambda} \).
\end{proof}

\begin{lemma}\label{quohom}
    Let \( W_1 \) and \( W_2 \) be walks in \( OR(\mathfrak{g}) \). If \( \overline{W_1} = \overline{W_2} \) in \( OR(\mathfrak{g}, \lambda) \), then \( \psi_{\lambda}^{W_1} = \psi_{\lambda}^{W_2} \).
\end{lemma}

\begin{proof}
    Since \( \overline{W_1} = \overline{W_2} \), note that the Verma modules corresponding to the starting and ending points of \( W_1 \) and \( W_2 \) are isomorphic. By \cref{hom.ignore}, we have the following equivalences:
    \[
    \psi_{\lambda}^{W_1} \neq 0 
    \iff \overline{W_1} \text{ is rainbow}
    \iff \overline{W_2} \text{ is rainbow}
    \iff \psi_{\lambda}^{W_2} \neq 0.
    \]

    Therefore, the lemma follows by \cref{5.2.ch_eq}.
\end{proof}

\begin{definition}
Now we generalize \cref{5.3.walkhom} to the setting of the quotient graph. This definition is well-defined by \cref{5.2RB}.

Let \( w = \mathfrak{b}_0 c_1 \mathfrak{b}_1 \dots c_t \mathfrak{b}_t \) be a directed walk in \( OR(\mathfrak{g}, \lambda) \).

The corresponding composition of nonzero homomorphisms
\[
    M^{\mathfrak{b}_0}(\lambda - \rho^{\mathfrak{b}_0}) 
    \xrightarrow{\psi_{\lambda}^{\mathfrak{b}_0, \mathfrak{b}_1}} 
    M^{\mathfrak{b}_1}(\lambda - \rho^{\mathfrak{b}_1}) 
    \xrightarrow{\psi_{\lambda}^{\mathfrak{b}_1, \mathfrak{b}_2}} 
    \dots 
    \xrightarrow{\psi_{\lambda}^{\mathfrak{b}_{t-1}, \mathfrak{b}_t}} 
    M^{\mathfrak{b}_t}(\lambda - \rho^{\mathfrak{b}_t})
\]
is denoted by \( \psi_{\lambda}^{w} \).
\end{definition}

\begin{lemma}\label{quohom2}
Let \( W \) be a walk in \( OR(\mathfrak{g}) \), and \( w \) a walk in \( OR(\mathfrak{g}, \lambda) \).
If \( \overline{W} = w \), then \( \psi_{\lambda}^{W} = \psi_{\lambda}^{w} \).
\end{lemma}

\begin{proof}
 There exists a walk \( W' \) in \( OR(\mathfrak{g}) \) such that \( \overline{W'} = w \) and \( \psi_{\lambda}^{W'} = \psi_{\lambda}^{w} \) by \cref{5.3.nei.dec}. Thus, the result follows from \cref{quohom}.
\end{proof}
% Theorem
\begin{theorem}\label{5.3main}
Let \( \lambda \in \mathfrak{h}^* \). For a walk \( w \) in \( OR(\mathfrak{g}, \lambda) \), the following are equivalent:
\begin{enumerate}
    \item \( \psi_{\lambda}^{w} \neq 0 \).
    \item \( w \) is rainbow.
    \item \( w \) is shortest.
\end{enumerate}
\end{theorem}
\begin{proof}
    The equivalence of 1. and 2. follows from \cref{hom.ignore}, \cref{quohom}, \cref{quohom2}.
The equivalence of 2. and 3. is precisely \cref{thm:odd_exchange}.

\end{proof}

\begin{remark}
    In \cref{5.3main}, it is important to emphasize that condition 3, being shortest, is independent of the coloring of the edges. However, the equivalence between conditions 1 and 3 is not immediately obvious without considering the coloring and using the rainbow condition as an intermediary.
\end{remark}

\begin{example}
Recall \cref{glmndef}
and consider \( \mathfrak{g} = \mathfrak{gl}(2|2) \) with \( \lambda = 0 \).

Since the walk \( (1) \to (2) \to (21) \) is shortest, we have
\[
\psi_\lambda^{(2),(21)} \circ \psi_\lambda^{(1),(2)} = \psi_\lambda^{(1),(21)} = \psi_\lambda^{(1^2),(21)} \circ \psi_\lambda^{(1),(1^2)}.
\]

Thus, we obtain the following equation illustrated below:
\[
\psi = \psi_\lambda^{(21),(1^2)} \circ \psi_\lambda^{(2),(21)} \circ \psi_\lambda^{(1),(2)} \circ \psi_\lambda^{\emptyset,(1)} = \psi_\lambda^{(21),(1^2)} \circ \psi_\lambda^{(1^2),(21)} \circ \psi_\lambda^{(1),(1^2)} \circ \psi_\lambda^{\emptyset,(1)}.
\]

On the other hand, we have
\[
\psi_\lambda^{(21),(1^2)} \circ \psi_\lambda^{(1^2),(21)} = 0, \text{ hence }  \psi = \psi_\lambda^{(21),(1^2)} \circ \psi_\lambda^{(2),(21)} \circ \psi_\lambda^{(1),(2)} \circ \psi_\lambda^{\emptyset,(1)}= 0.
\]

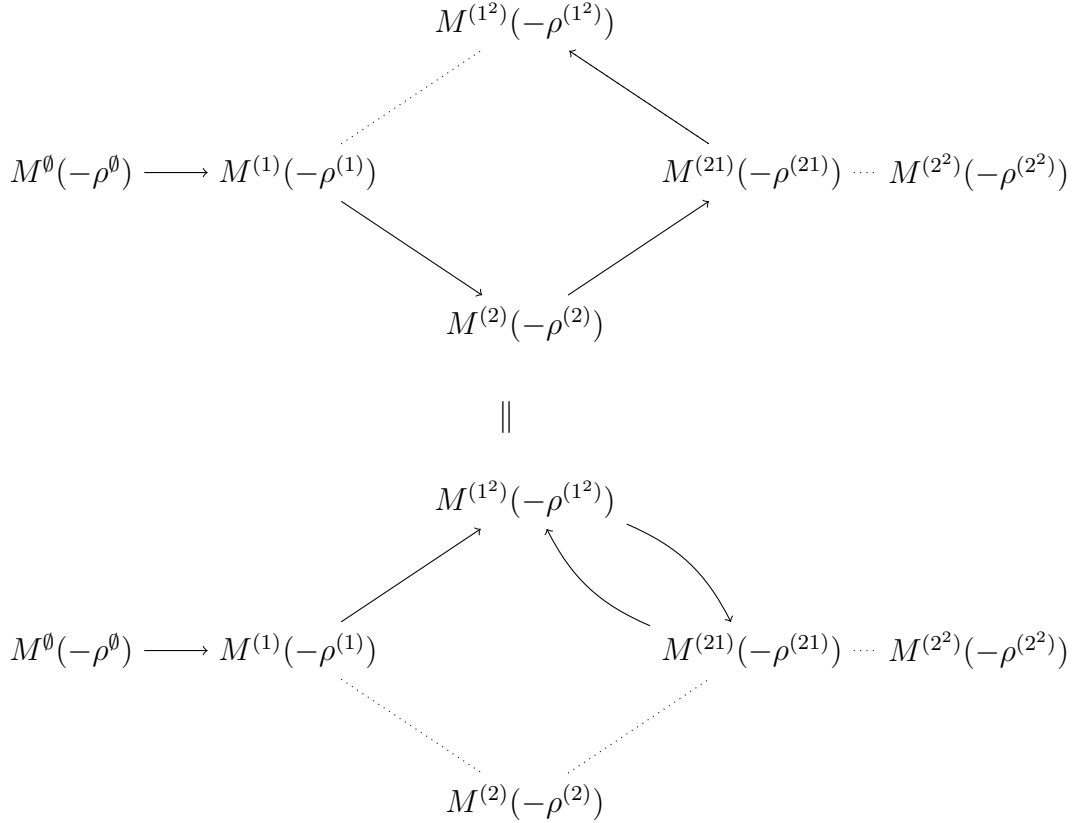
\begin{tikzpicture}
    % Nodes with Young diagrams and specified replacements
    \node (A) at (0,0) {\( M^{\emptyset}(-\rho^{\emptyset}) \)};
    \node (B) at (3,0) {\( M^{(1)}(-\rho^{(1)}) \)};
    \node (C) at (6,2) {\( M^{(1^2)}(-\rho^{(1^2)}) \)};
    \node (D) at (6,-2) {\( M^{(2)}(-\rho^{(2)}) \)};
    \node (E) at (9,0) {\( M^{(21)}(-\rho^{(21)}) \)};
    \node (F) at (12,0) {\( M^{(2^2)}(-\rho^{(2^2)}) \)};

    % Edges with labels and styling
    \draw[->] (A) -- (B);
    \draw[->] (B) -- (D);
    \draw[->] (D) -- (E);
    \draw[->] (E) -- (C);
    
    % Dotted edges for remaining connections
    \draw[dotted] (B) -- (C);
    \draw[dotted] (C) -- (E);
    \draw[dotted] (E) -- (F);
\end{tikzpicture}

\begin{center}
\(\parallel\)
\end{center}

\begin{tikzpicture}
    % Nodes with \( M^{b}(-\rho^{b}) \) labels
    \node (A) at (0,0) {\( M^{\emptyset}(-\rho^{\emptyset}) \)};
    \node (B) at (3,0) {\( M^{(1)}(-\rho^{(1)}) \)};
    \node (C) at (6,2) {\( M^{(1^2)}(-\rho^{(1^2)}) \)};
    \node (D) at (6,-2) {\( M^{(2)}(-\rho^{(2)}) \)};
    \node (E) at (9,0) {\( M^{(21)}(-\rho^{(21)}) \)};
    \node (F) at (12,0) {\( M^{(2^2)}(-\rho^{(2^2)}) \)};
     
    % Directed edges with arrows
    \draw[->] (A) -- (B);
    \draw[->] (B) -- (C);

    % Double arrows between C and E
    \draw[->] (E) to[bend left=20] (C);
    \draw[->] (C) to[bend left=20] (E);

    % Dotted edges for other connections
    \draw[dotted] (B) -- (D);
    \draw[dotted] (D) -- (E);
    \draw[dotted] (E) -- (F);
\end{tikzpicture}

\begin{center}
\begin{small}
The vertex \( \mathfrak{b} \) of the odd reflection graph \( OR(\mathfrak{g},0) = OR(\mathfrak{g}) \) 
is, identified with the module \( M^{\mathfrak{b}}(-\rho^{\mathfrak{b}}) \) in the sense of \cref{5.2RB}.
The edges of the underlying graph are represented with dashed lines, the homomorphisms between Verma modules corresponding to adjacent Borel subalgebras are depicted as arrows over those edges, and the entire diagram represents the composition of those homomorphisms.

The upper diagram represents the homomorphism 
\( \psi = \psi_\lambda^{(21),(1^2)} \circ \psi_\lambda^{(2),(21)} \circ \psi_\lambda^{(1),(2)} \circ \psi_\lambda^{\emptyset,(1)} \), 
while the lower diagram represents the homomorphism 
\( \psi_\lambda^{(21),(1^2)} \circ \psi_\lambda^{(1^2),(21)} \circ \psi_\lambda^{(1),(1^2)} \circ \psi_\lambda^{\emptyset,(1)} \).
\end{small}

\end{center}

\end{example}

\begin{corollary}
Let \( \mathfrak{b}_1, \mathfrak{b}_2, \mathfrak{b}_3 \in B \) and \( \lambda \in \mathfrak{h}^* \). Then, the following are equivalent:
\begin{enumerate}
    \item In \( OR(\mathfrak{g}, \lambda) \), there exists a shortest walk between \( [\mathfrak{b}_1] \) and \( [\mathfrak{b}_3] \) that passes through \( [\mathfrak{b}_2] \);
    \item \( \psi_\lambda^{\mathfrak{b}_2,\mathfrak{b}_1} \circ \psi_\lambda^{\mathfrak{b}_3,\mathfrak{b}_2} = \psi_\lambda^{\mathfrak{b}_3,\mathfrak{b}_1} \neq 0 \);
    \item \( \operatorname{Im} \psi_\lambda^{\mathfrak{b}_3,\mathfrak{b}_2} \not\subset \ker \psi_\lambda^{\mathfrak{b}_2,\mathfrak{b}_1} \);
    \item \( \operatorname{Im} \psi_\lambda^{\mathfrak{b}_3,\mathfrak{b}_1} \subset \operatorname{Im} \psi_\lambda^{\mathfrak{b}_2,\mathfrak{b}_1} \);
    \item \( \left[ \operatorname{Im} \psi_\lambda^{\mathfrak{b}_2,\mathfrak{b}_1} : L^{\mathfrak{b}_3}( \lambda - \rho^{\mathfrak{b}_3} ) \right] = 1 \);
    \item \( \ker \psi_\lambda^{\mathfrak{b}_3,\mathfrak{b}_2} \subset \ker \psi_\lambda^{\mathfrak{b}_3,\mathfrak{b}_1} \);
    \item \( \psi_\lambda^{\mathfrak{b}_2,\mathfrak{b}_3} \circ \psi_\lambda^{\mathfrak{b}_1,\mathfrak{b}_2} = \psi_\lambda^{\mathfrak{b}_1,\mathfrak{b}_3} \neq 0 \);
    \item \( \operatorname{Im} \psi_\lambda^{\mathfrak{b}_1,\mathfrak{b}_2} \not\subset \ker \psi_\lambda^{\mathfrak{b}_2,\mathfrak{b}_3} \);
    \item \( \operatorname{Im} \psi_\lambda^{\mathfrak{b}_1,\mathfrak{b}_3} \subset \operatorname{Im} \psi_\lambda^{\mathfrak{b}_2,\mathfrak{b}_3} \);
    \item \( \left[ \operatorname{Im} \psi_\lambda^{\mathfrak{b}_2,\mathfrak{b}_3} : L^{\mathfrak{b}_1}( \lambda - \rho^{\mathfrak{b}_1} ) \right] = 1 \);
    \item \( \ker \psi_\lambda^{\mathfrak{b}_1,\mathfrak{b}_2} \subset \ker \psi_\lambda^{\mathfrak{b}_1,\mathfrak{b}_3} \).
\end{enumerate}
\end{corollary}

\begin{proof}
This follows from \cref{5.3main} and \cref{5.2.ch_eq}. Note that (1) is symmetric with respect to \( \mathfrak{b}_1 \) and \( \mathfrak{b}_3 \). Additionally, (2), (3), (4), (5), and (6) are symmetric to (7), (8), (9), (10), and (11), respectively.
\end{proof}

\begin{corollary}\label{sb}
Let \( \bar{\mathfrak{b}} \in \mathfrak{B}(\mathfrak{g}) \) and \( \lambda \in \mathfrak{h}^* \).  
For each \( \mathfrak{b} \in \mathfrak{B}(\mathfrak{g}) \), define
\[
I_{\mathfrak{b}} := \left\{ i \,\middle|\, \text{There exists a rainbow path from } r_i \mathfrak{b} \text{ to } \bar{\mathfrak{b}} \text{ passing through } \mathfrak{b} \right\}.
\]
Then the collection
\[
H^{\bar{\mathfrak{b}}}_{ \lambda} := \left\{ B^{\mathfrak{b}, \bar{\mathfrak{b}}}(\lambda) := \operatorname{Im} \psi_{\lambda}^{\mathfrak{b}, \bar{\mathfrak{b}}} \Big/ \sum_{i \in I_{\mathfrak{b}}} \operatorname{Im} \psi_{\lambda}^{r_i \mathfrak{b}, \bar{\mathfrak{b}}} \right\}_{\mathfrak{b} \in \mathfrak{B}(\mathfrak{g})}
\]
forms a semibrick.
\end{corollary}

\begin{proof}
Follows from \cref{5.3main} and \cref{h.w.semibrick}.
\end{proof}
The structure of $\mathrm{Filt}\, H_{\lambda}^{\mathfrak{b}}$ seems to be in general difficult to determine. However, once it is shown that Verma modules are contained in it, Morita theory together with \cref{5.3main} allows us to describe it in terms of a quiver with relations.

The following is considered the most basic case of the odd reflection version of the problem discussed in \cite{ko2024join}.

\begin{lemma}\label{cap}
Suppose \( \lambda \in \mathfrak{h}^* \), and let \( \alpha^{ \mathfrak{b}}_1 \) and \( \alpha^{ \mathfrak{b}}_3 \) be \( \mathfrak{b} \)-simple odd isotropic roots such that 
\[
(\lambda, \alpha^{ \mathfrak{b}}_1) = (\lambda, \alpha^{ \mathfrak{b}}_3) = 0.
\]
Then we have:
\[
\operatorname{Im} \psi_{\lambda}^{r_1 \mathfrak{b}, \mathfrak{b}} \cap \operatorname{Im} \psi_{\lambda}^{r_3 \mathfrak{b}, \mathfrak{b}} =  
\begin{cases}
        0 & \text{if } (\alpha^{ \mathfrak{b}}_1, \alpha^{ \mathfrak{b}}_3) \neq 0, \\
        \operatorname{Im} \psi_{\lambda}^{r_1 r_3 \mathfrak{b}, \mathfrak{b}} & \text{if } (\alpha^{ \mathfrak{b}}_1, \alpha^{ \mathfrak{b}}_3) = 0.
\end{cases}
\]
\end{lemma}

\begin{proof}
Let \( f_1 \), \( f_3 \), and \( f \) be the root vectors corresponding to \( -\alpha^{ \mathfrak{b}}_1 \), \( -\alpha^{ \mathfrak{b}}_3 \), and the even root \( -\alpha^{ \mathfrak{b}}_1 - \alpha^{ \mathfrak{b}}_3 \), respectively. 

If \((\alpha^{ \mathfrak{b}}_1, \alpha^{ \mathfrak{b}}_3) \neq 0\), then we have
\[
f_1 f_3 v_\lambda^{\mathfrak{b}} + f_3 f_1 v_\lambda^{\mathfrak{b}} = f v_\lambda^{\mathfrak{b}}.
\]

Since $f_i$ is a simple root vector, every element of 
$U(\mathfrak{g}) f_i v_{\lambda}^{\mathfrak{b}}$ 
can be written uniquely in the form 
\[
  x f_i v_{\lambda}^{\mathfrak{b}}, \qquad x \in U(\mathfrak{n}^-).
\]

Suppose there exist \( y, z \in U(\mathfrak{n}^{-}) \) such that 
\[
y f_1 v_\lambda^{\mathfrak{b}} = z f_3 v_\lambda^{\mathfrak{b}}.
\]
Using the notation from the PBW theorem \cref{2.3.pbw}, express \( y \) in the basis where \( x_n = f_1 \) and \( x_{n-1} = f_3 \).

Define \( y_1 \) as the sum of PBW monomials in \( y \) where \( f_3 \) does not appear. Since terms containing \( f_1 \) in \( y \) can be ignored, we write
\[
(y-y_1)f_1 v_\lambda^{\mathfrak{b}} = y_2 f_3 f_1 v_\lambda^{\mathfrak{b}}
\]
for some \( y_2 \), where \( y_2 \) is the sum of PBW monomials in which neither \( f_1 \) nor \( f_3 \) appears.

Then, we obtain
\[
y f_1 v_\lambda^{\mathfrak{b}} = (y_1 + y_2 f_3) f_1 v_\lambda^{\mathfrak{b}} = y_1 f_1 v_\lambda^{\mathfrak{b}} - y_2 f_1 f_3 v_\lambda^{\mathfrak{b}} + y_2 f v_\lambda^{\mathfrak{b}}.
\]
Here, using the notation from the PBW theorem \cref{2.3.pbw}, the right-hand side is expressed in the PBW basis where \( x_n = f_3 \) and \( x_{n-1} = f_1 \).

Since we can assume that \( z f_3 v_\lambda^{\mathfrak{b}} \) is already represented in this PBW basis, it cannot be equal to \( y f_1 v_\lambda^{\mathfrak{b}} \) unless \( y_1 = y_2 = 0 \).

Therefore, \( \operatorname{Im} \psi_{\lambda}^{r_1 \mathfrak{b}, \mathfrak{b}} \), which is generated by \( f_1 v_\lambda^{\mathfrak{b}} \), and \( \operatorname{Im} \psi_{\lambda}^{r_3 \mathfrak{b}, \mathfrak{b}} \), which is generated by \( f_3 v_\lambda^{\mathfrak{b}} \), must intersect trivially.

If \((\alpha^{ \mathfrak{b}}_1, \alpha^{ \mathfrak{b}}_3) = 0\), then we have
\[
f_1 f_3 v_\lambda^{\mathfrak{b}} + f_3 f_1 v_\lambda^{\mathfrak{b}} = 0.
\]

Thus, \( \operatorname{Im} \psi_{\lambda}^{r_1 \mathfrak{b}, \mathfrak{b}} \cap \operatorname{Im} \psi_{\lambda}^{r_3 \mathfrak{b}, \mathfrak{b}} \) is the submodule generated by \( f_1 f_3 v_\lambda^{\mathfrak{b}} \), which is \( \operatorname{Im} \psi_{\lambda}^{r_1 r_3 \mathfrak{b}, \mathfrak{b}} \).
\end{proof}

\begin{example}
As noted in \cite[Corollary~5.2]{chen2021translated}, when \( \mathfrak{g} \) is of type I and \( \mathfrak{b} \) is either distinguished or anti-distinguished, the Verma module \( M^{\mathfrak{b}}(\lambda) \) has a simple socle for all \( \lambda \).  

On the other hand, according to \cref{cap}, if \( \mathfrak{b} \) is neither distinguished nor anti-distinguished, then the socle of \( M^{\mathfrak{b}}(\lambda) \) is often not simple as following.
\end{example}

\begin{example}\label{M3}

Under the setting of \cref{cap}, assuming \( (\alpha^{ \mathfrak{b}}_1, \alpha^{ \mathfrak{b}}_3) \neq 0 \), the following forms a semibrick by \cref{h.w.semibrick} and \cref{cap}:  
\[
Z(\mathfrak{b}, \lambda, \alpha^{ \mathfrak{b}}_1,\alpha^{ \mathfrak{b}}_3) := \{\operatorname{Im} \psi_{\lambda}^{r_1 \mathfrak{b}, \mathfrak{b}}, M^{\mathfrak{b}}(\lambda - \rho^{\mathfrak{b}}) / (\operatorname{Im} \psi_{\lambda}^{r_1 \mathfrak{b}, \mathfrak{b}} + \operatorname{Im} \psi_{\lambda}^{r_3 \mathfrak{b}, \mathfrak{b}}) , \operatorname{Im} \psi_{\lambda}^{r_3 \mathfrak{b}, \mathfrak{b}} \}.
\]

Let \( \operatorname{Filt} Z(\mathfrak{b}, \lambda, \alpha^{ \mathfrak{b}}_1, \alpha^{ \mathfrak{b}}_3) \) denote the filtration closure of \( Z(\mathfrak{b}, \lambda, \alpha^{ \mathfrak{b}}_1, \alpha^{ \mathfrak{b}}_3) \) in the category \( \mathcal{W} \).

By \cref{cap}, since \( \operatorname{Im} \psi_{\lambda}^{r_1 \mathfrak{b}, \mathfrak{b}} + \operatorname{Im} \psi_{\lambda}^{r_3 \mathfrak{b}, \mathfrak{b}} \cong \operatorname{Im} \psi_{\lambda}^{r_1 \mathfrak{b}, \mathfrak{b}} \oplus \operatorname{Im} \psi_{\lambda}^{r_3 \mathfrak{b}, \mathfrak{b}} \), the corresponding projective covers of \(\operatorname{Im} \psi_{\lambda}^{r_1 \mathfrak{b}, \mathfrak{b}},M^{\mathfrak{b}}(\lambda - \rho^{\mathfrak{b}}) / (\operatorname{Im} \psi_{\lambda}^{r_1 \mathfrak{b}, \mathfrak{b}} + \operatorname{Im} \psi_{\lambda}^{r_3 \mathfrak{b}, \mathfrak{b}}) \) and \(\operatorname{Im} \psi_{\lambda}^{r_3 \mathfrak{b}, \mathfrak{b}}\) in \( \operatorname{Filt} Z(\mathfrak{b}, \lambda, \alpha^{ \mathfrak{b}}_1,\alpha^{ \mathfrak{b}}_3) \) are 
\[
M^{r_1 \mathfrak{b}}(\lambda - \rho^{r_1 \mathfrak{b}}), \quad M^{\mathfrak{b}}(\lambda - \rho^{\mathfrak{b}}), \quad \text{and} \quad M^{r_3 \mathfrak{b}}(\lambda - \rho^{r_3 \mathfrak{b}}),
\]
respectively (these are Ext-orthogonal to the three bricks by \cref{7.2.kiso}).

 The composition of the homomorphisms between the three projective covers can be easily described by odd Verma's theorem \cref{5.3main}. By Morita theory, it can be verified that \( \operatorname{Filt} Z(\mathfrak{b}, \lambda, \alpha^{ \mathfrak{b}}_1,\alpha^{ \mathfrak{b}}_3) \) is equivalent to the category of finite-dimensional modules over a finite-dimensional algebra defined by the following quiver and relations.

\begin{center}
\begin{tikzpicture}[->, thick, shorten >=1pt, node distance=2cm]

% Nodes
\node (A) at (0,0) {1};
\node (B) [right of=A] {2};
\node (C) [right of=B] {3};

% Arrows (bidirectional with curves)
\draw[->] (A) to[bend left=20] node[above] {$a$} (B);
\draw[->] (B) to[bend left=20] node[below] {$b$} (A);
\draw[->] (B) to[bend left=20] node[above] {$c$} (C);
\draw[->] (C) to[bend left=20] node[below] {$d$} (B);

\end{tikzpicture}
\end{center}

The relations are:
\[
ab = ba = cd = dc = 0.
\]

On the other hand,  the set
\[
Z(r_1 \mathfrak{b}, \lambda, \alpha^{ \mathfrak{b}}_1, \alpha^{ \mathfrak{b}}_3) := \{ M^{r_1 \mathfrak{b}}(\lambda - \rho^{r_1 \mathfrak{b}})/\operatorname{Im} \psi_{\lambda}^{ \mathfrak{b}, r_1 \mathfrak{b}}, \operatorname{Im} \psi_{\lambda}^{ \mathfrak{b}, r_1 \mathfrak{b}}/\operatorname{Im} \psi_{\lambda}^{ r_3 \mathfrak{b}, r_1 \mathfrak{b}}, \operatorname{Im} \psi_{\lambda}^{ r_3 \mathfrak{b}, r_1 \mathfrak{b}} \}
\]
also forms a semibrick, and the projective covers of the bricks in \( \operatorname{Filt} Z(r_1 \mathfrak{b}, \lambda, \alpha^{ \mathfrak{b}}_1, \alpha^{ \mathfrak{b}}_3) \) are again
\[
M^{r_1 \mathfrak{b}}(\lambda - \rho^{r_1 \mathfrak{b}}), \quad M^{\mathfrak{b}}(\lambda - \rho^{\mathfrak{b}}), \quad \text{and} \quad M^{r_3 \mathfrak{b}}(\lambda - \rho^{r_3 \mathfrak{b}}).
\]

Therefore, \( \operatorname{Filt} Z(\mathfrak{b}, \lambda, \alpha^{ \mathfrak{b}}_1, \alpha^{ \mathfrak{b}}_3) \) and \( \operatorname{Filt} Z(r_1\mathfrak{b}, \lambda, \alpha^{ \mathfrak{b}}_1, \alpha^{ \mathfrak{b}}_3) \) are equivalent as categories.
\end{example}

\begin{remark}
\cref{M3} readily generalizes to the setting where the quiver consists of \( n \) vertices aligned in a straight line.
\end{remark}

\begin{example}\label{M4}
    Under the setting of \cref{cap}, assuming \( (\alpha^{ \mathfrak{b}}_1, \alpha^{ \mathfrak{b}}_3) = 0 \). Then, the following forms a semibrick by \cref{h.w.semibrick} and \cref{cap}:  
\[
\begin{aligned}
Y(\mathfrak{b}, \lambda, \alpha^{\mathfrak{b}}_1, \alpha^{\mathfrak{b}}_3) := \{ &
\operatorname{Im} \psi_{\lambda}^{r_1 \mathfrak{b}, \mathfrak{b}} / \operatorname{Im} \psi_{\lambda}^{r_1 r_3 \mathfrak{b}, \mathfrak{b}},
M^{\mathfrak{b}}(\lambda - \rho^{\mathfrak{b}}) / (\operatorname{Im} \psi_{\lambda}^{r_1 \mathfrak{b}, \mathfrak{b}} + \operatorname{Im} \psi_{\lambda}^{r_3 \mathfrak{b}, \mathfrak{b}}), \\
& \operatorname{Im} \psi_{\lambda}^{r_3 \mathfrak{b}, \mathfrak{b}} / \operatorname{Im} \psi_{\lambda}^{r_1 r_3 \mathfrak{b}, \mathfrak{b}}, 
\operatorname{Im} \psi_{\lambda}^{r_1 r_3 \mathfrak{b}, \mathfrak{b}} \}.
\end{aligned}
\]

Let \( \operatorname{Filt} Y(\mathfrak{b}, \lambda, \alpha^{ \mathfrak{b}}_1, \alpha^{ \mathfrak{b}}_3) \) denote the filtration closure of \( Y(\mathfrak{b}, \lambda, \alpha^{ \mathfrak{b}}_1, \alpha^{ \mathfrak{b}}_3) \) in the category \( \mathcal{W} \).

The corresponding projective covers of \( \operatorname{Im} \psi_{\lambda}^{r_1 \mathfrak{b}, \mathfrak{b}}/\operatorname{Im} \psi_{\lambda}^{r_1 r_3 \mathfrak{b}, \mathfrak{b}}, M^{\mathfrak{b}}(\lambda - \rho^{\mathfrak{b}}) / (\operatorname{Im} \psi_{\lambda}^{r_1 \mathfrak{b}, \mathfrak{b}} + \operatorname{Im} \psi_{\lambda}^{r_3 \mathfrak{b}, \mathfrak{b}}) , \operatorname{Im} \psi_{\lambda}^{r_3 \mathfrak{b}, \mathfrak{b}}/\operatorname{Im} \psi_{\lambda}^{r_1 r_3 \mathfrak{b}, \mathfrak{b}}, \operatorname{Im} \psi_{\lambda}^{r_1 r_3 \mathfrak{b}, \mathfrak{b}} \) in \( \operatorname{Filt} Y(\mathfrak{b}, \lambda, \alpha^{ \mathfrak{b}}_1, \alpha^{ \mathfrak{b}}_3) \) are  
\[
M^{r_1 \mathfrak{b}}(\lambda - \rho^{r_1 \mathfrak{b}}), \quad M^{\mathfrak{b}}(\lambda - \rho^{\mathfrak{b}}), \quad M^{r_3 \mathfrak{b}}(\lambda - \rho^{r_3 \mathfrak{b}}), \quad \text{and} \quad M^{r_1
r_3 \mathfrak{b}}(\lambda - \rho^{r_1 r_3 \mathfrak{b}}),
\]  
respectively (these are Ext-orthogonal to the three bricks by \cref{7.2.kiso}).  

The composition of the homomorphisms between the three projective covers can be easily described by odd Verma's theorem \cref{5.3main}. By Morita theory, it can be verified that \( \operatorname{Filt} Y(\mathfrak{b}, \lambda, \alpha^{ \mathfrak{b}}_1, \alpha^{ \mathfrak{b}}_3) \) is equivalent to the category of finite-dimensional modules over a finite-dimensional algebra defined by the following quiver and relations.

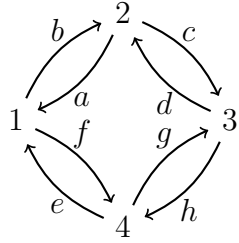
\begin{center}
    \begin{tikzpicture}[->, thick, shorten >=1pt, node distance=2cm]

% Nodes
\node (A) at (0,0) {2};
\node (B) [below left of=A] {1};
\node (C) [below right of=A] {3};
\node (D) [below right of=B] {4};

% Arrows (bidirectional with curves)
\draw[->] (A) to[bend left=20] node[below] {$a$} (B);
\draw[->] (B) to[bend left=20] node[above] {$b$} (A);
\draw[->] (A) to[bend left=20] node[above] {$c$} (C);
\draw[->] (C) to[bend left=20] node[below] {$d$} (A);
\draw[->] (D) to[bend left=20] node[below] {$e$} (B);
\draw[->] (B) to[bend left=20] node[above] {$f$} (D);
\draw[->] (D) to[bend left=20] node[above] {$g$} (C);
\draw[->] (C) to[bend left=20] node[below] {$h$} (D);

\end{tikzpicture}
\end{center}

The relations are:
\begin{align*}
& ab = ba = cd = dc = gh = hg = ef = fe = 0, \\
& bch = che = heb = ebc = afg = fgd = gda = daf = 0, \\
& af - ch = eb - gd = bc - fg = da - he = 0.
\end{align*}
\end{example}

\begin{remark}
The square-shaped structure in \cref{M4} readily generalizes to the setting of the hypercube version \( Q_n \) \cref{g.qn}.
\end{remark}

\section{Associated varieties and projective dimension of Verma modules} \label{sec:applications}
We retain the setting of  \cref{sec:shortest_paths}.
\subsection{BGG category \( \mathcal{O} \)}

\begin{definition}
Let \( \mathfrak{b} \in \mathfrak{B(g)} \).  The category $\mathcal{O}$ is defined as the Serre subcategory of $\mathcal{W}$ generated by the semibrick 
\[
\{ L^{\mathfrak{b}}(\lambda) \mid \lambda \in \mathfrak{h}^* \}.
\]
Moreover, the embedding \( \mathcal{O} \hookrightarrow \mathcal{W} \) is extension full \cite{delorme1980extensions,coulembier2015homological}.
According to \cref{prop:verma_homomorphism}, the structure as an abelian category depends only on \( \mathfrak{b}_{\overline{0}} \) (however,the highest weight structure depends strongly on \( \mathfrak{b} \)).
\end{definition}

We fix, once and for all, the standard even Borel subalgebra \( \mathfrak{b}_{\overline{0}} \).

Similarly, by replacing \( \mathfrak{g} \) with \( \mathfrak{g}_{\overline{0}} \), we define \( \mathcal{O}_{\overline{0}} \) as a full subcategory of \( \mathfrak{g}_{\overline{0}} \)-sMod. 

In other words,
\(
\mathcal{O} = \{ M \in \mathfrak{g}\text{-sMod} \mid \operatorname{Res}_{\mathfrak{g}_{\overline{0}}}^{\mathfrak{g}} M \in \mathcal{O}_{\overline{0}} \}.
\)
In particular, \( \mathcal{O} \) is finite length since \(\mathcal{O}_{\overline{0}}\) is finite length.

\(\operatorname{Res}_{\mathfrak{g}_{\overline{0}}}^{\mathfrak{g}}\) and \(\operatorname{Ind}_{\mathfrak{g}_{\overline{0}}}^{\mathfrak{g}} = U(\mathfrak{g}) \otimes_{U(\mathfrak{g}_{\overline{0}})}\) - are exact functors between \( \mathcal{O}_{\overline{0}} \) and \( \mathcal{O} \) (\cite{bell1993theory} Theorem 2.2,\cite{coulembier2017gorenstein} Section 6.1).

The projective cover of \( L^{\mathfrak{b}}(\lambda) \) in \( \mathcal{O} \) (whose existence is guaranteed by \cref{O.indproj}) is denoted by \( P^{\mathfrak{b}}(\lambda) \).  
Similarly, the projective cover of \( L_{\overline{0}}(\lambda) \) in \( \mathcal{O}_{\overline{0}} \) is denoted by \( P_{\overline{0}}(\lambda) \).

\begin{definition}

    For fixed \( \mathfrak{b} \in B(\mathfrak{g}) \), we say that an object has a \( \mathfrak{b} \)-Verma flag if it admits a filtration by \( \mathfrak{b} \)-Verma modules.
    Let \( \mathcal{F}{\Delta^{\mathfrak{b}}} \) denote the full subcategory of \( \mathcal{O} \) consisting of objects with a \( \mathfrak{b} \)-Verma flag.

\end{definition}

We summarize some basic facts about projective modules induced from the even part, which are generally known. 

\begin{proposition}\label{O.indproj}
Let \( \lambda \in \mathfrak{h}^* \).  
Fix a Borel subalgebra \( \mathfrak{b} \).

\begin{enumerate}

    \item \( \operatorname{Ind}_{\mathfrak{g}_{\overline{0}}}^{\mathfrak{g}} M_{\overline{0}}(\lambda) \in \bigcap_{\mathfrak{b} \in \mathfrak{B}} \mathcal{F}{\Delta^{\mathfrak{b}}}; \)

    \item \( \operatorname{Ind}_{\mathfrak{g}_{\overline{0}}}^{\mathfrak{g}} P_{\overline{0}}(\lambda) \in \bigcap_{\mathfrak{b} \in \mathfrak{B}} \mathcal{F}{\Delta^{\mathfrak{b}}}; \)

    \item \( \operatorname{Ind}_{\mathfrak{g}_{\overline{0}}}^{\mathfrak{g}} P_{\overline{0}}(\lambda) \) is projective in \( \mathcal{O}; \)

    \item \( \operatorname{Ind}_{\mathfrak{g}_{\overline{0}}}^{\mathfrak{g}} M_{\overline{0}}(\lambda) \) is a quotient of \( \operatorname{Ind}_{\mathfrak{g}_{\overline{0}}}^{\mathfrak{g}} P_{\overline{0}}(\lambda); \)

    \item  
    \( M^{\mathfrak{b}}(\lambda) \) is a quotient of \( \operatorname{Ind}_{\mathfrak{g}_{\overline{0}}}^{\mathfrak{g}} M_{\overline{0}}(\lambda); \)

    \item \( P^{\mathfrak{b}}(\lambda) \) is a direct summand of \( \operatorname{Ind}_{\mathfrak{g}_{\overline{0}}}^{\mathfrak{g}} P_{\overline{0}}(\lambda); \)

    \item \( P^{\mathfrak{b}}(\lambda) \in \bigcap_{\mathfrak{b} \in \mathfrak{B}} \mathcal{F}{\Delta^{\mathfrak{b}}}; \)

    \item \( \mathcal{O} \) is a highest weight category with respect to the partial order defined by \( \Delta^{\mathfrak{b}+} \), where \( \mathfrak{b} \)-Verma modules are the standard objects.
\end{enumerate}
\end{proposition}

\subsection{Duflo-Serganova functors}

For basic materials on associated varieties, we refer to \cite{gorelik2022duflo,coulembier2017homological}. 
% Definition and properties of associated variety
\begin{definition}

    Define a subset \( X \subseteq \mathfrak{g}_{\overline{1}} \) by:
    \[
    X := \{ x \in \mathfrak{g}_{\overline{1}} \mid [x, x] = 0 \}.
    \]
    
    For \( M \in \mathfrak{g}\text{-sMod} \), we define a supervector space \( DS_x M \) by:
    \[
    DS_x M := \ker x_M / \operatorname{Im} x_M,
    \]
    where \( x_M \) denotes the action of \( x \) on \( M \). The associated variety of \( M \), denoted \( X_M \), is defined by:
    \[
    X_M := \{ x \in X \mid DS_x M \neq 0 \}.
    \]
\end{definition}

\begin{lemma}[\cite{gorelik2022duflo} Lemma 2.20]\label{7.1.ind}
For \( x \in X \) and \( M \in \mathfrak{g}_{\overline{0}}\text{-sMod} \), we have:
\[
    DS_x \operatorname{Ind}_{\mathfrak{g}_{\overline{0}}}^{\mathfrak{g}} M = 0.
\]
\end{lemma}

\begin{proof}
    This follows from \cref{2.3.pbw}.
\end{proof}

\begin{proposition}[\cite{gorelik2022duflo}]
    Let \( x \in X \).
    \begin{enumerate}
        \item For the adjoint module \( \mathfrak{g} \in \mathfrak{g}\text{-sMod} \), the module \( \mathfrak{g}_x := DS_x(\mathfrak{g}) \) naturally inherits the structure of a basic Lie superalgebra;
        \item There exists a symmetric monoidal \( k \)-linear functor:
        \[
            DS_x: \mathfrak{g}\text{-sMod} \rightarrow \mathfrak{g}_x\text{-sMod}.
        \]
        This functor is known as the \emph{Duflo-Serganova functor}.
    \end{enumerate}
\end{proposition}

\begin{lemma}[Hinich's Lemma  \cite{gorelik2022duflo}]\label{Hinich}
    Let \( x \in X \).
    Given a short exact sequence 
    \[
    0 \rightarrow L \rightarrow M \rightarrow N \rightarrow 0
    \] 
    in \( \mathfrak{g}\text{-sMod} \), there exists \( E \in \mathfrak{g}_x\text{-sMod} \) such that the following sequence is exact in \( \mathfrak{g}_x\text{-sMod} \):
\[
0 \to E \to DS_x L \to DS_x M \to DS_x N \to \Pi E \to 0.
\]
\end{lemma}

\begin{proposition}[\cite{coulembier2017homological} Theorem 4.1]\label{7.1.pd}
    Let \( M \in \mathcal{O} \). If \( \operatorname{pd}_\mathcal{O} M < \infty \), then \( X_M = \{0\} \).
    Here, \( \operatorname{pd}_\mathcal{O} \) represents the projective dimension in \( \mathcal{O} \).
\end{proposition}

\begin{proof}
    Any indecomposable projective module can be expressed as a direct summand of projective modules induced by the even part, as shown in \cref{O.indproj}. Consequently, any projective module vanishes under the application of \( DS_x \), as stated in \cref{7.1.ind}. On the other hand, in \cref{Hinich}, if \( DS_x L = DS_x M = 0 \), then \( DS_x N = 0 \) follows as well. Thus, if \( \operatorname{pd}_\mathcal{O} M < \infty \), then applying \( DS_x \) finitely many times to a finite projective resolution of \( M \) results in the vanishing of all terms. This proves the assertion.
\end{proof}

\begin{remark}
     \cite[Theorem 4.1]{coulembier2017homological} states that when \( \mathfrak{g} = \mathfrak{gl}(m|n) \), the converse of the above proposition also holds. Moreover, it is shown that having both a Kac flag and a dual Kac flag are also equivalent. In particular, the modules in \( \mathcal{F} \Delta^\emptyset \cap \mathcal{F} \Delta^{(n^m)} \) have finite projective dimension.
\end{remark}

Recall, the finistic dimension \( \operatorname{fin.dim} \mathcal{A} \) of an abelian category \( \mathcal{A} \) is defined as 
\[
\operatorname{fin.dim} \mathcal{A} := \sup \{ \operatorname{pd}_\mathcal{A} M \mid M \in \mathcal{A}, \operatorname{pd}_\mathcal{A} M < \infty \}.
\]
If \( \operatorname{pd}_\mathcal{O} M < \infty \), then  it can be concretely bounded from above by the following theorem.

\begin{theorem}[\cite{mazorchuk2014parabolic,chen2023some}]
Let \( \mathfrak{g} \) be a basic Lie superalgebra. Then,
\[
\operatorname{fin.dim} \mathcal{O} = \operatorname{gl.dim} \mathcal{O}_{\overline{0}} = 2l(w_0),
\]
where \( w_0 \) is the longest element in the Weyl group of \( \mathfrak{g}_{\overline{0}} \).
\end{theorem}

\subsection{Singletons}

% Definitions related to orthogonal sets and associated varieties
Define \( \mathcal{S} := \{ S \subseteq \Delta_{\otimes} \mid S \text{ consists of mutually orthogonal roots} \} \). For \( S \in \mathcal{S} \), let \( x_S \in X \) denote the sum of the corresponding root vectors.

We define
\[
\mathcal{S}(M) := \{ S \in \mathcal{S} \mid x_S \in X_M \}.
\]
By \cite[Theorem 5.1]{coulembier2017homological}, this definition is independent of the choice of \( x_S \).

Moreover,  \cite[Theorem 5.1]{gorelik2022duflo} asserts that 
\[
\mathcal{S}(M) = \mathcal{S}(N) \iff X_M = X_N.
\]

We define
\[
\mathcal{S}_1 M := \{ \beta \in \Delta_{\otimes} \mid \{ \beta \} \in \mathcal{S} (M)\}.
\]
The set \( \mathcal{S}_1 M \) has a particular importance. Thus, we focus on it from now on. The following is clear from the definition.
\begin{lemma}\label{S1lemma}
For \( M \in \mathcal{O} \), the following are equivalent:
\begin{itemize}
    \item \( \alpha \in S_1 M \),
    \item There exists \( \lambda \in \mathfrak{h}^* \) and \( 0 \neq v \in M_\lambda \) such that for the root vector \( x = x_\alpha \), the following hold:
    \begin{enumerate}
        \item \( x v = 0 \); \item there does not exist \( u \in M_{\lambda - \alpha} \) such that \( x u = v \).
    \end{enumerate}
\end{itemize}
\end{lemma}

\begin{proposition}
\label{7.2.flag}
Let \( \mathfrak{b} \in \mathfrak{B(g)} \) and \( \lambda \in \mathfrak{h}^* \). Then:
\(
\mathcal{S}_1 M^{\mathfrak{b}}(\lambda) \subseteq \Delta_\otimes^{\mathfrak{b}+}.
\)
\end{proposition}

\begin{proof}
    For a root vector \( f \in \Delta_\otimes \setminus \Delta_\otimes^{\mathfrak{b}+} \), we can choose \( x_1 = f \) in the PBW basis of \( M^{\mathfrak{b}}(\lambda) \) as described in \cref{2.3.pbw}. Then, by considering this basis and \cref{S1lemma}, we have:
    \(
    DS_f M^{\mathfrak{b}}(\lambda) = 0.
    \)
    This implies the result.
\end{proof}

\begin{proposition}\label{1s}
Let \( \lambda \in \mathfrak{h}^* \), \( \mathfrak{b} \in B(\mathfrak{g}) \), and \( \alpha_{i}^{\mathfrak{b}} \in \Pi_{\otimes}^{\mathfrak{b}} \). Then:
\[
(\lambda, \alpha_{i}^{\mathfrak{b}}) = 0 \iff \alpha_{i}^{\mathfrak{b}} \in \mathcal{S}_1 M^{\mathfrak{b}}(\lambda).
\]
\end{proposition}

\begin{proof}
    (\( \Rightarrow \)) If \( (\lambda, \alpha_{i}^{\mathfrak{b}}) = 0 \), then the highest weight vector of \( M^{\mathfrak{b}}(\lambda) \) does not vanish under \( DS_{e_i^{\mathfrak{b}}} \) , implying \( \alpha_{i}^{\mathfrak{b}} \in \mathcal{S}_1 M^{\mathfrak{b}}(\lambda) \) by \cref{S1lemma}.

    (\( \Leftarrow \)) If \( (\lambda, \alpha_{i}^{\mathfrak{b}}) \neq 0 \), then \( M^{\mathfrak{b}}(\lambda) \simeq M^{r_i \mathfrak{b}}(\lambda - \alpha_{i}^{\mathfrak{b}}) \). By \cref{7.2.flag}, we have \( \mathcal{S}_1 M^{\mathfrak{b}}(\lambda) \subseteq \Delta_\otimes^{r_i \mathfrak{b}+} \). Since \( \alpha_{i}^{\mathfrak{b}} \notin \Delta_\otimes^{r_i \mathfrak{b}+} \) by definition, it follows that \( \alpha_{i}^{\mathfrak{b}} \notin \mathcal{S}_1 M^{\mathfrak{b}}(\lambda) \), completing the proof.
\end{proof}

\begin{corollary}\label{hwmS1}
Let \( \lambda \in \mathfrak{h}^* \), \( \mathfrak{b} \in B(\mathfrak{g}) \), and \( \alpha_{i}^{\mathfrak{b}} \in \Pi_{\otimes}^{\mathfrak{b}} \).
    If \( (\lambda, \alpha_{i}^{\mathfrak{b}}) = 0 \) and \( M \) is a \( \mathfrak{b} \)-highest weight module with highest weight \( \lambda \), then \( \alpha_{i}^{\mathfrak{b}} \in \mathcal{S}_1 M \).
\end{corollary}

\begin{proof}
    Since \( \dim M_{\lambda - \alpha_{i}^{\mathfrak{b}}} \leq 1 \), the assertion is similar to that of \( M^{\mathfrak{b}}(\lambda) \) above.
\end{proof}

\begin{proposition}
  Let \( \mathfrak{a} \) denote a \(\lambda\)-adjusted Borel subalgebra.
    Then,  \( \mathcal{S}_1 M^{\mathfrak{a}}(\lambda) \subseteq \Delta^{\mathfrak{a}} \).
\end{proposition}

\begin{proof}
    The proof is exactly the same as in \cref{7.2.flag}.
\end{proof}

\begin{example}
    If \( \Delta^{\mathfrak{a}} = \{\gamma\} \), then \( \mathcal{S}_1 M^{\mathfrak{a}}(\lambda + n \gamma) \) does not depend on \( n \).
\end{example}

\begin{proof}
    Consider the exact sequence 
    \[
    0 \to M^{\mathfrak{a}}(\lambda + (n+1)\gamma) \to \operatorname{Ind}_{\mathfrak{g}_{\overline{0}}}^{\mathfrak{g}} M_{\overline{0}}(\lambda + n\gamma) \to M^{\mathfrak{a}}(\lambda + n\gamma) \to 0,
    \]
    for each \( n \), and use \cref{Hinich} and \cref{7.1.ind}.
\end{proof}

\begin{proposition}[\cite{coulembier2017homological} Proposition 5.14]
    If \(S_1 L^{\mathfrak{b}}(\lambda - \rho^{\mathfrak{b}}) \subseteq \Delta_\otimes^{\text{pure+}} \), then \( OR(\mathfrak{g}, \lambda) \) consists of a single point (see, \cref{RBtriv}).

    In particular, for type I, \( S_1L^{\mathfrak{b}}(\lambda) = \emptyset \) if and only if \( \lambda \) is \(\mathfrak{b}\)-typical.
\end{proposition}

\begin{proof}
    This follows immediately from \cref{RBtriv} and \cref{hwmS1}.
\end{proof}

The following proposition improves \cite[Lemma5.12]{coulembier2017homological} and \cite[Propsition 34]{chen2023some}.
\begin{proposition}\label{d.CS5.12}
    If \( \beta \in \Delta^{\mathfrak{b}+}_\otimes \setminus \Delta^{\text{pure+}} \) and \( (\lambda, \beta) = 0 \), then \( \beta \in \mathcal{S}_1 M^{\mathfrak{b}}(\lambda - \rho^{\mathfrak{b}}) \).
\end{proposition}

\begin{proof} We identify the color set of \( RB(\mathfrak{g}) \) with \( \Delta^{\mathfrak{b}+} \setminus\Delta^{\operatorname{pure+}} \).
    By \cref{g.rb2.kai}, there exists a rainbow walk
    \(
    W = \mathfrak{b}_t \beta \mathfrak{b}_{t-1} \beta_{t-1} \dots \beta_1 \mathfrak{b}.
    \)

    By \cref{phi_nonzeroPBW}, we can write:
    \[
    \psi_\lambda^W(v_{\lambda - \rho^{\mathfrak{b}_t}}^{\mathfrak{b}_t}) = \psi^{\mathfrak{b}_{t-1} \beta_{t-1} \dots \beta_1 \mathfrak{b}}(f^{\mathfrak{b}_{t-1}}_\beta v_{\lambda - \rho^{\mathfrak{b}_{t-1}}}^{\mathfrak{b}_{t-1}}) = f^{\mathfrak{b}_{t-1}}_\beta
    \psi_\lambda^{\mathfrak{b}_{t-1} \beta_{t-1} \dots \beta_1 \mathfrak{b}}(v_{\lambda - \rho^{\mathfrak{b}_{t-1}}}^{\mathfrak{b}_{t-1}}) =
    f_\beta^{\mathfrak{b}} F v_{\lambda - \rho^{\mathfrak{b}}}^{\mathfrak{b}},
    \]
    where \( F v_{\lambda - \rho^{\mathfrak{b}}}^{\mathfrak{b}} \) is the \( \mathfrak{b}_{t-1} \)-highest weight vector. Since \( \beta \in \Delta^{\mathfrak{b}_{t-1}+} \), we have \( e_\beta^{\mathfrak{b}} F v_{\lambda - \rho^{\mathfrak{b}}}^{\mathfrak{b}} = 0 \).

    On the other hand, we know that \( M^{\mathfrak{b}}(\lambda - \rho^{\mathfrak{b}})_{\lambda - \rho^{\mathfrak{b}_{t-1}}} = k F v_{\lambda - \rho^{\mathfrak{b}}}^{\mathfrak{b}} \) by \cref{5.2.ch_eq}, and it follows that
    \[
    M^{\mathfrak{b}}(\lambda - \rho^{\mathfrak{b}})_{\lambda - \rho^{\mathfrak{b}_{t-1}} - \beta} = M^{\mathfrak{b}}(\lambda - \rho^{\mathfrak{b}})_{\lambda - \rho^{\mathfrak{b}_t}} = k f_\beta^{\mathfrak{b}} F v_{\lambda - \rho^{\mathfrak{b}}}^{\mathfrak{b}}.
    \]

    Thus, we have
    \[
    e_\beta^{\mathfrak{b}} f_\beta^{\mathfrak{b}} F v_{\lambda - \rho^{\mathfrak{b}}}^{\mathfrak{b}} = (-f_\beta^{\mathfrak{b}} e_\beta^{\mathfrak{b}} + [e_\beta^{\mathfrak{b}}, f_\beta^{\mathfrak{b}}]) F v_{\lambda - \rho^{\mathfrak{b}}}^{\mathfrak{b}} = 0 + (\lambda - \rho^{\mathfrak{b}_{t-1}}, \beta) F v_{\lambda - \rho^{\mathfrak{b}}}^{\mathfrak{b}} = 0,
    \].

    From the above, the claim follows by \cref{S1lemma}.
\end{proof}

\begin{corollary} \label{S1RB}
\( OR(\mathfrak{g}, \lambda) \) consists of a single point if and only if \( \mathcal{S}_1(M^{\mathfrak{b}}(\lambda - \rho^{\mathfrak{b}})) \subseteq \Delta^{\text{pure+}}_\otimes \).
\end{corollary}
\begin{proof}
If \( OR(\mathfrak{g}, \lambda) \) consists of a single point, then by \cref{7.2.flag}, \( \mathcal{S}_1(M^{\mathfrak{b}}(\lambda - \rho^{\mathfrak{b}})) \subseteq \Delta^{\text{pure+}}_\otimes \).

If \( OR(\mathfrak{g}, \lambda) \) weren't consists of a single point, then by \cref{d.CS5.12}, \( \mathcal{S}_1(M^{\mathfrak{b}}(\lambda - \rho^{\mathfrak{b}})) \not\subseteq \Delta^{\text{pure+}}_\otimes \).
\end{proof}

\begin{corollary}\label{typeI_RBtriv}
Let \( \lambda \in \mathfrak{h}^* \), and let \( \mathfrak{b} \in B(\mathfrak{g}) \) be an arbitrary Borel subalgebra. The following implications hold:
\(
(1) \Rightarrow (2) \Rightarrow (3) \Rightarrow (4) \Rightarrow (5).
\)
Moreover, if \( \mathfrak{g} \) is of type I, then \( (5) \Rightarrow (1) \) also holds.
\begin{enumerate}
    \item \( \lambda + \rho^{\mathfrak{b}} \) is \( \mathfrak{b} \)-typical;
    \item \( \operatorname{pd} M^{\mathfrak{b}}(\lambda) < \infty \);
    \item \( X_{M^{\mathfrak{b}}(\lambda)} = \{ 0 \} \);
    \item \( \mathcal{S}_1 M^{\mathfrak{b}}(\lambda) = \emptyset \);
    \item \( OR(\mathfrak{g}, \lambda) \) consists of a single point.
\end{enumerate}
\end{corollary}
\begin{proof}
If (1) holds, then by the linkage principle \cite{musson2012lie}, the block containing \( M^{\mathfrak{b}}(\lambda) \) has finitely many isomorphism classes of simple modules. Since \( \mathcal{O} \) is a highest weight category (\cref{O.indproj}), the projective dimension is finite. If (2) holds, then (3) follows directly from \cref{7.1.pd}. If (3) holds, then (4) is clear by definition. If (4) holds, then (5) follows from \cref{S1RB}. 

For Lie superalgebras of type I, if (5) holds, then (1) follows because, that  \( OR(\mathfrak{g}, \lambda) \) consists of a single point ensures that \( \lambda \) is \( \mathfrak{b} \)-typical (\cref{typeIRB}).
\end{proof}

\subsection{Verma modules over \(D(2, 1; \alpha\))}
\cref{d.CS5.12} was established through the consideration of one-dimensional weight spaces that can be described by products of odd root vectors. By also considering even roots, we find additional weight spaces that are one-dimensional. This observation allows for a slight generalization of \cref{d.CS5.12}.

\begin{lemma} \label{simpleevenlemma}
    Suppose that \( (\lambda, \beta) = 0 \) for \( \beta \in \Delta_\otimes^{\text{pure}+} \). If there exist \( \bar{\mathfrak{b}} \in B \) and \( \gamma \in \mathbb{Z}_{\geq 0} \Delta_{\overline{0}}^+ \) such that
    \[
    \gamma - \beta \notin \mathbb{Z}_{\geq 0} \Delta^{\bar{\mathfrak{b}}+}, \quad (\beta, \rho^{\bar{\mathfrak{b}}} + \gamma) = 0, \quad \text{and} \quad 
    \dim M^{\bar{\mathfrak{b}}}(\lambda - \rho^{\bar{\mathfrak{b}}})_{\lambda - \rho^{\bar{\mathfrak{b}}} - \beta - \gamma} = 1
    \]
    are satisfied, then for any \( \mathfrak{b} \in B \), we have \( \beta \in \mathcal{S}_1 M^{\mathfrak{b}}(\lambda - \rho^{\mathfrak{b}}) \).
\end{lemma}
\begin{proof}
By \cref{phi_nonzeroPBW}, we can write
\(
\psi_{\lambda}^{\bar{\mathfrak{b}}, \mathfrak{b}}(v_{\lambda - \rho^{\bar{\mathfrak{b}}}}^{\bar{\mathfrak{b}}}) = F v_{\lambda - \rho^{\mathfrak{b}}}^{\mathfrak{b}},
\)
where \( F \) is the product of distinct root vectors in \( -(\Delta_\otimes^{\mathfrak{b}+} \setminus \Delta^{\text{pure}+}) \).
Given \( \gamma \in \mathbb{Z}_{\geq 0} \Delta_{\overline{0}}^{+} \), by \cref{2.3.pbw}, we have \( f_{\gamma}^{\mathfrak{b}} F v_{\lambda - \rho^{\mathfrak{b}}}^{\mathfrak{b}} \neq 0 \).

Since
\(
\gamma - \beta \notin \mathbb{Z}_{\geq 0} \Delta^{\bar{\mathfrak{b}}+},
\)
we obtain
\[
e_{\beta}^{\mathfrak{b}} f_{\gamma}^{\mathfrak{b}} F v_{\lambda - \rho^{\mathfrak{b}}}^{\mathfrak{b}} = 
e_{\beta}^{\mathfrak{b}} f_{\gamma}^{\mathfrak{b}} \psi_{\lambda}^{\bar{\mathfrak{b}}, \mathfrak{b}}(v_{\lambda - \rho^{\bar{\mathfrak{b}}}}^{\bar{\mathfrak{b}}}) = 
 \psi_\lambda^{\bar{\mathfrak{b}}, \mathfrak{b}} (e_{\beta}^{\mathfrak{b}} f_{\gamma}^{\mathfrak{b}} v_{\lambda - \rho^{\bar{\mathfrak{b}}}}^{\bar{\mathfrak{b}}}) = \psi_\lambda^{\bar{\mathfrak{b}}, \mathfrak{b}} (0) = 0.
\]

On the other hand, by \cref{2.3.pbw}, \( \beta \in \Delta_\otimes^{\text{pure}+} \), and 
\(
\dim M^{\bar{\mathfrak{b}}}(\lambda - \rho^{\bar{\mathfrak{b}}})_{\lambda - \rho^{\bar{\mathfrak{b}}} - \beta - \gamma} = 1,
\)
we conclude
\( 
M^{\bar{\mathfrak{b}}}(\lambda - \rho^{\bar{\mathfrak{b}}})_{\lambda - \rho^{\bar{\mathfrak{b}}} - \beta - \gamma} = k f_{\beta}^{\mathfrak{b}} f_{\gamma}^{\mathfrak{b}} F v_{\lambda - \rho^{\mathfrak{b}}}^{\mathfrak{b}}.
\)

Since
\(
(\beta, \rho^{\bar{\mathfrak{b}}} + \gamma) = 0,
\)
we have 
\[
e_\beta^{\mathfrak{b}} f_{\beta}^{\mathfrak{b}} f_{\gamma}^{\mathfrak{b}} F v_{\lambda - \rho^{\mathfrak{b}}}^{\mathfrak{b}} = 
(-f_\beta^{\mathfrak{b}} e_\beta^{\mathfrak{b}} + [e_\beta^{\mathfrak{b}}, f_\beta^{\mathfrak{b}}]) f_{\gamma}^{\mathfrak{b}} F v_{\lambda - \rho^{\mathfrak{b}}}^{\mathfrak{b}} = 0 + (\beta, \rho^{\mathfrak{b}} + \gamma) F v_{\lambda - \rho^{\mathfrak{b}}}^{\mathfrak{b}} = 0.
\]

Thus the claim follows by \cref{S1lemma}.
\end{proof}

\begin{example}
\label{2.3.d21}\label{d21aaaaaaaaaaa}

Let \( \mathfrak{g} = D(2, 1; \alpha) \). See \cite{cheng2019character} for more information on this type of Lie superalgebra.

The vector space \( \mathfrak{h}^* \) has an orthogonal basis \( \{ \delta, \varepsilon_1, \varepsilon_2 \} \) with respect to the inner product \( (\cdot, \cdot) \), where
\(
(\delta, \delta) = -(1 + \alpha), \quad (\varepsilon_1, \varepsilon_1) = 1, \quad (\varepsilon_2, \varepsilon_2) = \alpha.
\)

The sets of roots are as follows:
\[
\Delta_{\overline{0}} = \{ \pm 2 \delta, \pm 2 \varepsilon_1, \pm 2 \varepsilon_2 \}
\]
\[
\Delta_{\overline{1}} = \Delta_{\otimes} = \{ \pm (\delta - \varepsilon_1 - \varepsilon_2), \pm (\delta + \varepsilon_1 - \varepsilon_2), \pm (\delta - \varepsilon_1 + \varepsilon_2), \pm (\delta + \varepsilon_1 + \varepsilon_2) \}
\]

The odd reflection graph \( OR(D(2,1;\alpha)) \) is described as follows.
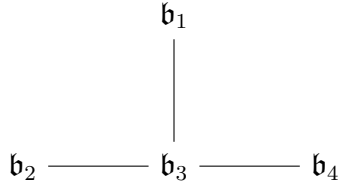
\begin{center}
\begin{tikzpicture}
    % Nodes
    \node (b1) at (0,2) {\(\mathfrak{b}_1\)};
    \node (b2) at (-2,0) {\(\mathfrak{b}_2\)};
    \node (b3) at (0,0) {\(\mathfrak{b}_3\)};
    \node (b4) at (2,0) {\(\mathfrak{b}_4\)};
    
    % Edges
    \draw (b1) -- (b3);
    \draw (b2) -- (b3);
    \draw (b3) -- (b4);
\end{tikzpicture}
\end{center}

The corresponding positive root systems for vertex \(\mathfrak{b}_1\) and \(\mathfrak{b}_3\) are:

\[
\begin{aligned}
\Delta^{\mathfrak{b}_1 +} &= \{ 2\varepsilon_1, \delta + \varepsilon_1 - \varepsilon_3, \delta + \varepsilon_1 + \varepsilon_2, 2\delta,  \delta - \varepsilon_1 - \varepsilon_3, \delta - \varepsilon_1 + \varepsilon_2, 2\varepsilon_3 \}\\ 
\Delta^{\mathfrak{b}_3 +} &= \{ 2\varepsilon_1, \delta + \varepsilon_1 - \varepsilon_3, \delta + \varepsilon_1 + \varepsilon_2, 2\delta,  -\delta + \varepsilon_1 + \varepsilon_3, \delta - \varepsilon_1 + \varepsilon_2, 2\varepsilon_3 \}
\end{aligned}
\]

We  also note that
\[
\Delta^{\text{pure}+} = \{ 2\delta, 2\varepsilon_1, 2\varepsilon_2, \delta + \varepsilon_1 + \varepsilon_2 \} 
, \quad
\Delta_\otimes^{\text{pure}+} = \{ \delta + \varepsilon_1 + \varepsilon_2 \}.
\]

Now we define
\(
\beta := \delta + \varepsilon_1 + \varepsilon_2, \quad
\gamma := 0, \quad
\bar{\mathfrak{b}} := \mathfrak{b}_3.
\)

Then, we have
\[
\rho^{\bar{\mathfrak{b}}} = \frac{1}{2} (2\delta + 2\varepsilon_1 + 2\varepsilon_2) - \frac{1}{2}(-\delta + \varepsilon_1 + \varepsilon_2 + \delta - \varepsilon_1 + \varepsilon_2 + \delta + \varepsilon_1 - \varepsilon_2 + \delta + \varepsilon_1 + \varepsilon_2) = 0.
\]

Note that
\(
\rho^{\mathfrak{b}_1} = \rho^{\bar{\mathfrak{b}}} - \delta + \varepsilon_1 + \varepsilon_2 = -\delta + \varepsilon_1 + \varepsilon_2, \quad
\beta = \rho^{\mathfrak{b}_1} + 2 \delta, \quad 2 \delta \in \Pi^{\mathfrak{b}_1}.
\)

With this setup, we find that
\[
\gamma - \beta = -\beta \notin \mathbb{Z}_{\geq 0} \Delta_{\bar{\mathfrak{b}}}^+,
\quad
(\beta, \rho^{\bar{\mathfrak{b}}} + \gamma) = (\beta, 0) = 0,
\]
\[
\dim M^{\bar{\mathfrak{b}}}(\lambda - \rho^{\bar{\mathfrak{b}}})_{\lambda - \rho^{\bar{\mathfrak{b}}} - \beta - \gamma} = \dim M^{\mathfrak{b}_1}(\lambda - \rho^{\mathfrak{b}_1})_{\lambda - \beta}
= \dim M^{\mathfrak{b}_1}(\lambda - \rho^{\mathfrak{b}_1})_{\lambda - \rho^{\mathfrak{b}_1} - 2\delta} = 1.
\]
Thus, the condition of \cref{simpleevenlemma} is verified.
\end{example}

\begin{corollary}\label{d21main}
For \( \mathfrak{g} = D(2, 1; \alpha) \), \( \lambda \in \mathfrak{h}^* \), and \( \mathfrak{b} \in \mathfrak{B} \), the following are equivalent:
\begin{enumerate}
    \item \( \lambda \) is \( \mathfrak{b} \)-typical;
    \item \( \operatorname{pd} M^{\mathfrak{b}}(\lambda) < \infty \);
   
    \item \( X_{M^{\mathfrak{b}}(\lambda)} = \{ 0 \} \);
     \item \( S_1 M^{\mathfrak{b}}(\lambda) = \emptyset \).
    
\end{enumerate}
\end{corollary}

\begin{proof}
Let \( \beta \in \Delta_\otimes^{\mathfrak{b}} \) and \( (\lambda, \beta) = 0 \). Then, if \( \beta \notin \Delta_\otimes^{\text{pure}+} \), it follows from \cref{d.CS5.12} that \( \beta \in S_1 M^{\mathfrak{b}}(\lambda - \rho^{\mathfrak{b}}) \). If \( \beta \in \Delta_\otimes^{\text{pure}+} \), then by \cref{simpleevenlemma,d21aaaaaaaaaaa}, we also have \( \beta \in S_1 M^{\mathfrak{b}}(\lambda - \rho^{\mathfrak{b}}) \).
The remainder of the proof follows the same argument as in \cref{typeI_RBtriv}.
\end{proof}

% 参考文献
\bibliographystyle{plainnat}  % natbib 用のスタイル
\bibliography{references}

\begin{thebibliography}{41}
\providecommand{\natexlab}[1]{#1}
\providecommand{\url}[1]{\texttt{#1}}
\expandafter\ifx\csname urlstyle\endcsname\relax
  \providecommand{\doi}[1]{doi: #1}\else
  \providecommand{\doi}{doi: \begingroup \urlstyle{rm}\Url}\fi

\bibitem[Andruskiewitsch and Angiono(2017)]{andruskiewitsch2017finite}
Nicol{\'a}s Andruskiewitsch and Iv{\'a}n Angiono.
\newblock On finite dimensional nichols algebras of diagonal type.
\newblock \emph{Bulletin of Mathematical Sciences}, 7:\penalty0 353--573, 2017.

\bibitem[Asai(2020)]{asai2020semibricks}
Sota Asai.
\newblock Semibricks.
\newblock \emph{International Mathematics Research Notices}, 2020\penalty0 (16):\penalty0 4993--5054, 2020.

\bibitem[Bell and Farnsteiner(1993)]{bell1993theory}
Allen~D Bell and Rolf Farnsteiner.
\newblock On the theory of frobenius extensions and its application to lie superalgebras.
\newblock \emph{Transactions of the American Mathematical Society}, 335\penalty0 (1):\penalty0 407--424, 1993.

\bibitem[Bernstein et~al.(1976)Bernstein, Gelfand, and Gelfand]{bgg1976}
I.~N. Bernstein, I.~M. Gelfand, and S.~I. Gelfand.
\newblock A certain category of $\mathfrak{g}$-modules.
\newblock \emph{Funktsional. Anal. i Prilozhen.}, 10\penalty0 (2):\penalty0 1--8, 1976.
\newblock English transl. in Functional Anal. Appl. 10 (1976), no. 2, 87--92.

\bibitem[Bonfert and Nehme(2024)]{bonfert2024weyl}
Lukas Bonfert and Jonas Nehme.
\newblock The weyl groupoids of \( \mathfrak{sl}(m|n) \) and \( \mathfrak{osp}(r|2n) \).
\newblock \emph{Journal of Algebra}, 641:\penalty0 795--822, 2024.

\bibitem[Brundan(2014)]{brundan2014representations}
Jonathan Brundan.
\newblock Representations of the general linear lie superalgebra in the bgg category ��.
\newblock In \emph{Developments and Retrospectives in Lie Theory: Algebraic Methods}, pages 71--98. Springer, 2014.

\bibitem[Brundan and Stroppel(2012)]{stroppel2012highest}
Jonathan Brundan and Catharina Stroppel.
\newblock Highest weight categories arising from khovanov's diagram algebra iv: the general linear supergroup.
\newblock \emph{Journal of the European Mathematical Society}, 14\penalty0 (2):\penalty0 373--419, 2012.

\bibitem[Chen and Coulembier(2020)]{chen2020primitive}
Chih-Whi Chen and Kevin Coulembier.
\newblock The primitive spectrum and category for the periplectic lie superalgebra.
\newblock \emph{Canadian Journal of Mathematics}, 72\penalty0 (3):\penalty0 625--655, 2020.

\bibitem[Chen and Mazorchuk(2023)]{chen2023some}
Chih-Whi Chen and Volodymyr Mazorchuk.
\newblock Some homological properties of category for lie superalgebras.
\newblock \emph{Journal of the Australian Mathematical Society}, 114\penalty0 (1):\penalty0 50--77, 2023.

\bibitem[Chen et~al.(2021)Chen, Coulembier, and Mazorchuk]{chen2021translated}
Chih-Whi Chen, Kevin Coulembier, and Volodymyr Mazorchuk.
\newblock Translated simple modules for lie algebras and simple supermodules for lie superalgebras.
\newblock \emph{Mathematische Zeitschrift}, 297\penalty0 (1):\penalty0 255--281, 2021.

\bibitem[Cheng and Wang(2012)]{cheng2012dualities}
Shun-Jen Cheng and Weiqiang Wang.
\newblock \emph{Dualities and representations of Lie superalgebras}.
\newblock American Mathematical Soc., 2012.

\bibitem[Cheng and Wang(2019)]{cheng2019character}
Shun-Jen Cheng and Weiqiang Wang.
\newblock Character formulae in category \( \mathcal{O} \) for exceptional lie superalgebras \( d(2|1; \zeta) \).
\newblock \emph{Transformation Groups}, 24\penalty0 (3):\penalty0 781--821, 2019.

\bibitem[Cheng et~al.(2015)Cheng, Lam, and Wang]{cheng2015brundan}
Shun-Jen Cheng, Ngau Lam, and Weiqiang Wang.
\newblock The brundan--kazhdan--lusztig conjecture for general linear lie superalgebras.
\newblock 2015.

\bibitem[Coggins et~al.(2024)Coggins, Donley~Jr, Gondal, and Krishna]{coggins2024visual}
Terrance Coggins, Robert~W Donley~Jr, Ammara Gondal, and Arnav Krishna.
\newblock A visual approach to symmetric chain decompositions of finite young lattices.
\newblock \emph{arXiv preprint arXiv:2407.20008}, 2024.

\bibitem[Coulembier(2017)]{coulembier2017gorenstein}
Kevin Coulembier.
\newblock Gorenstein homological algebra for rngs and lie superalgebras.
\newblock \emph{arXiv preprint arXiv:1707.05040}, 2017.

\bibitem[Coulembier and Mazorchuk(2015)]{coulembier2015homological}
Kevin Coulembier and Volodymyr Mazorchuk.
\newblock Some homological properties of category \(\mathcal{O}\). iii.
\newblock \emph{Advances in Mathematics}, 283:\penalty0 204--231, 2015.
\newblock \doi{10.1016/j.aim.2015.06.019}.

\bibitem[Coulembier and Serganova(2017)]{coulembier2017homological}
Kevin Coulembier and Vera Serganova.
\newblock Homological invariants in category �� for the general linear superalgebra.
\newblock \emph{Transactions of the American Mathematical Society}, 369\penalty0 (11):\penalty0 7961--7997, 2017.

\bibitem[Delorme(1980)]{delorme1980extensions}
P~Delorme.
\newblock Extensions in the bernsten-gelfand-gelfand category o.
\newblock \emph{Funct Anal. Appl}, 14:\penalty0 77--78, 1980.

\bibitem[Enomoto(2021)]{enomoto2021schur}
Haruhisa Enomoto.
\newblock Schur's lemma for exact categories implies abelian.
\newblock \emph{Journal of Algebra}, 584:\penalty0 260--269, 2021.

\bibitem[Etingof et~al.(2015)Etingof, Gelaki, Nikshych, and Ostrik]{etingof2015tensor}
Pavel Etingof, Shlomo Gelaki, Dmitri Nikshych, and Victor Ostrik.
\newblock \emph{Tensor categories}, volume 205.
\newblock American Mathematical Soc., 2015.

\bibitem[Gorelik et~al.(2022{\natexlab{a}})Gorelik, Hinich, and Serganova]{gorelik2022root}
Maria Gorelik, Vladimir Hinich, and Vera Serganova.
\newblock Root groupoid and related lie superalgebras.
\newblock \emph{arXiv preprint arXiv:2209.06253}, 2022{\natexlab{a}}.

\bibitem[Gorelik et~al.(2022{\natexlab{b}})Gorelik, Hoyt, Serganova, and Sherman]{gorelik2022duflo}
Maria Gorelik, Crystal Hoyt, Vera Serganova, and Alexander Sherman.
\newblock The duflo--serganova functor, vingt ans apr{\'e}s.
\newblock \emph{Journal of the Indian Institute of Science}, 102\penalty0 (3):\penalty0 961--1000, 2022{\natexlab{b}}.

\bibitem[Heckenberger(2009)]{heckenberger2009classification}
Istv{\'a}n Heckenberger.
\newblock Classification of arithmetic root systems.
\newblock \emph{Advances in Mathematics}, 220\penalty0 (1):\penalty0 59--124, 2009.

\bibitem[Heckenberger and Schneider(2020)]{heckenberger2020hopf}
Istv{\'a}n Heckenberger and Hans-J{\"u}rgen Schneider.
\newblock \emph{Hopf algebras and root systems}, volume 247.
\newblock American Mathematical Soc., 2020.

\bibitem[Heckenberger and Yamane(2008)]{heckenberger2008generalization}
Istv{\'a}n Heckenberger and Hiroyuki Yamane.
\newblock A generalization of coxeter groups, root systems, and matsumoto’s theorem.
\newblock \emph{Mathematische Zeitschrift}, 259:\penalty0 255--276, 2008.

\bibitem[Hirota(2025)]{Hirota_PathSubgroupoids}
Shunsuke Hirota.
\newblock Rainbow boomerang graphs.
\newblock arXiv:2503.00275, 2025.

\bibitem[Humphreys(2021)]{humphreys2021representations}
James~E Humphreys.
\newblock \emph{Representations of semisimple Lie algebras in the BGG category O}, volume~94.
\newblock American Mathematical Soc., 2021.

\bibitem[Kac(1977)]{kac1977lie}
Victor~G Kac.
\newblock Lie superalgebras.
\newblock \emph{Advances in mathematics}, 26\penalty0 (1):\penalty0 8--96, 1977.

\bibitem[Ko et~al.(2024)Ko, Mazorchuk, and Mrdjen]{ko2024join}
Hankyung Ko, Volodymyr Mazorchuk, and Rafael Mrdjen.
\newblock Join operation for the bruhat order and verma modules.
\newblock \emph{Israel journal of mathematics}, pages 1--65, 2024.

\bibitem[Li et~al.(2013)Li, Shi, and Sun]{li2013rainbow}
Xueliang Li, Yongtang Shi, and Yuefang Sun.
\newblock Rainbow connections of graphs: A survey.
\newblock \emph{Graphs and combinatorics}, 29:\penalty0 1--38, 2013.

\bibitem[Liu et~al.(2019)Liu, Luo, and Wang]{liu2019odd}
Jie Liu, Li~Luo, and Weiqiang Wang.
\newblock Odd singular vector formula for general linear lie superalgebras.
\newblock \emph{arXiv preprint arXiv:1903.04683}, 2019.

\bibitem[Mazorchuk(2014)]{mazorchuk2014parabolic}
Volodymyr Mazorchuk.
\newblock Parabolic category \( \mathcal{O} \) for classical lie superalgebras.
\newblock In \emph{Advances in Lie Superalgebras}, pages 149--166. Springer, 2014.

\bibitem[Musson(2023)]{musson2023weyl}
Ian~M Musson.
\newblock The weyl groupoid in type a, young diagrams and borel subalgebras.
\newblock \emph{arXiv preprint arXiv:2312.11046}, 2023.

\bibitem[Musson(2012)]{musson2012lie}
Ian~Malcolm Musson.
\newblock \emph{Lie superalgebras and enveloping algebras}, volume 131.
\newblock American Mathematical Soc., 2012.

\bibitem[Ringel(1976)]{Ringel1976RepresentationsOK}
Claus~Michael Ringel.
\newblock Representations of k-species and bimodules.
\newblock \emph{Journal of Algebra}, 41:\penalty0 269--302, 1976.
\newblock URL \url{https://api.semanticscholar.org/CorpusID:122262365}.

\bibitem[Sale(2019)]{sale2019singular}
Thomas Sale.
\newblock Singular vector formulas for verma modules of simple lie superalgebras.
\newblock \emph{Journal of Algebra}, 521:\penalty0 365--383, 2019.

\bibitem[Serganova(2011)]{serganova2011kac}
Vera Serganova.
\newblock Kac--moody superalgebras and integrability.
\newblock \emph{Developments and trends in infinite-dimensional Lie theory}, pages 169--218, 2011.

\bibitem[Serganova(2017)]{serganova2017representations}
Vera Serganova.
\newblock Representations of lie superalgebras.
\newblock \emph{Perspectives in Lie theory}, pages 125--177, 2017.

\bibitem[Stanley et~al.(2012)]{stanley2012topics}
Richard~P Stanley et~al.
\newblock Topics in algebraic combinatorics.
\newblock \emph{Course notes for Mathematics}, 192:\penalty0 13, 2012.

\bibitem[Vay(2023)]{vay2023linkage}
Cristian Vay.
\newblock Linkage principle for small quantum groups.
\newblock \emph{arXiv preprint}, 2023.
\newblock arXiv:2310.00103.

\bibitem[Verma(1968)]{verma1968}
Daya-Nand Verma.
\newblock Structure of certain induced representations of complex semisimple lie algebras.
\newblock \emph{Bulletin of the American Mathematical Society}, 74:\penalty0 160--166, 1968.
\newblock \doi{10.1090/S0002-9904-1968-11924-4}.

\end{thebibliography}

\noindent
\textsc{Shunsuke Hirota} \\
\textsc{Department of Mathematics, Kyoto University} \\
Kitashirakawa Oiwake-cho, Sakyo-ku, 606-8502, Kyoto \\
\textit{E-mail address}: \href{mailto:hirota.shunsuke.48s@st.kyoto-u.ac.jp}{hirota.shunsuke.48s@st.kyoto-u.ac.jp}

\end{document}